\documentclass[10pt,psamsfonts]{amsart}
\usepackage{amsmath}
\usepackage{amsthm}
\usepackage{amssymb}
\usepackage{amscd}
\usepackage{amsfonts}
\usepackage{amsbsy}
\usepackage{graphicx}
\usepackage[dvips]{psfrag}
\usepackage{array}
\usepackage{color}
\usepackage{epsfig}
\usepackage{url}
\usepackage{overpic}
\usepackage{tikz-cd}
\usepackage{enumitem}
\usepackage{hyperref}
\usepackage{soul}
\usepackage{float}
\usepackage[normalem]{ulem}

\usepackage{textcomp}

\newcommand{\R}{\ensuremath{\mathbb{R}}}

\newcommand{\CO}{\ensuremath{\mathcal{O}}}

\newcommand{\T}{\theta}
\newcommand{\vf}{\varphi}

\newcommand{\bx}{{\bf x}}

\newcommand{\bz}{{\bf z}}
\newcommand{\bg}{{\bf g}}
\newcommand{\bu}{{\bf u}}
\newcommand{\f}{{\bf f}}

\newcommand{\s}{\ensuremath{\mathbb{S}}}

\newcommand{\tr}{\mathrm{Tr}}

\def\p{\partial}
\def\e{\varepsilon}

\newtheorem {theorem} {Theorem}

\newtheorem {proposition}{Proposition}

\newtheorem {lemma}{Lemma}
\newtheorem {example} {Example}
\newtheorem {remark}{Remark}

\newtheorem {mtheorem} {Theorem}

\usepackage{mathpazo}

\begin{document}
\renewcommand{\arraystretch}{1.5}

\title[Detecting normally hyperbolic invariant tori]
{A mechanism for detecting\\ normally hyperbolic invariant tori\\ in differential equations}

\author[P.C.C.R. Pereira,  D.D. Novaes, and M.R. C\^{a}ndido]
{Pedro C.C.R. Pereira, Douglas D. Novaes, and Murilo R. C\^{a}ndido}

\address{Departamento de Matem\'{a}tica - Instituto de Matemática, Estatística e Computação Científica (IMECC) - Universidade Estadual de Campinas (UNICAMP), Rua S\'{e}rgio Buarque de Holanda, 651, Cidade
Universit\'{a}ria Zeferino Vaz, 13083--859, Campinas, SP, Brazil}
\email{pedro.pereira@ime.unicamp.br}
\email{ddnovaes@unicamp.br}
\email{mr.candido@unesp.br}



\begin{abstract}
Determining the existence of compact invariant manifolds is a central quest in the qualitative theory of differential equations. Singularities, periodic solutions, and invariant tori are examples of such invariant manifolds. A classical and useful result from the averaging theory relates the existence of isolated periodic solutions of non-autonomous periodic differential equations, given in a specific standard form, with the existence of simple singularities of the so-called guiding system, which is an autonomous differential equation given in terms of the first non-vanishing higher order averaged function. In this paper, we provide an analogous result for the existence of invariant tori. Namely, we show that a non-autonomous periodic differential equation, given in the standard form, has a normally hyperbolic invariant torus in the extended phase space provided that the guiding system has a hyperbolic limit cycle. We apply this result to show the existence of normally hyperbolic invariant tori in a family of jerk differential equations.
\end{abstract}

\maketitle

\tableofcontents

\section{Introduction and statement of the main result}

Determining the existence of compact invariant manifolds is a central quest in the {\it qualitative theory of differential equations}. Equilibria are trivial examples of such sets. Increasing dimension, periodic solutions are the first non-trivial examples of compact invariant manifolds and have been extensively studied in several contexts. For differential equations written in the so-called {\it standard form} \eqref{eq:e1} below, the {\it averaging method} \cite{SVM}  provides sufficient conditions for the existence of periodic solutions. Regarding higher dimensional compact invariant manifolds, the celebrated {\it Fenichel theory} \cite{Fenichel1971} concerns their persistence under normal hyperbolicity assumptions. We can also mention the classical {\it KAM theory}  \cite{moser}, that provides conditions for the persistence of invariant tori for nearly integrable Hamiltonian systems. 

Bogoliubov and Mitropolski \cite{BM} employed the {\it averaging method} to study the existence of invariant tori for $T$-periodic non-autonomous perturbative differential equations of the kind $\dot\bx=\e F_1(t,\bx).$ Their result has also been obtained by Hale \cite{haleinvariant}. Briefly speaking, under suitable assumptions, the existence of  invariant tori is related with the existence of limit cycles of the averaged equation
\[
\dot \bx=\dfrac{1}{T}\int_0^T F_1(t,\bx)dt.
\]

In this paper, we aim to generalize the previous result by providing sufficient conditions on higher order averaged functions to ensure the  existence of invariant tori for higher order perturbed $T$-periodic non-autonomous differential equations given in the following standard form:
\begin{align} \label{eq:e1}
    \dot \bx = \sum_{i=1}^N \varepsilon^i F_i(t, \bx) + \varepsilon^{N+1} {\Tilde{F}}(t, \bx,\varepsilon),\quad (t,\bx,\e)\in \R \times D \times [0,\varepsilon_0],
\end{align}
where $D$ is an open bounded subset of $\R^2,$  $\varepsilon_0>0,$ and $\Tilde{F}:\R \times D \times [0,\varepsilon_0]\to\R^2$ and each $F_i:\R \times D \rightarrow \R^2,$ $i\in\{1,\ldots,N\}$, are $C^r$ functions, with $r\geq 2$, $T$-periodic in the variable $t$. The proof of our main result is an adaptation of the method of continuation employed by Chicone and Liu in \cite{MR1740943} (see also \cite{Kopell85}).

It is important to mention that the invariant tori of the differential equation \eqref{eq:e1}, if existent, lie in the extended phase space, where, because of the periodicity of \eqref{eq:e1}, the time variable $t\in\R$ becomes an angular variable by taking $\dot t=1$ and $t\in\mathbb{S}^1$. 

\subsection{The averaging method}\label{sec:avm} Before stating our main result, Theorem \ref{theoremA}, we shall discuss the averaging method and introduce the concept of higher order averaged functions for  differential equation \eqref{eq:e1}.  Such method provides important tools for studying nonlinear oscillating systems which are affected by small perturbations. While its origins can be traced back to names such as Clairaut, Laplace, and Lagrange  (see \cite[Appendix A]{SVM}), a rigorous formalization of this theory was achieved by Fatou, Krylov, Bogoliubov, and Mitropolsky  only later, in the 20th century (see \cite{BM,Bo,fatou,BK}). The averaging theory is mainly applied to provide asymptotic estimates for solutions of non-autonomous differential equations given in the standard form \eqref{eq:e1} by means of the solutions of the so-called \textit{truncated averaged equation}:
\begin{align}\label{truncatedintro}
    \dot \bz = \sum_{i=1}^N \varepsilon^i \bg_i(\bz).
\end{align}
It states that any solution of the differential equation \eqref{eq:e1} remains $\e^N$-close to the solution of the truncated averaged equation \eqref{truncatedintro} with the same initial condition, for time $\CO(1/\e)$ and $\e$ small (see \cite[Theorem 2.9.2]{SVM}). The functions $\bg_i:D\rightarrow\R^2,$ for $i\in\{1,\ldots,N\},$ are obtained by solving homological equations arising from near-identity transformation given by the next result.
	\begin{theorem}[{\cite[Lemma 2.9.1]{SVM}}]\label{thm:av1}
	There exists a $C^r,$ $r\geq 2$, $T$-periodic near-identity transformation
	\begin{equation}\label{avtrans}
	\bx=U(t,\bz ,\e)=\bz+\sum_{i=1}^N \e^i\, \bu_i(t,\bz ),
	\end{equation}
	 satisfying $U(0,\bz,\e)=\bz$, such that the differential equation \eqref{eq:e1} is transformed into
	\begin{equation*}\label{fullaveq}
	\dot \bz=\sum_{i=1}^N\e^i\bg_i(\bz)+\e^{N+1} r_N(t,\bz,\e).
	\end{equation*}
	\end{theorem}

One can see that $\bg_1$ is the time-average of $F_1(t,\bx)$, that is,
\begin{align*}
    \bg_1(\bz) = \frac{1}{T} \int_0^T F_1(s,\bz) \, ds.
\end{align*}
The property $U(\xi,0,\e)=\xi$ is called {\it stroboscopic condition}. We notice that Theorem \ref{thm:av1} is still true without assuming such condition. However, in this case, the functions $\bg_i,$ for $i\geq2,$ are not uniquely determined. Under the stroboscopic condition, the uniqueness of each $\bg_i$ is guaranteed, and so we call it {\it averaged function of order $i$}.

In \cite{Novaes21b}, a general recursive formula for the higher order averaged functions was obtained in terms of Melnikov functions (see \cite{LliNovTei2014,N17}).
Accordingly, define $\f_i,$ for $i\in\{0,\ldots,N\},$ by
\begin{equation}\label{avfunc}
\f_0(\bz)=0\quad\text{and}\quad \f_i(\bz)=\dfrac{y_i(T,\bz)}{i!},
\end{equation}
where
\begin{equation}\label{yi}
\begin{aligned}
y_1(t,\bz)=& \int_0^tF_1(s,\bz)\,ds\,\, \text{ and }\vspace{0.3cm}\\
y_i(t,\bz)=& \int_0^t\bigg(i!F_i(s,\bz)+\sum_{j=1}^{i-1}\sum_{m=1}^j\dfrac{i!}{j!}\p_{\bx}^m F_{i-j} (s,\bz)B_{j,m}\big(y_1,\ldots,y_{j-m+1}\big)(s,\bz)\bigg)ds,
\end{aligned}
\end{equation}
for $i\in\{2,\ldots,N\}.$ In the formulae above, for  $p$ and $q$ positive integers, $B_{p,q}$ denotes the  {\it partial Bell polynomials}:
\[
B_{p,q}(x_1,\ldots,x_{p-q+1})=\sum\dfrac{p!}{b_1!\,b_2!\cdots b_{p-q+1}!}\prod_{j=1}^{p-q+1}\left(\dfrac{x_j}{j!}\right)^{b_j},
\]
where the sum is taken  
over all the $(p-q+1)$-tuple of nonnegative integers $(b_1,b_2,\cdots,b_{p-q+1})$ satisfying $b_1+2b_2+\cdots+(p-q+1)b_{p-q+1}=p,$ and
$b_1+b_2+\cdots+b_{p-q+1}=q.$

Notice that $\f_1(\bz)=T\bg_1(\bz).$ Indeed,
\[
\f_1(\bz)=\int_0^T F_1(t,\bz)dt=T\bg_1(\bz).
\]
The next result states that, under suitable conditions, the same relationship holds for higher order averaged functions.

\begin{proposition}[{\cite[Corollary A]{Novaes21b}}]\label{prop1}
Let $\ell\in\{2,\ldots,N\}$. If either $\f_1=\cdots=\f_{\ell-1}=0$ or $\bg_1=\cdots=\bg_{\ell-1}=0,$ then $\f_i=T\,\bg_i$ for $i\in\{1,\ldots,\ell\}.$
\end{proposition}

The next proposition is as a direct consequence of Theorem \ref{thm:av1} and Proposition \ref{prop1}.
\begin{proposition}[{\cite[Corollary B]{Novaes21b}}]\label{prop2}
Let $\ell\in\{1,\ldots,N\}$ satisfy  $\f_0=\cdots\f_{\ell-1}=0.$ Then, there exists a $C^r$, $r\geq 2$, $T$-periodic near-identity transformation $\bx=U(t,\bz ,\e)$,
	 satisfying $U(0,\bz,\e)=\bz$, such that the differential equation \eqref{eq:e1} is transformed into
\[
	\dot \bz=\e^{\ell}\dfrac{1}{T}\f_{\ell}(\bz)+\e^{\ell+1} r_{\ell}(t,\bz,\e).
\]
\end{proposition}

The averaging theory has 
been applied to investigate invariant manifolds \cite{haleinvariant}. In particular, it has found great success 
for studying periodic solutions (see, for example, \cite{llibreaveraging,haleordinary,LliNovTei2014,Novaes21a,NovaesSilva,verhulst}). Indeed, let $\ell\in\{1,\ldots,N\}$ be the index of the first non-vanishing Melnikov function, that is $\f_0=\ldots=\f_{\ell-1}=0,$ and $\f_{\ell}\neq0.$  The next theorem is a very classical result that relates the existence of simple singularities of the so-called {\it guiding system}
\begin{equation}\label{guiding}
\dot \bz=\dfrac{1}{T}\f_{\ell}(\bz)
\end{equation} 
with the existence of isolated $T$-periodic solutions of the differential equation \eqref{eq:e1}.
\begin{theorem}[{\cite[Theorem A]{LliNovTei2014}}]\label{thm:per}
Consider the differential equation \eqref{eq:e1}. Suppose that, for some $\ell\in\{1,\ldots,N\},$ $\f_0=\ldots=\f_{\ell-1}=0,$  $\f_{\ell}\neq0,$   and that  the guiding system \eqref{guiding} has a simple singularity $\bz^*$, that is $\f_{\ell}(\bz^*)=0$ and $d\f_{\ell}(\bz^*)$ is non-singular. Then, for $\e>0$ sufficiently small, the differential equation \eqref{eq:e1} has an isolated $T$-periodic solution converging to $\bz^*$ as $\e$ goes to $0$.
\end{theorem}

Regarding the existence of invariant tori of the differential equation \eqref{eq:e1}, one can find results in the research literature associating it with {\it Hopf bifurcation} in the guiding system \eqref{guiding}. This fact, for $\ell=1$, is briefly commented on \cite[Appendix $C.5$]{SVM}  (see also \cite[Section $4.C$]{lhopf} and \cite[Chapter $2$]{thesis}). In \cite{cannov20}, the authors investigated how a Hopf bifurcation in the truncated average equation \eqref{truncatedintro}, particularly in the guiding system \eqref{guiding} for arbitrary $\ell$, induces a {\it Neimark-Sacker bifurcation} \cite{N,S1,S2} in the time-$T$ map of the differential equation \eqref{eq:e1} and, consequently, a {\it torus bifurcation} in the extended phase space.

\subsection{Main result}
The main result of the present study provides a version of Theorem \ref{thm:per} for detecting invariant tori, instead of periodic solutions. More specifically, it provides sufficient conditions on the guiding system \eqref{guiding} to ensure the existence of normally hyperbolic invariant tori in the extended phase space of the differential equation \eqref{eq:e1}. 
\begin{mtheorem} \label{theoremA}
Consider the differential equation \eqref{eq:e1}. Suppose that, for some $\ell\in\{1,\ldots,N\},$ $\f_0=\ldots=\f_{\ell-1}=0,$  $\f_{\ell}\neq0,$   and that  the guiding system \eqref{guiding}  has an attracting hyperbolic limit cycle $\gamma$. Then, for each $\e>0$ sufficiently small, the differential equation \eqref{eq:e1}  has a $T$-periodic solution $\gamma_{int}$ and a normally hyperbolic attracting invariant torus in the extended phase space. In addition, the torus surrounds the periodic solution $\gamma_{int}$ and converges to $\gamma\times\s^1$ as $\e$ goes to $0$.
\end{mtheorem}

The first step in order to prove Theorem \ref{theoremA} consists in obtaining a differential system equivalent to the differential equation \eqref{eq:e1} for which an adaptation of the method of continuation, employed in \cite{MR1740943}, can be applied. Accordingly, Section \ref{sec:ts} combines near-identity transformations, provided by the averaging theory, with a moving orthonormal coordinate system along periodic solutions \cite{haleordinary} to obtain such equivalent differential system. Theorem \ref{theoremA} is then restated as Theorem \ref{theoremB}. Section \ref{sec:mcproof} is devoted to the proofs of Theorems \ref{theoremA} and \ref{theoremB}. 

\subsection{Application to jerk differential equations} 
In mechanics, and other areas of physics as well, if $x(t)$ denotes the {\it position} of a particle, then  $\dot{x}(t)$ and $\ddot{x}(t)$ are the {\it velocity} and {\it acceleration} of the particle, respectively. In this context, the quantity $\dddot{x}(t)$ is called {\it jerk}   of the particle and a third-order differential equation of the form $\dddot{x}=F\left(\ddot{x},\dot{x},x\right)$ is called \textit{jerk differential equation} (see, for instance, \cite{schot1978jerk}).  Many mechanical models and a wide variety of classical three-dimensional systems can be represented in this form. For instance, the Nos\'e-Hoover oscillator \cite{sprott2010}, the Lorenz system, and the R\"ossler system \cite{linz1997}. 

Here, in light of Theorem \ref{theoremA}, we shall provide in Proposition \ref{exprop} sufficient conditions for the existence of an attracting normally hyperbolic invariant torus living in the 3D space $(x,\dot x,\ddot x)$ for the following family of jerk differential equations
\begin{equation}\label{example}
\dddot x=-\dot x+\e^{N-1} P(x,\dot x,\ddot x)+\e^N Q(x,\dot x,\ddot x)+\e^{N+1} R(x,\dot x,\ddot x,\e),
\end{equation}
where  $\e$ is a small positive parameter and $P$, $Q,$ and $R$ are $C^r$, $r\geq 2$, functions. In addition, the following conditions will be assumed on $P$ and $Q$: 
\begin{itemize}
\item[{\bf H1.}] the following functions have vanishing average for every $z$ and $r$,
\[
\T\mapsto P(r\sin\T-z,r\cos\T,-r\sin\T)\,\,\text{ and }\,\, \T\mapsto P(r\sin\T-z,r\cos\T,-r\sin\T) \sin(\T)
\]
\item[{\bf H2.}] and
\[
\begin{aligned}
 Q(x,\dot x,\ddot x)=&\ddot x\big(-x-\ddot x+(\dot x^2+\ddot x^2-2)\big(1-(x+\ddot x)^2-\big(\dot x^2+\ddot x^2-2\big)^2\big)\big)\\&+2\ddot x^2(\dot x^2+\ddot x^2-2).
\end{aligned}
\]
\end{itemize}

\begin{proposition}\label{exprop}
Assume that {\bf H1} and {\bf H2} hold. Then,
for any integer $N\geq 3$ and $\e>0$ sufficiently small, the jerk differential equation \eqref{example} has a normally hyperbolic attracting invariant torus in the space $(x,\dot x,\ddot x)$ which converges, as $\e$ goes to $0$, to the limiting torus (see Figure \ref{unpertorus})
\[
\mathbb{T}=\{(\sqrt{r}\sin\T-z,\sqrt{r}\cos\T,-\sqrt{r}\sin\T):\, (r,z)\in \mathbb{S}^1+(2,0),\, \T\in[0,2\pi]\}.
\]
\end{proposition}

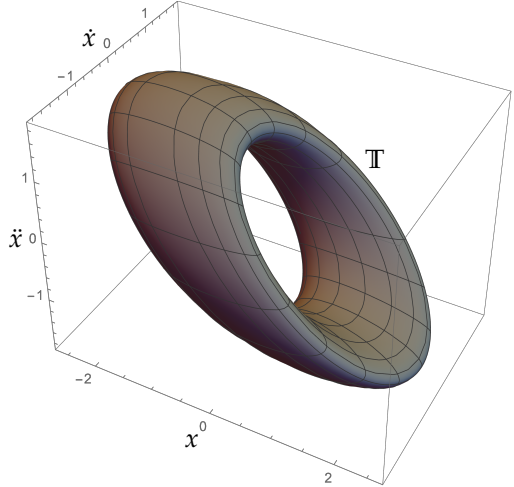
\begin{figure}[H]
	\begin{overpic}[width=6.5cm]{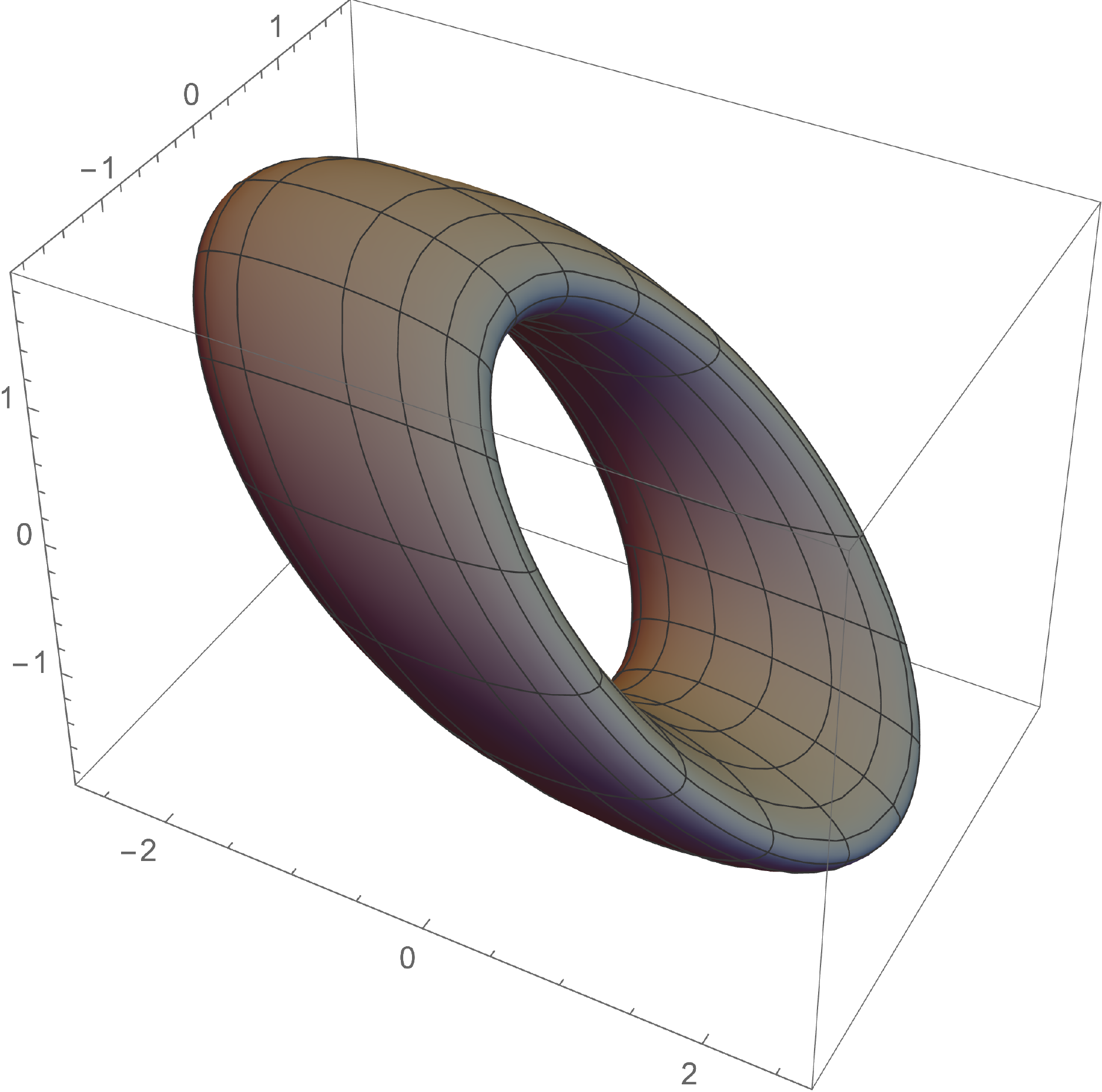}
	\put(33,8){$x$}
	\put(12,90){$\dot x$}
	\put(-3,49){$\ddot x$}
	\put(70,65){$\mathbb{T}$}
	\end{overpic}
	\caption{Limiting Torus $\mathbb{T}$ of the jerk differential equation \eqref{example}.}
	\label{unpertorus}	
\end{figure}

Proposition \ref{exprop} is proven in Section \ref{sec:app}.

\begin{example}
One can see that condition {\bf H1} holds for $P(x,\dot x,\ddot x)=-\dot x ^3$. Thus, Proposition \ref{exprop} can be applied to get the existence of a normally hyperbolic attracting invariant torus of the jerk differential equation \eqref{example}  in the space $(x,\dot x,\ddot x)\in\R^3$ for each $N\geq 2$ and $\e>0$. In Figure \ref{torusfig}, assuming $R=0$, $N=5$, and $\e=1/5$, we have used the software Mathematica to simulate a solution of the jerk differential equation approaching to the invariant torus.

\begin{figure}[H]
	\begin{overpic}[width=13cm]{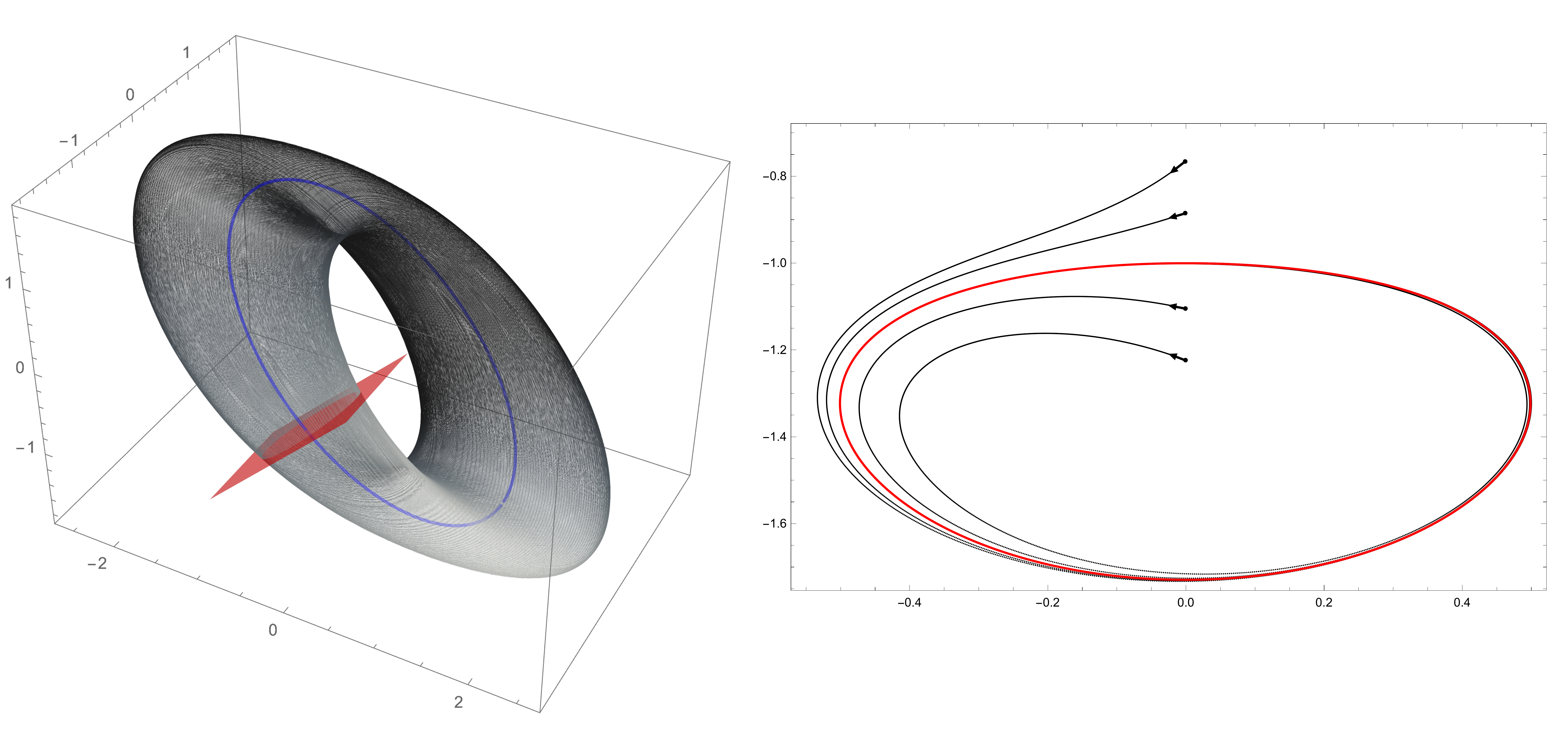}
	\end{overpic}
	\caption{Invariant torus of the differential equation \eqref{example} in the $(x,\dot x,\ddot x)$-space predicted by Proposition \ref{exprop}, assuming $R=0$, $N=5$, and $\e=1/5$. The left figure shows a particular trajectory of \eqref{example} evolving forward in time. The thick line corresponds to the periodic solution. The right figure depicts several trajectories of the Poincar\'{e} map defined in the section $\{x=z, y<0\}$. The stable invariant torus corresponds to a stable invariant closed curve.}
	\label{torusfig}	
\end{figure}

\end{example}

\section{Transformed differential system}\label{sec:ts}
In this section, by combining the near-identity transformation, provided by Proposition \ref{truncatedintro}, with a moving orthonormal coordinate system along periodic solutions \cite[Chapter VI]{haleordinary}, the differential equation \eqref{eq:e1}, under the hypotheses of Theorem \ref{theoremA}, will be transformed into an equivalent differential system for which an adaptation of the techniques employed in \cite{MR1740943} can be applied.

Consider the differential equation \eqref{eq:e1}. Let $\ell\in\{1,\ldots,N\}$ satisfying $\f_0=\f_1=\cdots\f_{\ell-1}=0$ and $\f_{\ell}\neq0$. By Proposition \eqref{prop2}, there exists a $T$-periodic near-identity transformation \eqref{avtrans} that transforms the differential equation \eqref{eq:e1} into
\begin{equation}\label{eq:trans1}
\dot \bz=\e^{\ell}\dfrac{1}{T}\f_{\ell}(\bz)+\e^{\ell+1} r_{\ell}(t,\bz,\e).
\end{equation}
By applying the time rescaling $\tau=\varepsilon^\ell \, t$, the differential equation \eqref{eq:trans1} becomes:
\begin{align} \label{systemrescaled}
    \bz'=\dfrac{1}{T}\f_\ell(\bz) + \varepsilon \, r_{\ell}(\tau/\e^{\ell},\bz,\e),
\end{align}
where, now, prime $'$ denotes derivative with respect to $\tau$.

Now, let $\varphi(\tau)$ be the hyperbolic limit cycle of the guiding system
$$
    \bz'=\dfrac{1}{T}\f_\ell(\bz).
$$
Without loss of generality, assume that $\varphi$ is $2\pi$-periodic.
In what follows, we shall transform the differential equation \eqref{systemrescaled} by using a moving orthonormal coordinate system along $\varphi(\tau)$  provided in \cite[Chapter VI]{haleordinary}.  Let $\big(v_1(\tau),v_2(\tau)\big)={\varphi'}(\tau)/{|{\varphi'}(\tau)}|$. From Theorem \cite[Theorem 1.2]{haleordinary},  there exists $r_0>0$ such that, for all $0\leq |\rho|\leq r_0$, the transformation
\begin{equation}\label{eq:moc}
\bz=\varphi(\sigma)+Z(\sigma)\rho,
\end{equation}
with $Z(\sigma)=({-v_2(\sigma),v_1(\sigma)})$, applied to the differential equation \eqref{systemrescaled} yields the equivalent differential system:
\begin{equation}\label{eq:trans2}
\begin{aligned}
\rho'&=A(\sigma)\rho+f_2(\rho, \sigma)+\e B(\rho, \sigma)r_{\ell}(\tau/\e^{\ell},\varphi(\sigma)+Z(\sigma)\rho,\e),\\
\sigma'&=1+f_1(\rho, \sigma)+\e \big\langle h(\rho, \sigma)\, ,\,r_{\ell}(\tau/\e^{\ell},\varphi(\sigma)+Z(\sigma)\rho,\e)\big\rangle.
\end{aligned}
\end{equation}
Here, $A(\sigma)\in\R$, $B(\rho, \sigma)$ is a $1\times 2$ matrix, $f_2(\rho, \sigma),\,h(\rho, \sigma)\in\R^2$, and  $f_1(\rho, \sigma)\in \R$, all of them $C^r$ functions, with $r\geq 2$, and $2\pi$-periodic in the variable $\sigma$. Moreover,
\begin{equation*}\label{eqbar}
| f_1(\rho, \sigma)|=\CO(|\rho|), \quad f_2(0, \sigma)=0, \quad \dfrac{\partial  f_2(0,\sigma )}{\partial \rho}=0.
\end{equation*}
Now, denoting
\begin{align*}
& f(\rho, \sigma)=A(\sigma)\rho+f_2(\rho, \sigma),\\
& g(\rho,\sigma)=1+f_1(\rho, \sigma),\\
& R(\rho, \sigma, t,\e)= B(\rho, \sigma)r_{\ell}(t,\varphi(\sigma)+Z(\sigma)\rho,\e),\,\,\text{and}\\
& S(\rho,\sigma,t,\e)=\big\langle h(\rho, \sigma)\, ,\,r_{\ell}(t,\varphi(\sigma)+Z(\sigma)\rho,\e)\big\rangle,
\end{align*}
the differential system \eqref{eq:trans2} writes
\begin{align} \label{systempolar}
    &\rho' = f(\rho,\sigma) + \varepsilon R(\rho,\sigma,\tau/\varepsilon^\ell,\varepsilon), \nonumber \\ 
    &\sigma' = g(\rho,\sigma) + \varepsilon S(\rho,\sigma, \tau/\varepsilon^\ell, \varepsilon),
\end{align}
where $f$, $g$, $R$, and $S$ are $C^r$ functions, with $r\geq 2$, $2\pi$-periodic in the variable $\sigma$, and $T$-periodic in the variable $\tau$.

Notice that the guiding system \eqref{guiding} is equivalent to the differential system
\begin{align} \label{systempolartruncated}
    &\rho' = f(\rho,\sigma), \nonumber\\
    &\sigma' = g(\rho,\sigma),
\end{align}
for which, due to the moving orthonormal coordinate transformation \eqref{eq:moc}, the attracting hyperbolic limit cycle of \eqref{guiding} is, now, given by $\Gamma=\{(0,\sigma):\sigma\in\s^1\}$.

In this new setting, Theorem \ref{theoremA} follows as a consequence of the next result:
\begin{mtheorem} \label{theoremB}
Suppose that $\Gamma=\{(0,\sigma):\sigma\in\s^1\}$ is an attracting hyperbolic limit cycle of the differential system \eqref{systempolar}. Then, for $\varepsilon>0$ sufficiently small, the differential system \eqref{systempolar} has a $T$-periodic solution $\Gamma_{int}$ and a $r$-normally hyperbolic attracting invariant manifold in the extended phase space that is the graph of a $C^r$ function of the angular variables $\sigma$ and $\tau$. In addition, this manifold surrounds the periodic solution $\Gamma_{int}$ and converges to $\Gamma\times\s^1$ as $\e$ goes to $0$.
\end{mtheorem}

\section{Method of continuation and proof of the main results}\label{sec:mcproof}

This section is devoted to the proof of Theorem \ref{theoremB}. We start by presenting the Method of Continuation in Section \ref{sec:mc}. In Section \ref{sec:loc}, the justification of the convergence of the perturbed tori as $\varepsilon$ goes to zero is provided. Theorem \ref{theoremB} is then proven in Section \ref{sec:proofB} under the assumption of validity of a key result, whose proof is presented in Section \ref{sec:proofmain}. Some essential lemmas, which are used throughout the proof in Section \ref{sec:proofmain}, are presented in Section \ref{sec:prel}. Finally, an Appendix is provided, containing technical lemmas that, albeit not directly related to the problem, are crucial for the argument.

\subsection{Method of continuation}\label{sec:mc}
We define $E^{\varepsilon,\mu}$, an auxiliary family of differential systems given by
\begin{equation}\label{eq:family}
	E^{\varepsilon,\mu}:
	\begin{cases}
		\rho'=f(\rho,\sigma) + \varepsilon R(\rho,\sigma,\tau/\mu^\ell,\mu),  \\
		\sigma'= g(\rho, \sigma) + \varepsilon S(\rho, \sigma, \tau/\mu^\ell, \mu),  \\
		\tau'=1,
	\end{cases}
\end{equation}
where $f$, $g$, $R$, and $S$ are the functions appearing \eqref{systempolar}, thus of class $C^r$, $r\geq2$. Also, periodicity allows us to view $\tau$ as an angular variable modulo $T \mu^\ell$.


In order to prove Theorem \ref{theoremB}, define, for each $\mu>0$, the set $A^\mu$ as the maximal interval contained in $[0,\mu]$ with left endpoint at $0$ such that, for $\varepsilon \in A^\mu$, system $E^{\varepsilon,\mu}$ has an invariant manifold $M(\varepsilon,\mu)$ satisfying the following properties:
\begin{itemize}
\item $M(\varepsilon,\mu)$ is $r$-normally hyperbolic invariant manifold;
\item $M(\varepsilon,\mu)$ is the graph $\rho=h^{\varepsilon}(\sigma,\tau)$ of a $C^r$ function $h^\varepsilon$ of the angular variables $\sigma$ and $\tau$.
\end{itemize}
Notice that $M(\varepsilon,\mu)$ is an invariant torus.

The method of continuation, employed in \cite{MR1740943}, uses the connectedness of $[0,\mu]$ to ensure that $A^\mu = [0,\mu]$ by showing that $A^\mu$ is a non-empty subset of $[0,\mu]$ which is also relatively open and closed.

First, note that $A^\mu$ is non-empty, because $0 \in A^\mu$ by hypothesis. The relative openness of $A^\mu$ is guaranteed by Fenichel's classic result about persistence of $r$-normally hyperbolic invariant manifolds:

\smallskip

\noindent{\bf Fenichel's Theorem} {\it
	Let $f$ be a $C^r$ vector field on $\mathbb{R}^n$, $r\geq1$. Let $M$ be a compact, connected $C^r$ manifold properly embedded in $\mathbb{R}^n$ and invariant under $f$. Suppose that $M$ is $r$-normally hyperbolic. Then, for any vector field $g$ in some $C^r$ neighbourhood of $f$, there is a $r$-normally hyperbolic $C^r$ manifold $M_g$, invariant under $g$, $C^r$ diffeomorphic to $M$ and $C^r$ near $M$.
}

\smallskip

\noindent Indeed, if $\varepsilon \in A^\mu$, Fenichel's Theorem ensures that there is an open interval $I_\varepsilon$ containing $\varepsilon$ such that $I_\varepsilon \cap [0,\mu] \subset A^\mu$. 

Finally, it only remains to show that $A^\mu$ is relatively closed. This will be achieved by proving the following result:

\begin{proposition}\label{prop:main}
	Suppose that the differential system (\ref{systempolartruncated}) has an attracting hyperbolic limit cycle $\Gamma=\{(0,\sigma):\sigma\in\s^1\}$. Let $\bar{\varepsilon}_{\mu}$ be the least upper bound of $A^\mu$. Then, there is $\mu_s>0$ such that, if $\mu \in (0,\mu_s]$, system $E^{\bar{\varepsilon}_{\mu},\mu}$ has an invariant $r$-normally hyperbolic manifold $M(\bar{\varepsilon}_{\mu},\mu)$ given as graph of a $C^r$ function of the angular variables. In particular, $\bar{\varepsilon}_{\mu}\in A^{\mu}.$
	\end{proposition}
Henceforth, as in the Proposition, $\bar{\varepsilon}_{\mu}$ will always denote the least upper bound of $A^\mu$.

We remark that system $E^{\varepsilon,\mu}$ is strikingly similar to the one studied by Chicone and Liu in \cite{MR1740943}, the only difference being that $\ell=2$ in that work. Thus, the proof of Proposition \ref{prop:main} will be a revision of the proof of an analogous result found in \cite{MR1740943}, namely Proposition 7.2 in that work. Many of the ideas presented therein will naturally be of use in our work, and will be reproduced here in order to provide the adaptations required by the greater freedom allowed in choosing the value of $\ell$. Preliminary results akin to those presented in \cite{MR1740943} are discussed in Section \ref{sec:prel}. The proof of Proposition \ref{prop:main} itself is found in Section \ref{sec:proofmain}.

\subsection{Localizing the invariant torus} \label{sec:loc}
The Method of Continuation outlined above guarantees the existence of invariant tori in the extended phase space of system (\ref{systempolar}). However, in order to ensure the convergence of the invariant tori as $\varepsilon \to 0$, we need a separate argument. First, we present the following lemma, which will be of use later on:
\begin{lemma} \label{lemmaapriori}
	Consider a planar differential equation $x'=f(x)$ with a hyperbolic limit cycle $\Gamma$ of period $\omega>0$. If $\Gamma$ is asymptotically stable, then there is a neighborhood $N \subset \mathbb{R}^2$ of $\Gamma$ and a constant $C>0$ such that any invariant set $\overline{M} \subset N \times \mathbb{S}^1$ of the system
\[
		x'=f(x) + g(x,\tau), \quad \tau'=1, \quad (x,\tau)\in N \times \mathbb{S}^1
\]
	satisfies the following estimate
\[
		\sup \{ d(y, \Gamma \times \mathbb{S}^1): y \in \overline{M} \} \leq C \|g\|_{C^0},
\]
	where $\mathbb{S}^1=\mathbb{R}/(T\mathbb{Z})$, $T\geq0$, and $g:N \times \mathbb{S}^1 \to \mathbb{R}^2$ is any smooth function. \end{lemma}

The proof of Lemma \ref{lemmaapriori} may be found in \cite{MR1740943}, wherein it is called Proposition 6.2. We remark that, even though it is stated slightly differently in that reference, the proof set forth in it is directly applicable to our case.

Now, assuming that Proposition \ref{prop:main} holds, we provide the following result that will be fundamental in establishing the convergence of the invariant tori as $\varepsilon$ goes to $0$.
\begin{proposition} \label{propositionconvergence}
	Suppose system (\ref{systempolartruncated}) has an attracting hyperbolic limit cycle $\Gamma=\{(0,\sigma):\sigma\in\s^1\}$. Then, there exist $\mu_L>0$ and a constant $K>0$ such that the inequality
\[
		|h^\varepsilon|_{C^0} < K \varepsilon
\]
	holds for all $\mu \in (0,\mu_L]$ and all $\varepsilon \in [0,\mu]$
\end{proposition}
\begin{proof}
 Let $\mu_s>0$ be as in Proposition \ref{prop:main}.  Since system  (\ref{systempolartruncated}) has a attracting hyperbolic limit cycle $\Gamma$, we can apply Lemma \ref{lemmaapriori} and obtain $N \subset \mathbb{R}^2$ and $C>0$ satisfying the conditions put forward in it. Moreover, if $R$ and $S$ are defined as in (\ref{eq:family}), we set
\[
		K_s := \sup \{|R(\rho,\sigma,\tau/\mu^\ell,\mu)|+|S(\rho,\sigma,\tau/\mu^\ell,\mu)|: (\rho,\sigma,\tau) \in N \times \mathbb{R}, \; \mu \in (0,\mu_s] \}.
\]
 Notice that $K_s<\infty$. Indeed, the functions $R$ and $S$ are continuous and periodic in the third variable.

	Let $r_0$ be defined by
\[
		r_0:= \inf \{d(y,\Gamma): y \in \mathbb{R}^2\setminus N \}>0.
\]
	Define
\[
		\mu_1:= \frac{1}{2} \min \left\{\mu_s, \frac{r_0}{C \, K_s}\right\}.
\]
	Fix $\mu \in (0,\mu_1]$. Since $\mu \in (0,\mu_s]$, Proposition \ref{prop:main} ensures that $[0,\mu]=A^\mu$. Define $I^\mu$ as the maximal interval contained in $A^\mu=[0,\mu]$ with left endpoint at 0 such that, for $\varepsilon \in I^\mu$, the following holds: $h^\varepsilon$ satisfies $|h^\varepsilon|_{C^0}<r_0$. The result will be proved if we show that $I^\mu=[0,\mu]$, because, in this case, Lemma \ref{lemmaapriori} implies that $|h^{\e}|_{C^0}<K \e$, with $K=C K_s$. This will be done by showing that $I^\mu$ is non-empty subset of the interval $[0,\mu]$ that is also relatively open and closed.
	
	It is clear that $I^\mu$ is non-empty, because $0 \in I^\mu$.  Furthermore, since $I^\mu \subset A^\mu$, we can apply Fenichel's Theorem to any $M(\varepsilon,\mu)$ such that $\varepsilon \in I^\mu$ and guarantee that $I^\mu$ is relatively open in $A^\mu$. 
	
	All that remains is showing that $I^\mu$ is also relatively closed in $A^\mu$, so that connectedness implies that $I^\mu=A^\mu=[0,\mu]$. In order to do so, let $\alpha_{\mu}$ be the least upper bound of $I^\mu$. By definition of $I^\mu$, we must have $[0,\alpha_{\mu}) \subset I^\mu \subset [0,\mu] = A^\mu$. Thus, by taking closures, it follows that the stronger inclusion $[0,\alpha_{\mu}] \subset A^\mu$ holds. Thus, Fenichel's theorem ensures that $h^\varepsilon$ is continuous with respect to its parameter $\varepsilon$ for $0\leq \varepsilon \leq \alpha_{\mu}$.
	
	Observe that, by definition of $r_0$, the following holds: if $|h^\varepsilon|_{C^0}<r_0$, then $M(\varepsilon,\mu) \subset N \times \mathbb{S}^1$. Since $[0,\alpha_{\mu}) \subset I^\mu$, Lemma \ref{lemmaapriori} guarantees that $|h^\varepsilon|_{C^0}<C K_s \varepsilon$ for all $\varepsilon \in [0,\alpha_{\mu})$. Thus, we have
\[
		|h^\varepsilon|_{C^0} < C K_s \varepsilon < C K_s \mu< C K_s \mu_1 < \frac{r_0}{2},
\]
	for all $\varepsilon \in [0,\alpha_{\mu})$. From the continuity of the function $\e\mapsto h^\varepsilon$, for $\varepsilon \in [0,\alpha_{\mu}]$, by taking the limit as $\varepsilon \to \alpha_{\mu}$, it follows  that
	\begin{align*}
		|h^{\alpha_{\mu}}|_{C^0}\ \leq \frac{r_0}{2} < r_0,
	\end{align*}
	so that $\alpha_{\mu} \in I^\mu$. This proves that $I^\mu$ must also be relatively closed in $[0,\mu]$. Therefore, by connectedness of $[0,\mu]$, it follows that $I^\mu=[0,\mu]$, finishing the proof.
\end{proof}
\subsection{Proof of Theorems \ref{theoremA} and \ref{theoremB}}\label{sec:proofB}

Consider the differential system \eqref{systempolar} written in the extended phase space:
\begin{equation} \label{autonomized system}
	\begin{aligned}
		\rho' =& f(\rho,\sigma) + \varepsilon R(\rho,\sigma,\tau/\varepsilon^\ell,\varepsilon),\\
		\sigma' = &g(\rho,\sigma) + \varepsilon S(\rho,\sigma, \tau/\varepsilon^\ell, \varepsilon), \\
		\tau'=&1,
	\end{aligned}
\end{equation}
where $\tau$ is an angular variable modulo $T \varepsilon^\ell$.

Notice that, by taking $\varepsilon=\mu$, the differential equation \eqref{autonomized system} corresponds to the member $E^{\varepsilon,\varepsilon}=E^{\mu,\mu}$ of the family \eqref{eq:family}. Assuming Proposition \ref{prop:main} is proven, it follows from the Method of Continuation introduced before that $A^\mu = [0,\mu]$. In particular, $\mu\in A^{\mu}$. Thus, the part of Theorem \ref{theoremB} regarding the existence of the $r$-normally hyperbolic invariant tori follows by taking $\varepsilon=\mu$.

The existence of the $T$-periodic solution $\Gamma_{int}$ inside the torus $M(\mu,\mu)$ follows directly from observing that, since $\tau'=1$, we can define a stroboscopic Poincaré map on a transversal section inside the torus that extends to its border, including this border. This map is continuous and its codomain can be chosen to be equal to its domain, a set that is homeomorphic to the closed disk. Thus, by the Brouwer fixed-point theorem, the Poincaré map thus defined must have a fixed point, which corresponds to the $T$-periodic solution $\Gamma_{int}$.

The convergence of the invariant manifold as $\varepsilon$ goes to $0$ in (\ref{autonomized system}) follows from Proposition \ref{propositionconvergence}. In fact, for $\mu$ sufficiently small, we can take $\varepsilon=\mu$  in Proposition \ref{propositionconvergence} and conclude that there is $K>0$, independent of the choice of $\mu$, such that $|h^\mu|_{C^0} < K \mu$. It follows that the invariant manifold $M(\mu,\mu)$ of system $E^{\mu,\mu}$ converges to $\Gamma \times \mathbb{S}^1$ as $\mu \to 0$.

Finally, Theorem \ref{theoremA} follows by noticing that the guiding system \eqref{guiding} is equivalent to the differential system \eqref{systempolartruncated}, and that the differential equation \eqref{eq:e1} is equivalent to the differential system \eqref{autonomized system}.

\subsection{Useful Estimates on $h^\varepsilon$}\label{sec:prel}

We remind that, for $\mu>0$ fixed, $h^{\varepsilon}$ denotes a $C^r$ function of the angular variables whose graph is a $r$-normally hyperbolic invariant manifold of system $E^{\varepsilon,\mu}$. Moreover, we define $\gamma^\varepsilon(s,q):=(h^\varepsilon(\sigma^\varepsilon(s,q),\tau(s)),\sigma^\varepsilon(s,q),\tau(s))$ to be the solution of system $E^{\varepsilon,\mu}$ starting at $(h^\varepsilon(q,0),q,0)$. As indicated before, $\bar{\varepsilon}_{\mu}$ denotes the least upper bound of $A^\mu$ for a given $\mu>0$.

Estimates on the size of the function $h^\varepsilon$, its derivatives and other useful quantities will be necessary. 
 More specifically, we shall prove that there are $\mu_0>0$ and $C_0>0$ such that:
 \begin{itemize}
 \item[(a)] $|h^\varepsilon|_{C^1}<C_0\varepsilon$ for all $\mu \in(0,\mu_0]$ and all $\varepsilon \in [0,\bar{\varepsilon}_\mu)$;
\item[(b)] for each $\mu \in (0,\mu_0]$, the family $\{h^\varepsilon: \varepsilon\in [0,\bar{\varepsilon}_\mu) \}$ is uniformly bounded in the $C_2$-norm.
\end{itemize}
Statement (a) will follow as a consequence of Lemma \ref{lemmaestimatesh} and Statement (b) will follow as a consequence of Lemmas \ref{lemmahsigmasigma} and \ref{lemmahsigmatauhtautau}.

Before presenting the proofs of those estimates, we outline the meaning of the terms used therein. Considering system $E^{\varepsilon,\mu}$, we define
\begin{align} \label{definition:G}
	F(\rho,\sigma,\tau,\mu,\varepsilon):=f(\rho,\sigma) + \varepsilon R (\rho,\sigma,\tau/\mu^\ell,\mu), \nonumber \\ 
	G(\rho,\sigma,\tau,\mu,\varepsilon):=g(\rho,\sigma) + \varepsilon S (\rho,\sigma,\tau/\mu^\ell,\mu).
\end{align} 
We also define
\[
	y^\varepsilon(s,q) := \exp \left(\int_0^s (G_\rho h^\varepsilon_\sigma + G_\sigma) \, dt \right) ,
\]
where the argument of the derivatives of $G$ is  $(h^\varepsilon(\sigma^\varepsilon(t,q),\tau(t)),\sigma^\varepsilon(t,q),\tau(t),\mu,\varepsilon)$ and the argument of the derivative of $h^\varepsilon$ is given by $(\sigma^\varepsilon(t,q),\tau(t))$. That integral arises naturally when solving the first variational equation associated to $E^{\varepsilon,\mu}$.

As indicated earlier, the first part of the estimates we will need later on is provided by Lemmas \ref{lemmareduction} and \ref{lemmaestimatesh} below, whose proofs are strongly inspired by the proofs of their counterparts in \cite{MR1740943}. We remark however that, because of the intended application we have for those results, we ensure the independence of the constants appearing in our proofs with respect to the particular choice of $\mu>0$.

\begin{lemma}\label{lemmareduction}
	Suppose there are $\mu_0>0$ and $C_0>0$ such that the following estimates hold
	\begin{align*}
		|h^{\varepsilon}(\sigma,\tau)| < C_0 \varepsilon, \qquad |h_\sigma^\varepsilon(\sigma,0) | < C_0 \varepsilon,
	\end{align*}
	for all $\mu \in (0,\mu_0]$, all $\varepsilon \in [0,\bar{\varepsilon}_{\mu})$ and all $(\sigma,\tau) \in [0,2\pi] \times [0,T \,\mu^\ell]$.
	Then, there are  $\widetilde{\mu} \in (0,\mu_0]$ and $\widetilde{C}>0$, independent of $\mu$, such that the estimate 
	\begin{align*}
		|h_\sigma^\varepsilon |_{C^0} < \widetilde{C} \varepsilon
	\end{align*}
	holds for all $\mu \in (0,\widetilde{\mu}]$ and all $\varepsilon \in [0,\bar{\varepsilon}_{\mu})$.
\end{lemma}
\begin{proof}
	Let $\mu_0$ and $C_0$ be as stated above. Also, we will denote the flow of system $E^{\varepsilon,\mu}$ by $\phi^\varepsilon(t,x)$, where $\phi^\varepsilon(0,x)=x$. For each $\mu \in (0,\mu_0]$, $\varepsilon \in [0,\bar{\varepsilon}_{\mu})$, $p \in [0,2\pi]$, and $\tau \in \mathbb{R}$, we may define a unique angle $q^\varepsilon$ by
	\begin{align*}
		(h^\varepsilon(q^\varepsilon,0),q^\varepsilon,0) = \phi^\varepsilon(-\tau,(h^\varepsilon(p,\tau),p,\tau)).
	\end{align*}
	Let $\mu_L>0$ and $K>0$ be as in Proposition \ref{propositionconvergence}. Define $\mu_1:=\min \{\mu_0,\mu_L\}>0$. An application of Lemma \ref{gronwalllemma} with $t_*=-T \,\mu_1^\ell$ and $Q=\{(h^\varepsilon(r,s),r,s,\varepsilon,\mu) \in \mathbb{R} \times [0,2\pi] \times \mathbb{R} \times \mathbb{R} \times (0,\mu_1]: s \in [0,T \,\mu^\ell], \; \varepsilon \in [0,\bar{\varepsilon}_{\mu}) \}$ ensures that there is $K_1>0$, independent of $\mu$ and $\varepsilon$, such that
	\begin{align*}
		|h^\varepsilon(q^\varepsilon,0)| + |q^\varepsilon-q^0| < K_1 (|h^\varepsilon(p,\tau)|+ \varepsilon),
	\end{align*}
	for all $\mu \in (0,\mu_1]$, $\varepsilon \in [0,\bar{\varepsilon}_{\mu})$, $p \in [0,2\pi]$, and all $\tau \in [0,T\mu^\ell]$.
	By hypothesis, it follows that,
\[
		|h^\varepsilon(q^\varepsilon,0)| + |q^\varepsilon-q^0| < K_1 (C_0+1)\varepsilon,
\]
	for all $\mu \in (0,\mu_1]$, $\varepsilon \in [0,\bar{\varepsilon}_{\mu})$, $p \in [0,2\pi]$, and all $\tau \in [0,T\mu^\ell]$.
	In particular, by defining the constant $K_2:=K_1(C_0+1)$, we have that
\[
		|q^\varepsilon - q^0| < K_2 \varepsilon.
\]
	
	A solution of the first variational equation of $E^{\varepsilon,\mu}$ along $\gamma^\varepsilon(t,q^\varepsilon)$ is given by
	\begin{align} \label{def:v2epsilon}
		v_2^\varepsilon(s,q^\varepsilon) := y^\varepsilon(s,q^\varepsilon) 
		\begin{pmatrix}
			h_\sigma^\varepsilon(\sigma^\varepsilon(s,q^\varepsilon),\tau(s)) \\
			1 \\
			0
		\end{pmatrix}.
	\end{align}
	In fact, $v_2^\varepsilon(s,q^\varepsilon) = \gamma_q^\varepsilon(s,q^\varepsilon)$, because $y^\varepsilon(s,q^\varepsilon)$ satisfies the same initial value problem as $\sigma_q^\varepsilon(s,q^\varepsilon)$. Let $\Psi^\varepsilon(s,q^\varepsilon)$ be the principal fundamental matrix solution of $E^{\varepsilon,\mu}$ at $t=0$ along $\gamma^\varepsilon(s,q^\varepsilon)$. Then, it follows that
	\begin{align*}
		v_2^\varepsilon(s,q^\varepsilon) = \Psi^\varepsilon(s,q^\varepsilon) \cdot v_2^\varepsilon(0,q^\varepsilon) =  \Psi^\varepsilon(s,q^\varepsilon) \cdot 
		\begin{pmatrix}
			h_\sigma^\varepsilon(q^\varepsilon,0) \\
			1 \\
			0
		\end{pmatrix}.
	\end{align*}
	Furthermore, by definition of $q^\varepsilon$, we have 
	\begin{align*}
		v_2(\tau,q^\varepsilon) = y^\varepsilon(\tau, q^\varepsilon) 
		\begin{pmatrix}
			h_\sigma^\varepsilon(p,\tau) \\
			1 \\
			0
		\end{pmatrix}.
	\end{align*}
	We remark that, by the triangle inequality,
\[
\begin{aligned}
|\Psi^\varepsilon(s,q^\varepsilon) \cdot v_2^\varepsilon(0,q^\varepsilon) - \Psi^0(s,q^0) \cdot v_2^0(0,q^0)| \leq |\Psi^\varepsilon(s,q^\varepsilon) -& \Psi^0(s,q^0)| |v_2^\varepsilon(0,q^\varepsilon)| \nonumber \\  &+ |\Psi^0(s,q^0)| |v_2^\varepsilon(0,q^\varepsilon) - v_2^0(0,q^0)|.
	\end{aligned}
\]
	Lemma \ref{gronwalllemma} applied with $t_*=T \mu_1^\ell$ and $Q=\{(h^\varepsilon(r,0),r,0,\varepsilon,\mu) \in \mathbb{R} \times [0,2\pi] \times \mathbb{R} \times \mathbb{R} \times (0,\mu_1]:  \varepsilon \in [0,\bar{\varepsilon}_{\mu})\}$ ensures that there is $K_3>0$ such that 
	\begin{align*}
		|\Psi^\varepsilon(s,q^\varepsilon) - \Psi^0(s,q^0)| < K_3 (|h^\varepsilon(q^\varepsilon,0)| + |q^\varepsilon-q^0|  + \varepsilon) < K_3 (K_2+1) \varepsilon,
	\end{align*}
	for all $\mu \in (0,\mu_1]$, $\varepsilon \in [0,\bar{\varepsilon}_{\mu})$, $p \in [0,2\pi]$, and all $s \in [0,T \mu^\ell_1]$. In particular, there is $K_4>0$ such that
	\begin{align} \label{ineq:psidif}
		|\Psi^\varepsilon(\tau,q^\varepsilon) - \Psi^0(\tau,q^0)| < K_4 \varepsilon,
	\end{align}
	for all $\mu \in (0,\mu_1]$, $\varepsilon \in [0,\bar{\varepsilon}_{\mu})$, $p \in [0,2\pi]$, and $\tau \in [0, T \mu^\ell]$. Furthermore, by hypothesis, there is $K_5>0$ such that
\[
		|h^\varepsilon_\sigma(\xi,0)|<K_5,
\]
	for all $\mu \in (0,\mu_1]$, $\varepsilon \in [0,\bar{\varepsilon}_{\mu})$, and all $\xi \in [0,2\pi]$. Hence, we have that
	\begin{align} \label{ineq:v2}
		|v_2^\varepsilon(0,q^\varepsilon)| < K_5+1,
	\end{align}
	for all $\mu \in (0,\mu_1]$, $\varepsilon \in [0,\bar{\varepsilon}_{\mu})$, $p \in [0,2\pi]$, and $\tau \in [0, T \mu^\ell]$. By continuity of $\Psi^0$, there is $K_6>0$ such that
	\begin{align}\label{ineq:psi0}
		|\Psi^0(s,\xi)| < K_6,
	\end{align}
	for all $s \in [0,T\mu_1^\ell]$ and all $\xi \in [0,2\pi]$. 
	
	Henceforth in this proof, all inequalities are to be read as valid for all $\mu \in (0,\mu_1]$, $\varepsilon \in [0,\bar{\varepsilon}_{\mu})$, $p \in [0,2\pi]$, and $\tau \in [0, T \mu^\ell]$, except when stated otherwise, and we shall omit their domain of validity for conciseness. Considering the hypotheses, it follows that
	\begin{align} \label{ineq:v2dif}
		|v_2^\varepsilon(0,q^\varepsilon) - v_2^0(0,q^0)| = |h^\varepsilon_\sigma(q^\varepsilon,0)| < C_0
		\varepsilon. 
	\end{align}
	Combining (\ref{ineq:psidif}), (\ref{ineq:v2}), (\ref{ineq:psi0}), and (\ref{ineq:v2dif}), we conclude that there is $K_7>0$ such that
\[
		|\Psi^\varepsilon(\tau,q^\varepsilon) \cdot v_2^\varepsilon(0,q^\varepsilon) - \Psi^0(\tau,q^0) \cdot v_2^0(0,q^0)| \leq K_7 \varepsilon.
\]
	Thus, it follows that
\[
		|y^\varepsilon(\tau,q^\varepsilon) h_\sigma^\varepsilon(p,\tau) | + |y^\varepsilon(\tau,q^\varepsilon) - y^0(\tau,q^0)| < K_7\varepsilon.
\]
	In particular,
\[
		|y^\varepsilon(\tau,q^\varepsilon)|| h_\sigma^\varepsilon(p,\tau)| \leq K_7 \varepsilon.
\]
	Furthermore, since $0<\bar{\varepsilon}_{\mu} \leq \mu$, another application of the reverse triangle inequality provides
\[
		|y^\varepsilon(\tau,q^\varepsilon)| \geq |y^0(\tau,q^0)| - |y^\varepsilon(\tau,q^\varepsilon)-y^0(\tau,q^0)| \geq |y^0(\tau,q^0)| - K_7 \mu.
\]
	Let $m>0$ denote the minimum value of $y^0$ in $[0,T\mu_1^\ell] \times [0,2\pi]$. Finally, define $\mu_2=\min\{\mu_1, \frac{m}{2K_8}\}$. Then, we conclude that 
\[
		| h_\sigma^\varepsilon(p,\tau)| \leq \frac{2K_{7}}{m} \varepsilon,
\]
	for all $\mu \in (0,\mu_2]$, $\varepsilon \in [0,\bar{\varepsilon}_{\mu})$, $p \in [0,2\pi]$, and $\tau \in [0, T \mu^\ell]$, as we wished to prove.
\end{proof}

\begin{lemma} \label{lemmaestimatesh}
	Suppose system (\ref{systempolartruncated}) has an attracting hyperbolic limit cycle $\Gamma=\{(0,\sigma):\sigma\in\s^1\}$. Then, there are $\widetilde{\mu}>0$ and $\widetilde{C}>0$ such that
	\begin{align*}
		|h^\varepsilon|_{C^0} < \widetilde{C} \varepsilon, \quad |h_\sigma^\varepsilon|_{C^0} < \widetilde{C} \varepsilon, \quad |h_\tau^\varepsilon|_{C^0}<\widetilde{C} \varepsilon,
	\end{align*}
	for all $\varepsilon \in [0,\bar{\varepsilon}_{\mu})$ and all $\mu \in (0,\widetilde{\mu}]$.
\end{lemma}
\begin{proof}
	Since (\ref{systempolartruncated}) has an attracting hyperbolic limit cycle, we can apply Lemma \ref{lemmaapriori} and obtain $N \subset\mathbb{R}^2$ and $C>0$ as stated. Let $R$ and $S$ be as in (\ref{eq:family}). Define
	\begin{align*}
		K:= \sup \left\{|R(\rho, \sigma, \tau/\mu^\ell, \mu)| + |S(\rho, \sigma, \tau/\mu^\ell, \mu)|: (\rho,\sigma,\tau) \in N \times \mathbb{R}, \, \mu \in \left(0,\, \frac{1}{\sqrt[\ell]{T}}\right]\right\} < \infty,
	\end{align*}
	and
	\begin{align*}
		r_0 : = \inf\{d(y,\Gamma): y \in \mathbb{R}^2 \setminus N\}.
	\end{align*}
	Observe that, by definition of $r_0$, if $|h^\varepsilon|_{C^0}<r_0$, then it follows that $M(\varepsilon,\mu) \subset N \times \mathbb{S}^1$.
	
	Let $\mu_1$ be given by
\begin{equation}\label{mu1}
		\mu_1 := \frac{1}{2} \min \left\{\frac{1}{\sqrt[\ell]{T}}, \, \frac{r_0}{CK} \right\},
\end{equation}
	and fix $\mu \in (0,\mu_1]$. Define $I^\mu$ as the maximal interval contained in $[0,\bar{\varepsilon}_{\mu})$ with left endpoint at 0 such that, for $\varepsilon \in I^\mu$, the function $h^\varepsilon$ satisfies $|h^\varepsilon|_{C^0}<r_0$.
	
	It is clear that $I^\mu$ is non-empty, because $0 \in I^\mu$. Furthermore, by definition of $\bar{\varepsilon}_{\mu}$, Fenichel's Theorem can be applied for every $\varepsilon \in I^\mu \subset [0,\bar{\varepsilon}_{\mu}) \subset A^\mu$, guaranteeing that $I^\mu$ is relatively open in $[0,\bar{\varepsilon}_{\mu})$. 
	
	We shall show that $I^\mu$ is relatively closed in $[0,\bar{\varepsilon}_{\mu})$, so that, by connectedness, we may conclude that $I^\mu = [0,\bar{\varepsilon}_{\mu})$. Let $\alpha_{\mu}$ be the least upper bound of $I^\mu\subset[0,\bar{\varepsilon}_{\mu})$. It follows immediately that $\alpha_{\mu} \leq \bar{\varepsilon}_{\mu}$. If $\alpha_{\mu}=\bar{\varepsilon}_{\mu}$, then $I^\mu=[0,\bar{\varepsilon}_{\mu})$, and we are done. 
	
	Alternatively, suppose that $\alpha_{\mu} \in [0,\bar{\varepsilon}_{\mu}) \subset A^\mu$.  Then, it follows that $[0,\alpha_{\mu}] \in [0,\bar{\varepsilon}_{\mu}) \subset A^\mu$. Therefore, Fenichel's Theorem ensures that the family $h^\varepsilon$ is continuous with respect to its parameter $\varepsilon$ if $0\leq \varepsilon \leq \alpha_{\mu}$. Since $0< \mu \leq \mu_1$ and $|h^\varepsilon|_{C^0} < r_0$ for all $\varepsilon \in [0,\alpha_{\mu})$, it follows by Lemma \ref{lemmaapriori} that 
	\begin{align*}
		|h^\varepsilon|_{C^0}< CK \alpha_{\mu} < CK \bar{\varepsilon}_{\mu} \leq CK \mu \leq CK \widetilde{\mu} \leq \frac{r_0}{2},
	\end{align*} 
	for all $\varepsilon \in [0,\alpha_{\mu})$. From continuity of the family $h^\varepsilon$ with respect to $\varepsilon$ for $\varepsilon \in [0,\alpha_{\mu})$, it follows by taking the limit as $\varepsilon \to \alpha_{\mu}$ that
	\begin{align*}
		|h^{\alpha_{\mu}}|_{C^0} \leq \frac{r_0}{2} < r_0.
	\end{align*}
	Thus, we conclude that $\alpha_{\mu} \in I^\mu$, so that $I^\mu$ is relatively closed. 
	
	That said, it follows that $I^\mu = [0,\bar{\varepsilon}_{\mu})$ for every $\mu \in (0,\mu_1]$, which means that 
	\begin{align} \label{ineq:deltahc0}
		|h^\varepsilon|_{C^0} < CK\varepsilon = \widetilde{C}_1 \varepsilon,
	\end{align}
	for all $\mu \in (0,\mu_1]$ and all $\varepsilon \in [0,\bar{\varepsilon}_{\mu})$, where we have defined $\widetilde{C}_1:=CK$ for conciseness.
	
	Once again, let $\Psi^\varepsilon(s,q)$ be the principal fundamental matrix solution of $E^{\varepsilon,\mu}$ at $t=0$ along $\gamma^\varepsilon(s,q)$. Observe that $\gamma^0(s,q)=(0,\sigma^0(s,q),s)$ is a solution to the unperturbed system $E^0$:
\[
		E^{0}:
		\begin{cases}
			\rho'=f(\rho,\sigma),  \\
			\sigma'= g(\rho, \sigma),  \\
			\tau'=1,
		\end{cases}
\]
	By considering that $\gamma^0$ must be a solution to $E^0$, it follows that $f(0,\sigma^0(s,q)) = 0$. Hence, since $\big(f(0,\sigma^0(s,q)),g(0,\sigma^0(s,q)),0\big)$ is a solution to the first variational equation of $E^0$ along $\gamma^0(s,q)$, we have that 
	\begin{align} \label{eq:variationalatzero}
		\Psi^0(s,q)
		\begin{pmatrix}
			0 \\
			g(0,q) \\
			0
		\end{pmatrix} = 
		\begin{pmatrix}
			0 \\ 
			g(0,\sigma^0(s,q)) \\
			0
		\end{pmatrix}.
	\end{align}
	From (\ref{eq:variationalatzero}), it follows that 
\[
	\psi^0_{22}(s,q) = \frac{g(0,\sigma^0(s,q))}{g(0,q)}.
\]
	Therefore, since $g$ cannot vanish on the limit cycle $\Gamma$, there is a constant $M_0>0$ such that $|\psi^0_{22}(s,q)|\geq M_0$ for all $s \in \mathbb{R}$ and all $q \in \mathbb{S}^1$.
	
	Notice that $\Psi^0(T,q)$ has two eigenvalues equal to $1$, corresponding to $(f(0,q),g(0,q),1)$ and $(0,0,1)$, and one eigenvalue which is equal to the eigenvalue $\lambda_P$ of the derivative of the Poincaré map on  any transversal section of $\Gamma$ at its fixed point, thus less than one. Hence, it follows that $\det \Psi^0(T,q)= \lambda_P<1$ for all $q \in [0,2\pi]$. Moreover, if $\mathcal{F}^0$ denotes the vector field appearing in $E^0$, the function
	\begin{align} \label{def:I(t,q)}
		I(t,q) := \exp \int_0^t \tr \, \big( d\mathcal{F}^0(0,\sigma^0(s,q),s) \big) ds
	\end{align}
	has a maximum $M_I$ in the domain $[0,T] \times \mathbb{S}^1$. For any $t\geq0$, we may find $z \in \mathbb{N}$ and $r_t \in [0,T)$ such that $t= z T + r_t$. Then, by Liouville's formula, it follows that
	\begin{align} \label{eq:liouvillesformula}
		\det \Psi^0(t,q) = \exp \int_0^{zT+r_t} \tr \big( d\mathcal{F}^0(0,\sigma^0(s,q),s) \big) ds \leq M_I \,(\det\Psi^0(T,q))^z.
	\end{align}
	Therefore, by taking $T_0$ sufficiently large, we can ensure that $\det \Psi^0(t,q) \leq M_0^2/4$ for all $t\geq T_0$. Now, since $0<\mu<\mu_1$, we know that $0<T\, \mu^\ell<1$, so that we can define $n_\mu$ as the smallest positive integer such that $T_0\leq T \mu^\ell n_\mu < T_0+1$. We also define the value $T_\mu:= T\mu^\ell n_\mu$. We remark that, even though $T_\mu$ varies depending on the choice of $\mu$, its value is always an integer multiple of the period of the perturbation terms present in $E^{\varepsilon,\mu}$, and that $T_\mu$ tends to $T_0$ as $\mu$ is made close to zero.
	
	Observe that, for any $\mu \in (0,\mu_1]$ and any $\varepsilon \in [0,\bar{\varepsilon}_{\mu})$, the function
	\begin{align*}
		v_2^\varepsilon(s,q) = 
		\begin{pmatrix}
			h_\sigma^\varepsilon(\sigma^\varepsilon(s,q),s) \, y^\varepsilon(s,q) \\
			y^\varepsilon(s,q) \\
			0
		\end{pmatrix} = \Psi^\varepsilon(s,q) \cdot
		\begin{pmatrix}
			h_\sigma^\varepsilon(q,0) \\
			1 \\
			0
		\end{pmatrix}
	\end{align*}
	is a solution of the first variational equation of $E^{\varepsilon,\mu}$ along $\gamma^\varepsilon(s,q)$. Thus, it follows that 
	\begin{align} \label{eq:hsigmaandpsi}
		h_\sigma^\varepsilon (\sigma^\varepsilon(s,q),s) = \frac{\psi_{11}^\varepsilon(s,q) h_\sigma^\varepsilon(q,0)+\psi^\varepsilon_{12}(s,q)}{\psi_{21}(s,q) h_\sigma^\varepsilon(q,0)+\psi_{22}^\varepsilon(s,q) }.
	\end{align}
	
	Next, we shall prove the following claim: there are $\mu_2>0$, $C'>0$, and $r_1>0$ such that, for all $\mu \in (0,\mu_2]$ and all $\varepsilon \in [0,\bar{\varepsilon}_{\mu})$, if $|h_\sigma^\varepsilon|_{C^0}\leq r_1$, then $|h_\sigma^\varepsilon|_{C^0}\leq C' \varepsilon$. In order to do so, let $\mu \in (0,\mu_1]$, $\varepsilon \in [0,\bar{\varepsilon}_{\mu})$, $p \in [0,2\pi]$ and define $q^\varepsilon$ by $p=\sigma^\varepsilon(T_\mu,q^\varepsilon)$. Then, by setting $s=T_\mu$ in (\ref{eq:hsigmaandpsi}), the following identity is obtained:
	\begin{align} \label{eq:hsigmap}
		h_\sigma^\varepsilon (p,0) = \frac{\psi_{11}^\varepsilon(T_\mu,q^\varepsilon) h_\sigma^\varepsilon(q^\varepsilon,0)+\psi^\varepsilon_{12}(T_\mu,q^\varepsilon)}{\psi_{21}(T_\mu,q^\varepsilon) h_\sigma^\varepsilon(q^\varepsilon,0)+\psi_{22}^\varepsilon(T_\mu,q^\varepsilon) }.
	\end{align}
	In particular, by taking $\varepsilon=0$, we obtain
	\begin{align} \label{eq:hsigmapzero}
		\frac{\psi^0_{12}(T_\mu,q^0)}{\psi_{22}^0(T_\mu,q^0) } =0.
	\end{align}

	By Lemma \ref{lemmatecnico} applied with $B=[T_0,T_0+1)$, there is $r_1>0$ such that, if $t \in[T_0,T_0+1)$, $\xi \in \mathbb{R}$, and $D \in M_2(\mathbb{R})$ satisfy $|\xi|< r_1$ and $|D-\Psi^0(t,q)|<r_1$, then
	\begin{align} \label{ineq:conlusionofthelemma}
		|d_{21} \xi +d_{22}| > \frac{2}{3} M_0, \qquad |\det D - \det \Psi^0(t,q)|< \frac{1}{8}M_0^2.
	\end{align}
	Furthermore, since $T_\mu$ must be in the bounded set $[T_0,T_0+1)$, Lemma \ref{gronwalllemma} applied to $E^{\varepsilon,\mu}$ ensures that there are $C_1,C_2>0$ such that
	\begin{align} \label{ineq:gronwall1}
		|h^\varepsilon(q^\varepsilon,0)| + |q^\varepsilon- q^0| \leq C_1 (|h^\varepsilon(p,0)|+\varepsilon)
	\end{align}
	and
	\begin{align}\label{ineq:gronwall2}
		|\Psi^\varepsilon(T_\mu,q^\varepsilon) - \Psi^0(T_\mu,q^0)| \leq C_2(|h^\varepsilon(q^\varepsilon,0)| + |q^\varepsilon-q^0|+\varepsilon),
	\end{align}
	for all $\mu \in (0,\mu_1]$, $\varepsilon \in [0,\bar{\varepsilon}_{\mu})$, and all $p \in [0,2\pi]$.
	Considering (\ref{ineq:deltahc0}) and (\ref{ineq:gronwall1}), it follows that there is $C_3>0$ such that 
\[
		|q^\varepsilon-q^0| <C_3 \varepsilon.
\]
	Combining (\ref{ineq:deltahc0}), (\ref{ineq:gronwall1}), and (\ref{ineq:gronwall2}), it follows that there is $C_4>0$ such that
	\begin{align} \label{ineq:psiepsilon-psi0}
		|\Psi^\varepsilon(T_\mu,q^\varepsilon) - \Psi^0(T_\mu,q^0)| \leq C_4\varepsilon \leq C_4 \mu_1,
	\end{align}
	for all $\mu \in (0,\mu_1]$, $\varepsilon \in [0,\bar{\varepsilon}_{\mu})$, and all $p \in [0,2\pi]$.
	
	From (\ref{eq:hsigmap}) and (\ref{eq:hsigmapzero}), it follows that 
	\begin{align} \label{eq:deltahsigma}
		|h_\sigma^\varepsilon(p,0) |\leq \left|\frac{\psi_{11}^\varepsilon(T_\mu,q^\varepsilon)h_\sigma^\varepsilon(q^\varepsilon,0) +\psi_{12}^\varepsilon(T_\mu,q^\varepsilon)}{\psi_{21}^\varepsilon(T_\mu,q^\varepsilon)h_\sigma^\varepsilon(q^\varepsilon,0) +\psi_{22}^\varepsilon(T_\mu,q^\varepsilon)}  - \frac{\psi_{12}^\varepsilon(T_\mu,q^\varepsilon)}{\psi_{22}^\varepsilon(T_\mu,q^\varepsilon)} \right| 
		+\left|\frac{\psi_{12}^\varepsilon(T_\mu,q^\varepsilon)}{\psi_{22}^\varepsilon(T_\mu,q^\varepsilon)} - \frac{\psi^0_{12}(T_\mu,q^0)}{\psi_{22}^0(T_\mu,q^0) }\right|.
	\end{align}
	In order to simplify notation, we name the first and second summand on the right-hand side of the inequality above as $I$ and $II$, respectively. To estimate those quantities, we define the function $u:\mathbb{R} \times M_2(\mathbb{R}) \to \mathbb{R}$ by
	\begin{align*}
		u(\xi,D) = \frac{d_{11}\xi + d_{12}}{d_{21}\xi + d_{22}},
	\end{align*}
	so that $I$ and $II$ are given by
	\begin{align*}
		&I= \big|u\big(h^\varepsilon_\sigma(q^\varepsilon,0),\Psi^\varepsilon(T_\mu,q^\varepsilon)\big) - u\big(0,\Psi^\varepsilon(T_\mu,q^\varepsilon)\big) \big|, \\
		&II=\big|u\big(0,\Psi^\varepsilon(T_\mu,q^\varepsilon)\big) - u\big(0,\Psi^0(T_\mu,q^0)\big) \big|.
	\end{align*}
	
	To estimate $I$, observe that 
	\begin{align*}
		u_\xi (\xi, D) = \frac{\det D}{(d_{21}\xi + d_{22})^2}.
	\end{align*}
	Hence, the mean value theorem applied to the real function $\xi \mapsto u(\xi,D)$ guarantees that there is $\xi$ between $0$ and $h^\varepsilon_\sigma(q^\varepsilon,0)$ such that:
\[
		I \leq \left|\frac{\det \Psi^\varepsilon(T_\mu,q^\varepsilon)}{(\psi^\varepsilon_{21}(T_\mu,q^\varepsilon)\xi+\psi^\varepsilon_{22}(T_\mu,q^\varepsilon))^2}\right| |h^\varepsilon_\sigma(q^\varepsilon,0)|.
\]
	Define 
\[
		\mu_2:=\min\left\{\mu_1, \frac{r_1}{2C_4} \right\}.
\]
	By (\ref{ineq:psiepsilon-psi0}) and considering that $0\leq\varepsilon\leq \bar{\varepsilon}_{\mu}<\mu$, we conclude that by further restricting the value of $\mu$ to be in $(0,\mu_2]$, it follows that $|\Psi^\varepsilon(T_\mu,q^\varepsilon)-\Psi^0(T_\mu,q^0)|<r_1$. Suppose then that $|h^\varepsilon_\sigma|_{C^0}<r_1$, as in the hypotheses of the claim. By (\ref{ineq:conlusionofthelemma}) and an application of the triangle inequality, we conclude that 
	\begin{align}\label{ineq:I}
		I \leq \frac{3}{8} M_0^2 \cdot \frac{9}{4 M_0^2} \cdot |h^\varepsilon_\sigma(q^\varepsilon,0)| \leq \frac{27}{32} \; \sup_{\alpha \in [0,2\pi]} |h_\sigma^\varepsilon(\alpha,0)|,
	\end{align}
	for all $\mu \in (0,\mu_2]$, $\varepsilon \in [0,\bar{\varepsilon}_{\mu})$, and all $p \in [0,2\pi]$.
	
	Now, to estimate $II$, we remark that $\{\Psi^0(s,q): s \in [T_0,T_0+1], q \in \mathbb{S}^1\}$ is a compact set.  Hence, considering (\ref{ineq:psiepsilon-psi0}) combined with an application of the triangle inequality, it follows that the set $\{\Psi^\varepsilon(T_\mu,q): \mu \in (0,\mu_2],  \; \varepsilon \in [0,\bar{\varepsilon}_{\mu}), \; q \in \mathbb{S}^1 \} \subset M_2(\mathbb{R})$ is bounded. Therefore, the function $u$ must be Lipschitz continuous on the set
	\begin{align*}
		\{\big(0,\Psi^\varepsilon(T_\mu,q)\big) \in \mathbb{R} \times M_2(\mathbb{R}): \mu \in(0,\mu_2], \; \varepsilon \in [0,\bar{\varepsilon}_{\mu}), \; q \in \mathbb{S}^1 \}.
	\end{align*}
	Thus, there is $L_1>0$ such that
	\begin{align*}
		II \leq L_1 |\Psi^\varepsilon(T_\mu,q^\varepsilon) - \Psi^0(T_\mu,q^0)|.
	\end{align*}
	Considering (\ref{ineq:psiepsilon-psi0}) and defining the constant $C_5:= L_1 C_4$, we conclude that $II\leq C_5 \varepsilon$ for all $\mu \in(0,\mu_2]$, $ \varepsilon \in [0,\bar{\varepsilon}_{\mu})$, and all $p \in [0,2\pi]$. Hence, taking into consideration (\ref{eq:deltahsigma}) and (\ref{ineq:I}), it follows that 
	\begin{align*}
		|h_\sigma^\varepsilon(p,0)| \leq I + II \leq \frac{27}{32} \; \sup_{\alpha \in [0,2\pi]} |h_\sigma^\varepsilon(\alpha,0)| + C_5\varepsilon,
	\end{align*} 
	wherefrom, by taking the supremum on the left-hand side, we conclude that
	\begin{align*}
		\sup_{\alpha \in [0,2\pi]} |h_\sigma^\varepsilon(\alpha,0) | \leq \frac{32}{5} C_5 \varepsilon,
	\end{align*}
	for all $\mu \in(0,\mu_2]$ and all $\varepsilon \in [0,\bar{\varepsilon}_{\mu})$. Then, by Lemma \ref{lemmareduction}, it follows that there is $C'>0$ such that 
	\begin{align*}
		|h^\varepsilon_\sigma |_{C^0} \leq C' \varepsilon,
	\end{align*}
	for all $\mu \in(0,\mu_2]$ and all $\varepsilon \in [0,\bar{\varepsilon}_{\mu})$, finishing the proof of the claim.
	
	Having proved the claim, we proceed to showing that there are $\widetilde{C}_2>0$ and $\mu_3>0$ such that $|h_\sigma^\varepsilon |_{C^0} < \widetilde{C}_2 \varepsilon$ for all $\mu \in (0,\mu_3]$ and all $\varepsilon \in [0,\bar{\varepsilon}_{\mu})$. In order to do so, define
	\begin{align*}
		\mu_3:= \min \left\{\mu_2, \frac{r_1}{2C'} \right\}.
	\end{align*}
	For each $\mu \in (0,\mu_3]$, let $J^\mu$ be the maximal interval contained in $[0,\bar{\varepsilon}_{\mu})$ with left endpoint at $0$ such that, for $\varepsilon \in J^\mu$, the function $h^\varepsilon$ satisfies $|h^\varepsilon_\sigma|_{C^0}\leq r_1$. We will prove that $J^\mu=[0,\bar{\varepsilon}_{\mu})$.
	
	Observe that $J^\mu$ is non-empty because $0 \in J^\mu$. Moreover, Fenichel's theorem ensures that $J^\mu$ is relatively open in $[0,\bar{\varepsilon}_{\mu})$. Thus, if we prove that $J^\mu$ is also relatively closed in $[0,\bar{\varepsilon}_{\mu})$, by connectedness of $[0,\bar{\varepsilon}_{\mu})$, we will have that $J^\mu = [0,\bar{\varepsilon}_{\mu})$.
	
	Let $\alpha'_{\mu}$ be the least upper bound of $J^\mu \subset [0,\bar{\varepsilon}_{\mu})$. Then, we know that $\alpha'_{\mu}\leq \bar{\varepsilon}_{\mu}$. If $\alpha'_{\mu} = \bar{\varepsilon}_{\mu}$, then $J^\mu = [0,\bar{\varepsilon}_{\mu})$, and we are done. On the other hand, if $\alpha'_{\mu}<\bar{\varepsilon}_{\mu}$, then $[0,\alpha'_{\mu}] \subset [0,\bar{\varepsilon}_{\mu}) \subset A^\mu$. In that case, Fenichel's Theorem ensures that the family $h^\varepsilon$ is continuous in the $C^1$-norm with respect to $\varepsilon$ if $0\leq \varepsilon \leq \alpha$. By definition of $J^\mu$, since $\mu \in (0,\mu_3]$, and considering the claim, it follows that
	\begin{align*}
		|h^{\varepsilon}_\sigma|_{C^0} \leq C' \varepsilon < C' \bar{\varepsilon}_{\mu} \leq C' \mu \leq C' \frac{r_1}{2C'} \leq \frac{r_1}{2},
	\end{align*}
	for all $\varepsilon \in [0,\alpha'_{\mu})$. By passing to the limit $\varepsilon \to \alpha'_{\mu}$, which we can do because of the continuity in the $C^1$-norm of the family $h^\varepsilon$, it follows that 
	\begin{align*}
		|h^{\alpha'_{\mu}}_\sigma|_{C^0}\leq r_1/2 \leq r_1,
	\end{align*}
	so that $\alpha'_{\mu} \in J^\mu$, proving that $J^\mu$ is relatively closed. Hence, we conclude that $J^\mu=[0,\bar{\varepsilon}_{\mu})$ if $\mu \in (0,\mu_3]$.
	
	Let us define $\widetilde{C}_2 = C'>0$. The fact that $J^\mu = [0,\bar{\varepsilon}_{\mu})$ if $\mu \in (0,\mu_3]$, in combination with the claim, guarantees that
	\begin{align} \label{ineq: resultsigma}
		|h^\varepsilon_\sigma|_{C^0} \leq C'\varepsilon = \widetilde{C}_2 \varepsilon,
	\end{align} 
	for all $\mu \in (0,\mu_3]$ and all $\varepsilon \in [0,\bar{\varepsilon}_{\mu})$.
	
	Finally, let us estimate $h_\tau^\varepsilon$. Since
	\begin{align*}
		f(0,\sigma^0(s,q)) =0,
	\end{align*}
	for all $s \in \mathbb{R}$ and all $q \in [0,2\pi]$, it follows that $f$ vanishes when $\rho = 0$, regardless of the choice of $\sigma \in [0,2\pi]$. Moreover, $f$ is Lipschitz continuous on the compact set $[-\widetilde{C}_1 \mu_3,\widetilde{C}_1 \mu_3] \times [0,2\pi]$. Since $|h^\varepsilon |_{C^0}< \widetilde{C}_1 \varepsilon < \widetilde{C}_1 \mu_3$, we conclude that there is $L_3>0$ such that
	\begin{align} \label{ineq:f-h0g}
		|f(h^\varepsilon(\sigma,\tau),\sigma)|=|f(h^\varepsilon(\sigma,\tau),\sigma) - f(0,\sigma)| \leq L_3 |h^\varepsilon(\sigma,\tau)| < L_3 \widetilde{C}_1 \varepsilon,
	\end{align}
	for all $\sigma \in [0,2\pi]$, $\tau \in [0,T \mu^\ell]$, $\mu \in (0,\mu_3]$, and all $\varepsilon \in [0,\bar{\varepsilon}_{\mu})$. The invariance of the graph of $h^\varepsilon$ with respect to $E^{\varepsilon,\mu}$ is equivalent to the following identities:
\[\begin{aligned}
		h_\sigma^\varepsilon(\sigma^\varepsilon,s) (\sigma^\varepsilon)' + h_\tau^\varepsilon (\sigma^\varepsilon,s) &= f(h^\varepsilon(\sigma^\varepsilon,s),\sigma^\varepsilon) + \varepsilon R (h^\varepsilon(\sigma^\varepsilon,s), \sigma^\varepsilon, s/\mu^\ell,\mu) \nonumber \\
		(\sigma^\varepsilon)' &= g(h^\varepsilon(\sigma^\varepsilon,s),\sigma^\varepsilon) + \varepsilon S (h^\varepsilon(\sigma^\varepsilon,s), \sigma^\varepsilon, s/\mu^\ell,\mu),
	\end{aligned}
	\]
	where the arguments of $\sigma^\varepsilon$ and its derivative $(\sigma)'$ have been omitted for conciseness, but should be read as $(s,q)$. Since $(s,q)$ may be chosen arbitrarily, and by substituting the value of $(\sigma^\varepsilon)'$ given in the second equation in the first one, it follows that 
	\begin{align} \label{eq:identityhtau}
		h_\sigma^\varepsilon(\sigma,\tau)\big( g(h^\varepsilon(\sigma,\tau),\sigma) + \varepsilon S(h^\varepsilon(\sigma,\tau),\sigma,&\tau/\mu^\ell,\mu)\big)+ h_\tau^\varepsilon(\sigma,\tau) =  \nonumber \\ &f(h^\varepsilon(\sigma,\tau),\sigma)+ \varepsilon R(h^\varepsilon(\sigma,\tau),\sigma,\tau/\mu^\ell,\mu),
	\end{align}
	for all $\mu \in(0,\mu_3]$, $\varepsilon \in [0,\bar{\varepsilon}_{\mu})$, $\sigma \in [0,2\pi]$, and all $\tau \in [0,T \mu^\ell]$. 
	
	We define the function $H(\sigma,\tau,\varepsilon)$ by $h^\varepsilon(\sigma,\tau) = \varepsilon H(\sigma,\tau,\varepsilon)$. Observe that (\ref{ineq: resultsigma}) ensures that $H_\sigma(\sigma,\tau,\varepsilon) \leq \widetilde{C}_2$ for all $\mu \in(0,\mu_3]$, $\varepsilon \in [0,\bar{\varepsilon}_{\mu})$, $\sigma \in [0,2\pi]$, and all $\tau \in [0,T \mu^\ell]$.  Hence, from (\ref{eq:identityhtau}), we conclude that:
	
	\begin{align} \label{eq:invarianceH}
		\big(g(h^\varepsilon(\sigma,\tau),\sigma) + \varepsilon S(h^\varepsilon(\sigma,\tau),\sigma,\tau/\mu^\ell,\mu) \big) H_\sigma(\sigma,\tau,\varepsilon) + H_\tau(\sigma,\tau,\varepsilon) &=\nonumber \\
		R(h^\varepsilon(\sigma,\tau),\sigma,\tau/\mu^\ell,\mu) &+ \frac{1}{\varepsilon} \big(f(h^\varepsilon(\sigma,\tau),\sigma) \big), 
	\end{align}
	for all $\mu \in(0,\mu_3]$, $\varepsilon \in [0,\bar{\varepsilon}_{\mu})$, $\sigma \in [0,2\pi]$, and all $\tau \in [0,T \mu^\ell]$. Observe that $g$ is bounded, say by $K_g>0$, on the compact set $[-\widetilde{C}_1 \mu_3,\widetilde{C}_1 \mu_3] \times [0,2\pi]$. Furthermore, $R$ and $S$ are bounded by the constant $K>0$ defined at the beginning of the proof of this lemma for $\mu \in(0,\mu_3]$, $\varepsilon \in [0,\bar{\varepsilon}_{\mu})$, $\sigma \in [0,2\pi]$, and $\tau \in [0,T \mu^\ell]$. Therefore, an application of the triangle inequality combined with inequality (\ref{ineq:f-h0g}) ensures that
	\begin{align} \label{ineq:Htau}
		|H_\tau|_{C^0}\leq (K_g + \mu_3 K)\widetilde{C}_2 + K  + L_3\widetilde{C}_1,
	\end{align}
	for all $\mu \in(0,\mu_3]$, $\varepsilon \in [0,\bar{\varepsilon}_{\mu})$, $\sigma \in [0,2\pi]$, and all $\tau \in [0,T \mu^\ell]$. Finally, by defining the constant $\widetilde{C}_3 :=(K_g + \mu_3 K)\widetilde{C}_2 + K + L_3\widetilde{C}_1>0$ and by the definition of $H$, we conclude from (\ref{ineq:Htau}) that 
\[
		|h^\varepsilon_\tau|_{C^0} \leq \widetilde{C}_3 \varepsilon,
\]
	for all $\mu \in(0,\mu_3]$, $\varepsilon \in [0,\bar{\varepsilon}_{\mu})$, $\sigma \in [0,2\pi]$, and all $\tau \in [0,T \mu^\ell]$. The proof is thus concluded by defining $\widetilde{C}:=\max\{\widetilde{C}_1,\widetilde{C}_2,\widetilde{C}_3\}$ and $\widetilde{\mu}:= \mu_3$.
\end{proof}

Having proved the first part of the necessary estimates, we proceed to finding an upper and a lower bound for the value of $y^\varepsilon(s,q)$, as in \cite[Lemma 8.4]{MR1740943}. This argument requires deeper revision, since the value of the exponent $\ell$ is actively present therein. The following Lemma will be essential in the task of generalizing that result.
\begin{lemma} \label{lemmaidentity}
	Let $S_\tau$ denote the derivative of $S$ with respect to its third entry. For all $n\geq2$ the following equality holds:
	\begin{align} \label{eq:lemmaest}
		\frac{\varepsilon^2}{\mu^\ell} \frac{S S_\tau}{G^2} =& \sum_{k=2}^n \frac{\varepsilon^k}{k} \frac{d}{ds} \left( \frac{S^k}{G^k}\right) - (S_\rho h_\sigma^\varepsilon G + S_\rho h^\varepsilon_\tau+S_\sigma G) \sum_{k=2}^{n}\varepsilon^k\frac{ S^{k-1}}{G^k} \nonumber \\
		&+ (G_\rho h^\varepsilon_\sigma G + G_\rho h^\varepsilon_\tau + G_\sigma G) \sum_{k=2}^{n} \varepsilon^k \frac{S^k}{G^{k+1}} + \frac{\varepsilon^{n+1}}{\mu^\ell} \frac{S^n  S_\tau}{G^{n+1}},
	\end{align}
	where the argument of the function $G$ and its derivatives is $(h^\varepsilon(\sigma^\varepsilon(s,q),\tau(s)),\sigma^\varepsilon(s,q),\tau(s),\mu,\varepsilon)$, the argument of $h^\varepsilon$ and its derivatives is given by $(\sigma^\varepsilon(s,q),\tau(s))$ and the argument of $S$ and its derivatives is given by the expression $(h^\varepsilon(\sigma^\varepsilon(s,q),\tau(s)),\sigma^\varepsilon(s,q),\tau(s)/\mu^\ell,\mu)$.
\end{lemma}
\begin{proof}
	We will omit the arguments of the functions in order to simplify notation. The proof is done by induction. For $n=2$, we simply differentiate $S^2/G^2$ with respect to $s$:
	\begin{align*}
		\frac{d}{ds} \left( \frac{S^2}{G^2} \right) =  &\frac{2S(S_\rho h^\varepsilon_\sigma G + S_\rho h^\varepsilon_\tau + S_\sigma G)}{G^2} +\frac{2}{\mu^\ell }\frac{S S_\tau}{ G^2} \\&- \frac{2 S^2(G_\rho h^\varepsilon_\sigma G + G_\rho h^\varepsilon_\tau + G_\sigma G)}{G^3} -\frac{2\varepsilon}{\mu^\ell} \frac{S^2 S_\tau}{G^3}.
	\end{align*}
	Rearranging the terms and multiplying by $\varepsilon^2/2$, we get
	\begin{align*}
		\frac{\varepsilon^2}{\mu^\ell}\frac{S S_\tau}{ G^2} = 
		&\frac{\varepsilon^2}{2} \frac{d}{ds} \left( \frac{S^2}{G^2} \right) - \varepsilon^2 S \frac{(S_\rho h^\varepsilon_\sigma G + S_\rho h^\varepsilon_\tau + S_\sigma G)}{G^2} \\
		&+ \varepsilon^2 S^2 \frac{(G_\rho h^\varepsilon_\sigma G + G_\rho h^\varepsilon_\tau + G_\sigma G)}{G^3} + \frac{\varepsilon^3}{\mu^\ell}  \frac{S^2 S_\tau}{G^3},
	\end{align*}
	as wanted. 
	
	Suppose that \eqref{eq:lemmaest} is valid for some $n\geq2$. We must prove that it is also valid for $n+1$ in this case. We differentiate $S^{n+1}/G^{n+1}$ with respect to $s$:
	\begin{align*}
		\frac{d}{ds} \left( \frac{S^{n+1}}{G^{n+1}} \right) =  &\frac{(n+1)S^n(S_\rho h^\varepsilon_\sigma G + S_\rho h^\varepsilon_\tau + S_\sigma G)}{G^{n+1}} +\frac{(n+1)}{\mu^\ell }\frac{S^n S_\tau}{ G^{n+1}} \\&- \frac{(n+1) S^{n+1}(G_\rho h^\varepsilon_\sigma G + G_\rho h^\varepsilon_\tau + G_\sigma G)}{G^{n+2}} -\frac{(n+1)\varepsilon}{\mu^\ell} \frac{S^{n+1} S_\tau}{G^{n+2}}.
	\end{align*}
	Rearranging and multiplying by $\varepsilon^{n+1}/(n+1)$, we get:
	\begin{align*}
		\frac{\varepsilon^{n+1}}{\mu^\ell}\frac{S^n S_\tau}{ G^{n+1}} = 
		&\frac{\varepsilon^{n+1}}{(n+1)} \frac{d}{ds} \left( \frac{S^{n+1}}{G^{n+1}} \right) - \varepsilon^{n+1} S^n \frac{(S_\rho h^\varepsilon_\sigma G + S_\rho h^\varepsilon_\tau + S_\sigma G)}{G^{n+1}} \\
		&+ \varepsilon^{n+1} S^{n+1} \frac{(G_\rho h^\varepsilon_\sigma G + G_\rho h^\varepsilon_\tau + G_\sigma G)}{G^{n+2}} + \frac{\varepsilon^{n+2}}{\mu^\ell}  \frac{S^{n+1} S_\tau}{G^{n+2}},
	\end{align*}
	By substituting this expression in \eqref{eq:lemmaest}, we get
	\begin{align*}
		\frac{\varepsilon^2}{\mu^\ell} \frac{S S_\tau}{G^2} =& \sum_{k=2}^{n+1} \frac{\varepsilon^k}{k} \frac{d}{ds} \left( \frac{S^k}{G^k}\right) - (S_\rho h_\sigma^\varepsilon G + S_\rho h^\varepsilon_\tau+S_\sigma G) \sum_{k=2}^{n+1}\varepsilon^k\frac{ S^{k-1}}{G^k} \nonumber \\
		&+ (G_\rho h^\varepsilon_\sigma G + G_\rho h^\varepsilon_\tau + G_\sigma G) \sum_{k=2}^{n+1} \varepsilon^k \frac{S^k}{G^{k+1}} + \frac{\varepsilon^{n+2}}{\mu^\ell} \frac{S^{n+1}  S_\tau}{G^{n+2}},
	\end{align*}
	which is the formula for $n+1$. The lemma is thus proved.
\end{proof}
We proceed to effectively establishing bounds for the value of $y^\varepsilon(s,q)$. Once again, we ensure that those bounds do not depend on the particular choice of $\mu$. 
\begin{lemma} \label{lemmayepsilon}
	Suppose that (\ref{systempolartruncated}) has an attracting hyperbolic limit cycle $\Gamma=\{(0,\sigma):\sigma\in\s^1\}$. Then, there are $\widetilde{\mu}>0$ and $C>0$ such that
	\begin{align*}
		e^{-C} e^{-C \varepsilon |s|} \leq y^\varepsilon (s,q) \leq e^{C} e^{C \varepsilon |s|},
	\end{align*}
	for all $s \in \mathbb{R}$, $q \in [0,2\pi]$, $\mu \in (0,\widetilde{\mu}]$, and all $\varepsilon \in [0,\bar{\varepsilon}_{\mu})$.
\end{lemma}
\begin{proof}
	First, we shall prove that, if $\mu>0$ is sufficiently small, then the function $s \mapsto G(h^\varepsilon(\sigma^\varepsilon(s),\tau(s)),\sigma^\varepsilon(s),\tau(s),\mu,\varepsilon)$ has no zeros for all $s \in \mathbb{R}$, $q \in [0,2\pi]$, and all $\varepsilon \in [0,\bar{\varepsilon}_{\mu})$. In order to so, we remark that, by Lemma \ref{lemmaestimatesh}, there are $K_0>0$ and $\mu_0>0$ such that
\[
		|h^\varepsilon(\sigma^\varepsilon(s,q),\tau(s)) | \leq K_0\mu,
\]
	for all $s\in \mathbb{R}$, $q \in [0,2\pi]$, $\mu \in (0,\mu_0]$, and all $\varepsilon \in [0,\bar{\varepsilon}_{\mu})$. Let $I_0$  denote the interval $[-K_0\mu_0, \, K_0\mu_0]$. Observe that $h^\varepsilon(\sigma,\tau) \in I_0$ for all $(\sigma,\tau) \in \mathbb{R}^2$.  Since $g$ is Lipschitz continuous on $I_0 \times [0,2\pi]$, it follows that there is $L_0>0$ such that
\[
		|g(h^\varepsilon(\sigma^\varepsilon(s,q),\tau(s)),\sigma^\varepsilon(s,q)) - g(0,\sigma^\varepsilon(s,q))| \leq L_0 K_0 \mu,
\]
	for all $s \in \mathbb{R}$, $q \in [0,2\pi]$, $\mu \in (0,\mu_0]$, and all $\varepsilon \in [0,\bar{\varepsilon}_{\mu})$. Define 
	\begin{align*}
		K_S:= \sup \left\{|S(\rho, \sigma, t, \mu)|: (\rho,\sigma,t,\mu) \in I_0 \times \mathbb{R} \times \mathbb{R} \times [0,\mu_0]  \right\} < \infty.
	\end{align*}
	By the reverse triangle inequality, and considering that $0\leq\varepsilon<\bar{\varepsilon}_{\mu}\leq \mu$, we have
	\begin{align} \label{ineq:Gtriangle}
		|G(h^\varepsilon(\sigma,\tau),\sigma,\tau,\mu,\varepsilon)| \geq |g(0,\sigma)| -|g(h^\varepsilon(\sigma,\tau),\sigma) - g(0,\sigma)| - \mu |S(h^\varepsilon(\sigma,\tau),\sigma,\tau/\mu^\ell,\mu)|.
	\end{align}
	Since \eqref{systempolartruncated} is assumed to have a hyperbolic limit cycle that is the graph of a function of $\sigma$, it follows that there is $m>0$ such that $|g(0,\sigma))|>m$ for all $(\sigma,\tau) \in \mathbb{R}^2$. Define
	\begin{align*}
		\mu_1:=\min \left\{\mu_0, \frac{m}{4L_0K_0}, \frac{m}{4K_s} \right\}.
	\end{align*}
	Then, from (\ref{ineq:Gtriangle}), it follows that 
	\[
		|G(h^\varepsilon(\sigma^\varepsilon(s),\tau(s)),\sigma^\varepsilon(s),\tau(s),\mu,\varepsilon)|\geq  \frac{m}{2} >0,
\]
	for all $s\in \mathbb{R}$, $q \in [0,2\pi]$, $\mu \in (0,\mu_1]$, and all $\varepsilon \in [0,\bar{\varepsilon}_{\mu})$, as we set out to prove.
	
	Note that
	\begin{align} \label{eq: d/ds ln|G|}
		\frac{d}{ds} \ln{|G(h^\varepsilon(\sigma^\varepsilon(s,q),\tau(s)),\sigma^\varepsilon(s),\tau(s),\mu,\varepsilon)|} &= \frac{1}{G} (G_\rho h^\varepsilon_\sigma (\sigma^\varepsilon)' + G_\rho h^\varepsilon_\tau +G_\sigma (\sigma^\varepsilon)' + G_\tau) 
		\nonumber \\
		&= G_\rho h^\varepsilon_\sigma + G_\sigma + \frac{G_\rho h^\varepsilon_\tau}{G} + \frac{G_\tau}{G}.
	\end{align}
	where the arguments omitted in the expressions on the right-hand side of the equation are $(h^\varepsilon(\sigma^\varepsilon(s,q),\tau(s)),\sigma^\varepsilon(s,q),\tau(s),\mu,\varepsilon)$ for the function $G$ and its derivatives, $(\sigma^\varepsilon(s,q),\tau(s))$ for $h^\varepsilon$ and its derivatives, and $s$ for $(\sigma^\varepsilon)'$. Henceforth, whenever the arguments of those functions are omitted, they are to be understood as stated. By integrating (\ref{eq: d/ds ln|G|}) and rearranging the terms, we get:
	\begin{align}\label{eq:intGrhohsigma+Gsigma}
		\int_0^s (G_\rho h^\varepsilon_\sigma + G_\sigma) \, dt = \ln &\left|\frac{G(h^\varepsilon(\sigma^\varepsilon(s,q),\tau(s)),\sigma^\varepsilon(s,q),\tau(s),\mu,\varepsilon)}{G(h^\varepsilon(\sigma^\varepsilon(0),\tau(0)),\sigma^\varepsilon(0),\tau(0),\mu,\varepsilon)} \right|  \nonumber \\
		&-\int_0^s \frac{G_\rho h^\varepsilon_\tau}{G} \, dt -\int_0^s \frac{G_\tau}{G} \, dt.
	\end{align}
	We will provide an upper bound for the absolute value of each of the terms appearing on the right-hand side of (\ref{eq:intGrhohsigma+Gsigma}), so that we obtain an estimate for the left-hand side.
	
	In order to estimate the last term on the right-hand side of (\ref{eq:intGrhohsigma+Gsigma}), recall the definition of $G$, found in (\ref{definition:G}). We observe that, if $S$ is the function present in that definition, we have
	\begin{align} \label{eq:d/ds (S/G)}
		\frac{d}{ds} \left( \frac{S}{G} \right) = \; &\frac{S_\rho h^\varepsilon_\sigma (\sigma^\varepsilon)' +S_\rho h^\varepsilon_\tau +S_\sigma (\sigma^\varepsilon)' +\frac{1}{\mu^\ell}S_\tau}{G} \nonumber \\
		&-\frac{S(G_\rho h^\varepsilon_\sigma (\sigma^\varepsilon)' +G_\rho h^\varepsilon_\tau +G_\sigma (\sigma^\varepsilon)' +\frac{\varepsilon}{\mu^\ell}S_\tau) }{G^2},
	\end{align}
	where the argument of the function $S$ and its derivatives are given by the expression $(h^\varepsilon(\sigma^\varepsilon(s,q),\tau(s)),\sigma^\varepsilon(s,q),\tau(s)/\mu^\ell,\mu)$. We remark that $S_\tau$ is a symbol representing the partial derivative of $S$ with respect to its third entry. Since $G_\tau = \frac{\varepsilon}{\mu^\ell} S_\tau$ and $(\sigma^\varepsilon)'=G$, rearranging (\ref{eq:d/ds (S/G)}), we find that
	\begin{align} \label{eq:Gtau/G}
		\frac{G_\tau}{G} = \frac{\varepsilon}{\mu^\ell} \left(\frac{S_\tau}{G} \right) = \; &\varepsilon  \frac{S(G_\rho h^\varepsilon_\sigma G +G_\rho h^\varepsilon_\tau +G_\sigma G ) }{G^2} + \frac{\varepsilon^2}{\mu^\ell}\frac{ S S_\tau}{G^2} \nonumber \\
		&- \varepsilon  \frac{S_\rho h^\varepsilon_\sigma G +S_\rho h^\varepsilon_\tau +S_\sigma G}{G} +\varepsilon \, \frac{d}{ds} \left( \frac{S}{G} \right).
	\end{align}
	By integrating (\ref{eq:Gtau/G}), we obtain
	\begin{align} \label{eq:intGtau/G}
		\int_0^s \frac{G_\tau}{G} \, dt =\; &\varepsilon \int_0^s \frac{S G_\rho h^\varepsilon_\sigma G + S G_\rho h^\varepsilon_\tau + S G_\sigma G}{G^2} \, dt \nonumber \\
		& - \varepsilon \int_0^s \frac{S_\rho h^\varepsilon_\sigma G +S_\rho h^\varepsilon_\tau +S_\sigma G}{G} \, dt \nonumber\\
		& +\varepsilon   \frac{ \,S(h^\varepsilon(\sigma^\varepsilon(s,q),\tau(s)),\sigma^\varepsilon(s,q),\frac{\tau(s)}{\mu^\ell},\mu)}{G(h^\varepsilon(\sigma^\varepsilon(s,q),\tau(s)),\sigma^\varepsilon(s,q),\tau(s),\mu,\varepsilon)} \nonumber \\
		&- \varepsilon  \frac{ \, S(h^\varepsilon(\sigma^\varepsilon(0,q),\tau(0)),\sigma^\varepsilon(0,q),\frac{\tau(0)}{\mu^\ell},\mu)}{G(h^\varepsilon(\sigma^\varepsilon(0,q),\tau(0)),\sigma^\varepsilon(0,q),\tau(0),\mu,\varepsilon)} \nonumber \\
		&+ \frac{\varepsilon^2}{\mu^\ell} \int_0^s \frac{ S S_\tau}{G^2} \, dt.
	\end{align}
	
	Observe that by Lemma \ref{lemmaestimatesh}, there is $\mu_h>0$ such that the set $\mathcal{H}:=\{h^\varepsilon : \varepsilon \in [0,\bar{\varepsilon}_{\mu}), \; \mu \in (0,\mu_h]\, \}$ is uniformly bounded in the $C^0$ norm.  The sets $\mathcal{H}_\sigma:=\{h_\sigma^\varepsilon : \varepsilon \in [0,\bar{\varepsilon}_{\mu}), \; \mu \in (0,\mu_h]\, \}$ and $\mathcal{H}_\tau:=\{h_\tau^\varepsilon : \varepsilon \in [0,\bar{\varepsilon}_{\mu}), \; \mu \in (0,\mu_h]\, \}$, defined similarly, are also uniformly bounded. Therefore, there is $M>0$ such that $|h^\varepsilon|_{C^0}<M$, $|h^\varepsilon_\sigma|_{C^0}<M$, and $|h^\varepsilon_\tau|_{C^0}<M$ for all $\mu \in (0,\mu_h]$ and all $\varepsilon \in [0,\bar{\varepsilon}_{\mu})$. Moreover, $S$, $G$, and their derivatives are bounded on $[-M,M] \times \mathbb{R} \times \mathbb{R} \times [0,\mu_h].$
	
	Define $\mu_3:=\min\{\mu_2,\mu_h\}$. Considering that $G$ is also bounded from below by $m/2$ for $\mu \in (0,\mu_3]$, the integrands of the first and second integrals on the right-hand side of (\ref{eq:intGtau/G}) must be bounded for all $\mu \in (0,\mu_3]$, $\varepsilon \in [0,\bar{\varepsilon}_{\mu})$, $q \in [0,2\pi]$, and all $t \in \mathbb{R}$, so that there is $K_1>0$ such that
\[
		\left| \varepsilon \int_0^{s} \frac{S G_\rho h^\varepsilon_\sigma G + S G_\rho h^\varepsilon_\tau + S G_\sigma G}{G^2} \, dt \right| \leq \varepsilon \int_0^{|s|} \left| \frac{S G_\rho h^\varepsilon_\sigma G + S G_\rho h^\varepsilon_\tau + S G_\sigma G}{G^2} \right| \, dt \leq K_1 \varepsilon {|s|},
\]
	and
\[
		\left|\varepsilon \int_0^{s} \frac{S_\rho h^\varepsilon_\sigma G +S_\rho h^\varepsilon_\tau +S_\sigma G}{G}  \, dt \right|\leq \varepsilon \int_0^{|s|} \left| \frac{S_\rho h^\varepsilon_\sigma G +S_\rho h^\varepsilon_\tau +S_\sigma G}{G} \right| \, dt \leq K_1 \varepsilon{|s|},
\]
	for all $s\in \mathbb{R}$, $q \in [0,2\pi]$, $\mu \in (0,\mu_3]$, and all $[0,\bar{\varepsilon}_{\mu})$. Moreover, for the same reason, there is $K_2>0$ such that 
\[
		\varepsilon \left|\frac{ \,S(h^\varepsilon(\sigma^\varepsilon(s,q),\tau(s)),\sigma^\varepsilon(s,q),\frac{\tau(s)}{\mu^\ell},\mu)}{G(h^\varepsilon(\sigma^\varepsilon(s,q),\tau(s)),\sigma^\varepsilon(s,q),\tau(s),\mu,\varepsilon)} \right| \leq K_2,
\]
	and 
\[
		\varepsilon \left|\frac{ \, S(h^\varepsilon(\sigma^\varepsilon(0,q),\tau(0)),\sigma^\varepsilon(0,q),\frac{\tau(0)}{\mu^\ell},\mu)}{G(h^\varepsilon(\sigma^\varepsilon(0,q),\tau(0)),\sigma^\varepsilon(0,q),\tau(0),\mu,\varepsilon)}   \right| \leq K_2,
\]
	for all $s\in \mathbb{R}$, $q \in [0,2\pi]$, $\mu \in (0,\mu_3]$, and all $\varepsilon \in [0,\bar{\varepsilon}_{\mu})$. Then, apply Lemma \ref{lemmaidentity} with $n=\ell$ and integrate to obtain
	\begin{align}  \label{eq:lastterm}
		\frac{\varepsilon^2}{\mu^\ell} \int_0^s \frac{S S_\tau}{G^2} dt =& \sum_{k=2}^\ell \frac{\varepsilon^k}{k} \int_0^s \frac{d}{dt} \left( \frac{S^k}{G^k}\right) dt - (S_\rho h_\sigma^\varepsilon G + S_\rho h^\varepsilon_\tau+S_\sigma G) \sum_{k=2}^{\ell}\varepsilon^k\int_0^s\frac{ S^{k-1}}{G^k} dt\nonumber \\
		&+ (G_\rho h^\varepsilon_\sigma G + G_\rho h^\varepsilon_\tau + G_\sigma G) \sum_{k=2}^{\ell} \varepsilon^k \int_0^s\frac{S^k}{G^{k+1}} dt+ \frac{\varepsilon^{\ell+1}}{\mu^\ell} \int_0^s\frac{S^\ell S_\tau}{G^{\ell+1}} dt.
	\end{align}
	Similarly as above, the first term on right-hand side is bounded by a positive constant $K_3$. The second and third terms are bounded by $ K_4
	\varepsilon {|s|}$, where $K_4$ is also a positive constant. For the last term, because $\varepsilon\leq\mu$, there is $K_5>0$ such that
\[
		\frac{\varepsilon^{\ell+1}}{\mu^\ell} \int_0^{|s|}\left|\frac{S^\ell  S_\tau}{G^{\ell+1}}\right| dt \leq K_5 \varepsilon |s|,
\]
	for all $s\in \mathbb{R}$, $q \in [0,2\pi]$, $\mu \in (0,\mu_3]$, and all $[0,\bar{\varepsilon}_{\mu})$. Therefore, it follows from the triangle inequality applied to (\ref{eq:lastterm}) that 
	\[
		\frac{\varepsilon^2}{\mu^\ell} \left|\int_0^s \frac{S S_\tau}{G^2}\, dt\right| \leq K_3 +2K_4 \varepsilon |s| +K_5 \varepsilon |s|,
\]
	for all $s\in \mathbb{R}$, $q \in [0,2\pi]$, $\mu \in (0,\mu_3]$, and all $\varepsilon \in [0,\bar{\varepsilon}_{\mu})$.
	
	We can finally estimate the last term on the right-hand side of (\ref{eq:intGrhohsigma+Gsigma}). Define the positive constant $K=\max\{2K_1,2K_2,2K_3,2K_4,2K_5\}$, and an application of the triangle inequality to (\ref{eq:Gtau/G}) ensures that 
\[
		\left| \int_0^s \frac{G_\tau}{G} \, dt \right| \leq 2K_1\varepsilon |s| +2K_2 +K_3 +2K_4\varepsilon |s| +K_5 \varepsilon |s| \leq  K +K \varepsilon |s|, 
\]
	for all $s\in \mathbb{R}$, $q \in [0,2\pi]$, $\mu \in (0,\mu_1]$, and all $[0,\bar{\varepsilon}_{\mu})$.
	We remark that $K$ is independent of the choice of $\varepsilon$ and $\mu$, provided that $\mu \in (0,\mu_3]$. 
	
	We proceed now to establishing an upper bound for the first term on the right-hand side of (\ref{eq:intGrhohsigma+Gsigma}). It suffices to observe that, by continuity of the function $t \mapsto \ln|t|$, there is $K'>0$ such that
\[
		\ln \left|\frac{G(h^\varepsilon(\sigma^\varepsilon(s),\tau(s)),\sigma^\varepsilon(s),\tau(s),\mu,\varepsilon)}{G(h^\varepsilon(\sigma^\varepsilon(0),\tau(0)),\sigma^\varepsilon(0),\tau(0),\mu,\varepsilon)} \right| \leq K',
\]
	for all $s\in \mathbb{R}$, $q \in [0,2\pi]$, $\mu \in (0,\mu_3]$, and all $[0,\bar{\varepsilon}_{\mu})$. 
	
	Finally, in order to estimate the second term on the right-hand side of (\ref{eq:intGrhohsigma+Gsigma}), we remark that, by Lemma \ref{lemmaestimatesh}, there is $C_\tau>0$ such that $|h^\varepsilon_\tau|_{C^0}\leq C_\tau \varepsilon$ for all $\mu \in (0,\mu_3]$ and all $\varepsilon \in [0,\bar{\varepsilon}_{\mu})$. Thus, there is $K''>0$ such that
\[
		\left| \int_0^s \frac{G_\rho h^\varepsilon_\tau}{G} \, dt \right| \leq \int_0^{|s|}\left| \frac{G_\rho C_\tau}{G} \right| \varepsilon \, dt \leq K'' \varepsilon |s|,
\]
	for all $s \in \mathbb{R}$, $q \in [0,2\pi]$, $\mu \in (0,\mu_3]$, and all $[0,\bar{\varepsilon}_{\mu})$.
	
	Having estimated all the terms we set out to, we can finally estimate the left-hand side of (\ref{eq:intGrhohsigma+Gsigma}). In order to do so, define $C=\max \{K, K', K''\}$, and an application of the triangle inequality to (\ref{eq:intGrhohsigma+Gsigma}) guarantees that
	\[
		\left| \int_0^s (G_\rho h^\varepsilon_\sigma + G_\sigma) \, dt \right| \leq C + C\varepsilon |s| ,
\]
	for all $s \in \mathbb{R}$, $q \in [0,2\pi]$, $\mu \in (0,\mu_3]$, and all $[0,\bar{\varepsilon}_{\mu})$. By definition of $y^\varepsilon(s,q)$, we conclude that
	\begin{align*}
		e^{-C} e^{-C \varepsilon |s|} \leq y^\varepsilon (s,q) \leq e^{C} e^{C \varepsilon |s|},
	\end{align*}
	for all $s\in \mathbb{R}$, $q \in [0,2\pi]$, $\mu \in (0,\mu_3]$, and all $[0,\bar{\varepsilon}_{\mu})$, thus proving the lemma.
\end{proof}

Finally, we present the last estimates regarding $h^\varepsilon$ necessary to our proof, namely upper bounds for the  second order partial derivatives of $h^\varepsilon$, results that mirror  \cite[Lemma 8.5]{MR1740943}. We remark that the  original argument relies on the lower bound for $y^\varepsilon(s,q)$ that is presented in \cite[Lemma 8.4]{MR1740943}, and that such bound is here replaced by the similar Lemma \ref{lemmayepsilon} in order to cover all the values $\ell$ may assume. It should be noted that the following bounds depend on the choice of $\mu>0$.

\begin{lemma}[{\cite[Lemma 8.5]{MR1740943}}]\ \label{lemmahsigmasigma}
	Suppose system (\ref{systempolartruncated}) has an attracting hyperbolic limit cycle $\Gamma=\{(0,\sigma):\sigma\in\s^1\}$. There is $\widetilde{\mu}>0$ such that, for each $\mu \in (0,\widetilde{\mu}]$, there is $C^\mu>0$ satisfying $|h^\varepsilon_{\sigma \sigma}|_{C^0} \leq C^\mu $ for all $\varepsilon \in [0,\bar{\varepsilon}_{\mu})$.
\end{lemma}
\begin{proof}
	We begin by setting $\mu_1$ and $C_1$ as being the quantities obtained in Lemma \ref{lemmaestimatesh}. Apply Lemma $\ref{lemmayepsilon}$ to obtain $\widetilde{\mu}>0$ and $C>0$ such that $e^{-C -C\varepsilon s} \leq y^\varepsilon(s,q) \leq e^{C+ C \varepsilon s}$ for all $s \geq 0$, $q \in [0,2\pi]$, $\mu \in (0,\widetilde{\mu}]$, and all $\varepsilon \in [0,\bar{\varepsilon}_{\mu})$. As in the proof of Lemma \ref{lemmaestimatesh}, we find $M_0>0$ such that $|\psi^0_{22}(s,q)|\geq M_0$ for all $s \in \mathbb{R}$ and all $q \in [0,2\pi]$. We remind the reader that $T>0$ is the period of the perturbation terms $R$ and $S$ with respect to their third entry, and that $\det \Psi^0(T,q)$ is, for any $q \in [0,2\pi]$, equal to the eigenvalue $\lambda_P<1$ of the derivative of the Poincaré map defined on any transversal section of $\Gamma$ evaluated at its fixed point.  We define
	\begin{align*}
		\mu_2:= \min\left\{\mu_1,\widetilde{\mu},\frac{-\ln\lambda_P}{2CT}\right\}.
	\end{align*}
	As was done in Lemma \ref{lemmaestimatesh}, we observe that function $I(t,q)$, defined in (\ref{def:I(t,q)}), has a maximum value $M_I>0$ on $[0,T] \times [0,2\pi]$. For any $t\geq0$, we may find $z \in \mathbb{N}$ and $r_t \in [0,T)$ such that $t= z T + r_t$. Hence, an application of Liouville's formula similar to (\ref{eq:liouvillesformula}) ensures that 
\[
		e^{C\mu_2 t} \det \Psi^0(t,q) \leq M_I e^{C \mu_2 T} \,(\det\Psi^0(T,q)e^{C\mu_2T})^z \leq M_I e^{C\mu_2 T} (\sqrt{\lambda_P})^z.
\]
	Thus, since $\lambda_P<1$, by taking $T_0>0$ sufficiently large, we can guarantee that
	\begin{align} \label{ineq:psi0T0}
		\det \Psi^0(t,q) \leq \frac{M_0^2}{4e^{2C+2CT\mu_2^{\ell+1}}} e^{-C\mu_2t} \leq \frac{M_0^2}{4} e^{-2C-2CT \mu_2^{\ell+1} -C\mu_2(T_0+1)},
	\end{align}
	for all $t \in [T_0,T_0+1)$ and all $q \in [0,2\pi]$. 
	
	Let $p \in [0,2\pi]$. For each $\mu \in (0,\mu_2]$, define $T_\mu \in [T_0,T_0+1)$ as in the proof of Lemma \ref{lemmaestimatesh}. Also, define s $\mapsto$ $\gamma^\varepsilon(s,q,w)=\big(h^\varepsilon(\sigma^\varepsilon(s,q,w),s+w),\sigma^\varepsilon(s,q,w),s+w\big)$ as the solution of $E^{\varepsilon,\mu}$ satisfying $\gamma^\varepsilon(0,q,w)=(h^\varepsilon(q,w),q,w)$, and let $q^{w,\varepsilon}$ be defined by $p=\sigma^\varepsilon(T_\mu,q^{w,\varepsilon},w)$. Finally, let $\Psi^\varepsilon(s,q,w)$ denote the principal fundamental matrix solution of $E^{\varepsilon,\mu}$ at $s=0$ along $\gamma^\varepsilon(s,q,w)$, and let $y^\varepsilon(s,q,w)$ be defined by
	\begin{align*}
		y^\varepsilon(s,q,w) := \exp \left( \int_0^s (G_\rho h^\varepsilon_\sigma + G_\sigma) \, dt\right),
	\end{align*}
	where the argument of the derivatives of $G$ is $\big(h^\varepsilon(\sigma^\varepsilon(t,q,w),t+w),\sigma^\varepsilon(t,q,w),t+w,\mu,\varepsilon\big)$ and the argument of $h_\sigma^\varepsilon$ is $(\sigma^\varepsilon(t,q,w),t+w)$. We remark that the change of variables $u(t)=t+w$ in the integral above guarantees that
\[
		y^\varepsilon(s,q,w) = \frac{y^\varepsilon(s+w,\widetilde{q})}{y^\varepsilon(w,\widetilde{q})},
\]
	where $\widetilde{q} = \sigma^\varepsilon(-w,q,w)$. Hence, an application of Lemma \ref{lemmayepsilon} ensures that 
\[
		|y^\varepsilon(s,q,w)| \geq e^{-2C-2C\varepsilon w} \, e^{-C \varepsilon s} \geq e^{-2C-2C T \mu_2^{\ell+1} } \, e^{-C\mu_2 s},
\]
	for all $s\geq0$, $\mu \in (0,\mu_2]$, $\varepsilon \in [0,\bar{\varepsilon}_{\mu})$, $q \in [0,2\pi]$, and all $w \in [0,T \mu^\ell]$.
	
	Observe that, by a similar argument as the one found in (\ref{def:v2epsilon}), for any $\mu \in (0,\mu_2]$ and any $\varepsilon \in [0,\bar{\varepsilon}_{\mu})$, the function
	\begin{align*}
		\begin{pmatrix}
			h^\varepsilon_\sigma(\sigma^\varepsilon(s,q,w),s+w) \, y^\varepsilon(s,q,w) \\
			y^\varepsilon(s,q,w) \\
			0
		\end{pmatrix} =
		\Psi^\varepsilon(s,q,w) \cdot 
		\begin{pmatrix}
			h_\sigma^\varepsilon(q,w) \\
			1 \\
			0
		\end{pmatrix}
	\end{align*}
	is a solution of the first variational equational of $E^{\varepsilon,\mu}$ along $\gamma^\varepsilon(s,q,w)$. Thus, since $h_\sigma$ is $T\mu^\ell$-periodic in its second argument, it follows that
	\begin{align*}
		h_\sigma^\varepsilon (p,w) = h_\sigma^\varepsilon (p,w+T_\mu) = \frac{\psi_{11}^\varepsilon(T_\mu,q^{w,\varepsilon},w) h_\sigma^\varepsilon(q^{w,\varepsilon},w)+\psi^\varepsilon_{12}(T_\mu,q^{w,\varepsilon},w)}{\psi_{21}(T_\mu,q^{w,\varepsilon},w) h_\sigma^\varepsilon(q^{w,\varepsilon},w)+\psi_{22}^\varepsilon(T_\mu,q^{w,\varepsilon},w) }.
	\end{align*}
	Differentiating both sides with respect to $p$, we obtain:
	\begin{align} \label{eq:hsigmasigmapw}
		h_{\sigma \sigma}^\varepsilon(p,w) &= \frac{1}{(\psi^\varepsilon_{21} h_\sigma^\varepsilon(q^{w,\varepsilon},w) + \psi^\varepsilon_{22})^2}\frac{\partial q^{w,\varepsilon}}{\partial p} \Big( \left(\psi_{21}^\varepsilon (\psi_{11}^\varepsilon)_q - \psi_{11}^\varepsilon (\psi_{21}^\varepsilon)_q\right) (h_\sigma^\varepsilon(q^{w,\varepsilon},w))^2 \nonumber \\ &+
		\left(\psi_{22}^\varepsilon(\psi_{11}^\varepsilon)_q - \psi_{12}^\varepsilon (\psi_{21}^\varepsilon)_q + \psi_{21}^\varepsilon(\psi_{12}^\varepsilon)_q - \psi_{11}^\varepsilon(\psi_{22}^\varepsilon)_q\right)h_\sigma^\varepsilon(q^{w,\varepsilon},w)\nonumber \\ &+(\psi_{11}^\varepsilon \psi_{22}^\varepsilon - \psi_{12}^\varepsilon \psi_{21}^\varepsilon )h_{\sigma \sigma}^\varepsilon(q^{w,\varepsilon},w) + \psi_{22}^\varepsilon(\psi_{12}^\varepsilon)_q - \psi_{12}^\varepsilon(\psi_{22}^\varepsilon)_q \Big),
	\end{align}
	where the subscript $\cdot_q$ indicates a derivative with respect to $q$, and the functions $\psi_{ij}^\varepsilon$ and their derivatives are evaluated at $(T_\mu,q^{w,\varepsilon},w)$. 
	
	By definition of $q^{w,\varepsilon}$, it follows that 
	\begin{align*}
		\frac{\partial q^{w,\varepsilon}}{\partial p} = \frac{1}{\sigma_q^\varepsilon(T_\mu,q^{w,\varepsilon},w)}.
	\end{align*}
	Moreover, since $\sigma_q(s,q,w)$ and $y^\varepsilon(s,q,w)$ are solutions of the same initial value problem, we obtain that
	\begin{align} \label{ineq:dqepsilon/dp}
		\left| \frac{\partial q^{w,\varepsilon}}{\partial p} \right| = \left| \frac{1}{\sigma_q^\varepsilon(T_\mu,q^{w,\varepsilon},w)} \right| = \left| \frac{1}{y^\varepsilon(T_\mu,q^{w,\varepsilon},w)} \right| \leq e^{2C+2C T \mu_2^{\ell+1} } \, e^{C\mu_2 T_\mu},
	\end{align}
	for all $\mu \in (0,\mu_2]$, $\varepsilon \in [0,\bar{\varepsilon}_{\mu})$, $p \in [0,2\pi]$, and all $w \in [0,T \mu^\ell]$. 
	
	Apply Lemma \ref{lemmatecnico} and obtain $r>0$ such that, if $t \in [T_0,T_0+1]$; $q,\xi \in \mathbb{R}$, and $D \in M_2(\mathbb{R})$ are such that $|\xi|<r$ and $|D-\Psi^0(t,q)|<r$, then
	\begin{align*}
		|d_{21} \xi +d_{22}| > \frac{2}{3} M_0, \qquad |\det D - \det \Psi^0(t,q)|< \frac{M_0^2}{8} e^{-2C-2CT \mu_2^{\ell+1} -C\mu_2(T_0+1)}.
	\end{align*}
	We remark that, since $F(\rho,\sigma,\tau,\mu,0)$ and $G(\rho,\sigma,\tau,\mu,0)$ do not depend on the choice of $\tau$, the first variational equations of $E^{\varepsilon,\mu}$ along $\gamma^0(s,q,w)$ and along $\gamma^0(s,q,0)=\gamma^0(s,q)$ are the same. Thus, $\Psi^0(s,q)=\Psi^0(s,q,w)$ for all $w \in \mathbb{R}$. 
	
	We proceed as in the proof of Lemma \ref{lemmaestimatesh} and observe that, since $T_\mu$ must be in the bounded set $[T_0,T_0+1)$, Lemma \ref{gronwalllemma} applied to $E^{\varepsilon,\mu}$ ensures that there are $C_2,C_3>0$ such that
	\begin{align} \label{ineq:gronwall3}
		|h^\varepsilon(q^{w,\varepsilon},w)| + |q^{w,\varepsilon}- q^{w,0}| \leq C_2 (|h^\varepsilon(p,w) |+\varepsilon)
	\end{align}
	and
	\begin{align}\label{ineq:gronwall4}
		|\Psi^\varepsilon(T_\mu,q^{w,\varepsilon},w) - \Psi^0(T_\mu,q^{w,0},w)| \leq C_3(	|h^\varepsilon(q^{w,\varepsilon},w)| + |q^{w,\varepsilon}- q^{w,0}|+\varepsilon),
	\end{align}
	for all $\mu \in (0,\mu_2]$, $\varepsilon \in [0,\bar{\varepsilon}_{\mu})$, $w \in [0,T \mu^\ell]$, and all $p \in [0,2\pi]$. But, by Lemma \ref{lemmaestimatesh}, we know that $|h^\varepsilon|_{C^0} \leq C_1 \varepsilon$ for all $\mu \in (0,\mu_1]$ and all $\varepsilon \in [0,\bar{\varepsilon}_{\mu})$. Hence, from (\ref{ineq:gronwall3}), we conclude that there is $C_4>0$ satisfying $|q^{w,\varepsilon}-q^{w,0}|<C_4\varepsilon$ for all $\mu \in (0,\mu_2]$, $\varepsilon \in [0,\bar{\varepsilon}_{\mu})$, $w \in [0,T \mu^\ell]$, and all $p \in [0,2\pi]$. Combining this with (\ref{ineq:gronwall4}) and the fact that $\Psi^0(T_\mu,q^{w,\varepsilon})=\Psi^0(T_\mu,q^{w,\varepsilon},w)$ for all $w \in \mathbb{R}$, it finally follows that $|\Psi^\varepsilon(T_\mu,q^{w,\varepsilon},w) - \Psi^0(T_\mu,q^0)| < C_3(C_1 + C_4 + 1)\varepsilon$ for all $\mu \in (0,\mu_2]$, $\varepsilon \in [0,\bar{\varepsilon}_{\mu})$, $w \in [0,T \mu^\ell]$, and all $p \in [0,2\pi]$. Therefore, by defining 
	\begin{align*}
		\mu_3 : = \min \left\{\mu_2, \frac{r}{2C_1}, \frac{r}{2C_3(C_1+ C_4+1)}\right\},
	\end{align*}
	we can ensure that $|\Psi^\varepsilon(T_\mu,q^{w,\varepsilon},w) - \Psi^0(T_\mu,q^0)| < r$ and $|h^\varepsilon|_{C^0}<r$ for all $\mu \in (0,\mu_3]$, $\varepsilon \in [0,\bar{\varepsilon}_{\mu})$, $w \in [0,T \mu^\ell]$, and all $p \in [0,2\pi]$. Thus, it follows that 
	\begin{align} \label{ineq:4/9M_0^2}
		|\psi_{21}^\varepsilon(T_\mu,q^{w,\varepsilon},w) h_\sigma^\varepsilon(q^{w,\varepsilon},w) + \psi_{22}^\varepsilon(T_\mu,q^{w,\varepsilon},w)|^2 > \frac{4}{9}M_0^2 
	\end{align}
	and 
	\begin{align} \label{ineq:psiepsilon-psi0(2)}
		|\det \Psi^\varepsilon(T_\mu,q^{w,\varepsilon},w) - \det \Psi^0(T_\mu,q^{w,\varepsilon})| <\frac{M_0^2}{8} e^{-2C-2CT \mu_2^{\ell+1} -C\mu_2(T_0+1)},
	\end{align}
	for all $\mu \in (0,\mu_2]$, $\varepsilon \in [0,\bar{\varepsilon}_{\mu})$, $w \in [0,T \mu^\ell]$, and all $p \in [0,2\pi]$.
	
	Considering (\ref{ineq:psi0T0}) and (\ref{ineq:psiepsilon-psi0(2)}), an application of the triangle inequality yields
	\begin{align}\label{ineq:detpsiepsilon}
		|\det \Psi^\varepsilon(T_\mu,q^{w,\varepsilon},w)| < \frac{3M_0^2}{8} e^{-2C-2CT \mu_2^{\ell+1} -C\mu_2(T_0+1)},
	\end{align}
	for all $\mu \in (0,\mu_2]$, $\varepsilon \in [0,\bar{\varepsilon}_{\mu})$, $w \in [0,T \mu^\ell]$, and all $p \in [0,2\pi]$. Furthermore, combining (\ref{ineq:dqepsilon/dp}), (\ref{ineq:4/9M_0^2}) and (\ref{ineq:detpsiepsilon}), we conclude that 
	\begin{align} \label{ineq:coeficientehsigmasigma}
		\frac{\left|\psi_{11}^\varepsilon(T_\mu,q^{w,\varepsilon},w)\psi_{22}^\varepsilon(T_\mu,q^{w,\varepsilon},w) - \psi_{12}^\varepsilon(T_\mu,q^{w,\varepsilon},w) \psi_{21}^\varepsilon(T_\mu,q^{w,\varepsilon},w) \right|}{\left|\psi^\varepsilon_{21}(T_\mu,q^{w,\varepsilon},w) h_\sigma^\varepsilon(q^\varepsilon,w) + \psi^\varepsilon_{22}(T_\mu,q^{w,\varepsilon},w)\right|^2}\left|\frac{\partial q^\varepsilon}{\partial p}\right| \leq \frac{27}{32},
	\end{align}
	for all $\mu \in (0,\mu_2]$, $\varepsilon \in [0,\bar{\varepsilon}_{\mu})$, $w \in [0,T \mu^\ell]$, and all $p \in [0,2\pi]$.
	
	For a given $\mu>0$, let $s \mapsto \varphi^\varepsilon(s,\rho,\sigma,\tau)$ be the solution of system $E^{\varepsilon,\mu}$ satisfying the equation $\varphi^\varepsilon(0,\rho,\sigma,\tau)=(\rho,\sigma,\tau)$. Define $\Phi^\varepsilon(s,\rho,\sigma,\tau)$ to be the principal fundamental matrix solution of $E^{\varepsilon,\mu}$ at $t=0$ along  $\varphi^\varepsilon(s,\rho,\sigma,\tau)$. By definition of $\Psi^\varepsilon$, it follows that $\Psi^\varepsilon(s,q,w) = \Phi^\varepsilon(s,h^\varepsilon(q,w),q,w)$. By differentiating this identity with respect to $q$, we have that
	\[
		\Psi^\varepsilon_q(s,q,w) = \, \Phi^\varepsilon_\rho(s,h^\varepsilon(q,w),q,w) \, h^\varepsilon_\sigma(q,w) +  \Phi^\varepsilon_\sigma(s,h^\varepsilon(q,w),q,w).
\]
	
	Henceforth in this proof, we consider $\mu \in (0,\mu_3]$ to be fixed. As a result, some of the bounds we will work with depend directly on the choice of $\mu$. A superscripted $\mu$ will indicate when that is the case. Then, $\Phi^\varepsilon$, $\Phi^\varepsilon_\rho$ and $\Phi^\varepsilon_\sigma$ are continuous $\varepsilon$-families of continuous functions that are furthermore periodic in their third and fourth argument. Moreover, the sets $\{h^\varepsilon: \varepsilon \in [0,\bar{\varepsilon}_{\mu})\}$ and $\{h^\varepsilon_\sigma: \varepsilon \in [0,\bar{\varepsilon}_{\mu})\}$ are uniformly bounded by a constant $M>0$. By continuity and periodicity, there is $C_5^\mu>0$ such that $|\Phi^\varepsilon(T_\mu,\rho,\sigma,\tau)|<C_5^\mu$, $|\Phi^\varepsilon_\rho(T_\mu,\rho,\sigma,\tau)|<C_5^\mu$, and $|\Phi^\varepsilon_\sigma(T_\mu,\rho,\sigma,\tau)|<C_5^\mu$ for all $\varepsilon \in [0,\bar{\varepsilon}_{\mu}]$ and all $(\rho,\sigma,\tau) \in [-M,M] \times \mathbb{R} \times \mathbb{R}$. Thus, it follows that
	\begin{align} \label{ineq:psiepsilontmuq}
		|\Psi^\varepsilon(T_\mu,q,w)| = |\Phi^\varepsilon(T_\mu,h^\varepsilon(q,w),q,w)| \leq C_5^\mu
	\end{align}
	and
	\begin{align} \label{ineq:psi_qepsilontmuq}
		|\Psi^\varepsilon_q(T_\mu,q,w)| \leq \, |\Phi^\varepsilon_\rho(T_\mu,h^\varepsilon(q,w),q,w)| \, |h^\varepsilon_\sigma(q,w)| +  |\Phi^\varepsilon_\sigma(T_\mu,h^\varepsilon(q,w),q,w)| \leq C_5^\mu(M+1),
	\end{align}
	for all $\varepsilon \in [0,\bar{\varepsilon}_{\mu})$, all $q \in [0,2\pi]$, and all $w \in[0,T \mu^\ell]$.
	
	Finally, to estimate $h_{\sigma \sigma}^\varepsilon$, we take the absolute value of (\ref{eq:hsigmasigmapw}), apply the triangle inequality on the right-hand side and take the supremum of the right-hand side over $q^{w,\varepsilon} \in [0,2\pi]$. Then, considering (\ref{ineq:coeficientehsigmasigma}), (\ref{ineq:psiepsilontmuq}), (\ref{ineq:psi_qepsilontmuq}), and the uniform boundedness of $h_\sigma^\varepsilon$, it follows that there is $C_6^\mu>0$ such that
	\[
		|h^\varepsilon_{\sigma \sigma}(p,w)| \leq \frac{27}{32}  \left(\sup_{q \in [0,2\pi]}|h^\varepsilon_{\sigma \sigma}(q,w)| \right) + C_6^\mu,
\]
	for all $\varepsilon \in [0,\bar{\varepsilon}_{\mu})$, all $w \in [0,T\mu^\ell]$, and all  $p \in [0,2\pi]$. By taking the supremum over $(p,w) \in [0,2\pi] \times [0,T \mu^\ell]$ and rearranging the terms, we conclude that 
	\[
		|h^\varepsilon_{\sigma \sigma}|_{C_0} \leq \frac{32}{5} C_6^\mu,
\]
	for all $\varepsilon \in [0,\bar{\varepsilon}_{\mu})$, finishing the proof.
\end{proof}

\begin{lemma} \label{lemmahsigmatauhtautau}
	Suppose system (\ref{systempolartruncated}) has an attracting hyperbolic limit cycle $\Gamma=\{(0,\sigma):\sigma\in\s^1\}$. There is $\widetilde{\mu}>0$ such that, for each $\mu \in (0,\widetilde{\mu}]$, there is $\widetilde{C}^\mu>0$ satisfying $|h^\varepsilon_{\sigma \tau}|_{C^0} \leq \widetilde{C}^\mu $ and $|h_{\tau \tau}^{\varepsilon}|_{C^0}\leq \widetilde{C}^\mu$ for all $\varepsilon \in [0,\bar{\varepsilon}_{\mu})$.
\end{lemma}
\begin{proof}
	Let $\mu_1>0$ be sufficiently small to satisfy Lemmas \ref{lemmaestimatesh} and \ref{lemmahsigmasigma}. Let $C_1$ be the constant obtained in Lemma \ref{lemmaestimatesh} and, for each $\mu \in (0,\mu_1]$, let $C^\mu$ be the constant obtained in Lemma \ref{lemmahsigmasigma}. As before, we define $H$ by $h^\varepsilon(\sigma,\tau) = \varepsilon H(\sigma,\tau,\varepsilon)$. For convenience, we restate (\ref{eq:invarianceH}):
	\begin{align*} 
		\big(g(h^\varepsilon(\sigma,\tau),\sigma) + \varepsilon S(h^\varepsilon(\sigma,\tau),\sigma,\tau/\mu^\ell,\mu) \big) H_\sigma(\sigma,\tau,\varepsilon) + H_\tau(\sigma,\tau,\varepsilon) &=\nonumber \\
		R(h^\varepsilon(\sigma,\tau),\sigma,\tau/\mu^\ell,\mu) &+ \frac{1}{\varepsilon} \big(f(h^\varepsilon(\sigma,\tau),\sigma) \big), 
	\end{align*}
	
	Define the following functions:
	\begin{align*}
		&A_1(\sigma,\tau,\mu,\varepsilon) : = g_\rho \, h_\sigma^\varepsilon  + g_\sigma + \varepsilon S_\rho \, h_\sigma^\varepsilon + \varepsilon S_\sigma, \\
		&A_2(\sigma,\tau,\mu,\varepsilon) : = g + \varepsilon S, \\
		&A_3(\sigma,\tau,\mu,\varepsilon) : = R_\rho \, h_\sigma^\varepsilon + R_\sigma, \\
		&A_4(\sigma,\tau,\mu,\varepsilon) := f_\rho \, h^\varepsilon_{\sigma} + f_\sigma,
	\end{align*}
	where $f$, $g$, and their derivatives are to be evaluated at  $(h^\varepsilon(\sigma,\tau),\sigma)$; $h^\varepsilon$ and its derivatives are to be evaluated at $(\sigma,\tau)$; and $R$, $S$, and their derivatives are to be evaluated at $(h^\varepsilon(\sigma,\tau),\sigma,\tau/\mu^\ell,\mu)$.
	By differentiating both sides of (\ref{eq:invarianceH}) with respect to $\sigma$, it follows that
	\begin{align} \label{eq:Hsigmatau}
		A_1 H_\sigma + A_2 H_{\sigma \sigma} + H_{\tau \sigma} = A_3 + \frac{1}{\varepsilon} A_4,
	\end{align}
	where $A_1$, $A_2$, $A_3$, and $A_4$ are to be evaluated at $(\sigma,\tau,\mu,\varepsilon)$; and $H$ and its  derivatives are to be evaluated at $(\sigma,\tau,\varepsilon)$.
	
	Since $R$, $S$, and their derivatives are continuous and periodic in the second and third arguments, it follows that they are bounded on the set $[- C_1 \mu_1 , C_1 \mu_1] \times \mathbb{R} \times \mathbb{R} \times [0,\mu_1]$. Similarly, $f$, $g$, and their derivatives are bounded on $[ C_1 \mu_1 , C_1 \mu_1] \times \mathbb{R}$. In combination with the boundedness of $h^\varepsilon$ and its derivatives provided in Lemma \ref{lemmaestimatesh}, we conclude that there is $C_2>0$ such that $|A_i(\sigma,\tau,\mu,\varepsilon)|<C_2$ for all $\sigma \in \mathbb{R}$, $\tau \in \mathbb{R}$, $\mu \in (0,\mu_1]$, $\varepsilon \in [0,\bar{\varepsilon}_{\mu})$, and $i \in \{1,2,3,4\}$.
	
	By Lemma \ref{lemmaestimatesh} and the definition of $H$, it follows that $H_\sigma(\sigma,\tau,\varepsilon)\leq C_1$ for all $\sigma \in \mathbb{R}$, $\tau \in \mathbb{R}$, $\mu \in (0,\mu_1]$, and all $\varepsilon \in [0,\bar{\varepsilon}_{\mu})$. Furthermore, by Lemma \ref{lemmahsigmasigma}, it follows that 
	\[
		|H_{\sigma \sigma}(\sigma,\tau,\varepsilon)| \leq \frac{C^\mu} {\varepsilon}.
	\]
	An application of the triangle inequality to (\ref{eq:Hsigmatau}) in combination with all the boundedness results stated above ensures that
	\[
		|H_{\tau \sigma}(\sigma, \tau, \varepsilon)| \leq C_2 \left( 1+ C_1+ \frac{1+ C^\mu}{\varepsilon} \right),
\]	for all $\sigma \in \mathbb{R}$, $\tau \in \mathbb{R}$, $\mu \in (0,\mu_1]$, and all $\varepsilon \in [0,\bar{\varepsilon}_{\mu})$. Thus, there is $K_1^\mu>0$, dependent on $\mu$, satisfying $\varepsilon |H_{\tau \sigma}(\sigma,\tau,\varepsilon)| \leq K_1^\mu$ for all $\sigma \in \mathbb{R}$, $\tau \in \mathbb{R}$, $\mu \in (0,\mu_1]$, and all $\varepsilon \in [0,\bar{\varepsilon}_{\mu})$. It follows that
	\[
		|h^\varepsilon_{\tau \sigma}|_{C^0} \leq K_1^\mu,
\]
	for all $\mu \in (0,\mu_1]$ and all $\varepsilon \in [0,\bar{\varepsilon}_{\mu})$.
	
	In order to estimate $|h_{\tau \tau}|_{C^0}$, define the following functions:
	\begin{align*}
		&B_1(\sigma,\tau,\mu,\varepsilon) : = g_\rho \, h_\tau^\varepsilon + \varepsilon S_\rho \, h_\tau^\varepsilon + \frac{\varepsilon}{\mu^\ell} S_\tau, \\
		&B_2(\sigma,\tau,\mu,\varepsilon) : = R_\rho \, h_\tau^\varepsilon , \\
		&B_3(\sigma,\tau,\mu,\varepsilon) := f_\rho \, h^\varepsilon_{\tau},
	\end{align*}
	where $f$, $g$, and their derivatives are to be evaluated at  $(h^\varepsilon(\sigma,\tau),\sigma)$; $h^\e$ and its derivatives are to be evaluated at $(\sigma,\tau)$; and $R$, $S$, and their derivatives are to be evaluated at $(h^\varepsilon(\sigma,\tau),\sigma,\tau/\mu^\ell,\mu)$.
	By differentiating both sides of (\ref{eq:invarianceH}) with respect to $\tau$, it follows that
\[
		B_1 H_\sigma + A_2 H_{\sigma \tau} + H_{\tau \tau} = B_2 + \frac{1}{\varepsilon} B_3,
\]
	where $B_1$, $A_2$, $B_2$, and $B_3$ are to be evaluated at $(\sigma,\tau,\mu,\varepsilon)$; and $H$ and its  derivatives are to be evaluated at $(\sigma,\tau,\varepsilon)$. Hence, Lemma \ref{lemmaestimatesh}, combined with an analysis of periodicity and continuity as done above, ensures that there are $C_3>0$ and $C_4>0$ such that 
\[
	\begin{aligned}
		&|B_1(\sigma,\tau,\mu,\varepsilon)| \leq C_3 + \frac{C_4}{\mu^\ell}, \\
		&|B_2(\sigma,\tau,\mu,\varepsilon)| \leq C_3 + \frac{C_4}{\mu^\ell},
	\end{aligned}
	\]
	and 
	\[
	\begin{aligned}
		|B_3(\sigma,\tau,\mu,\varepsilon)| \leq C_3 \leq C_3 + \frac{C_4}{\mu^\ell},
	\end{aligned}
	\]
	for all $\sigma \in \mathbb{R}$, $\tau \in \mathbb{R}$, $\mu \in (0,\mu_1]$, and all $\varepsilon \in [0,\bar{\varepsilon}_{\mu})$. Thus, defining 
	\begin{align*}
		C_5^\mu := C_3 + \frac{C_4}{\mu^\ell},
	\end{align*}
	it follows, from an application of the triangle inequality combined with the upper bounds that we have established for $B_1$, $A_2$, $B_2$, $B_3$, $H_{\sigma}$, and $H_{\sigma \tau} = H_{\tau \sigma}$, that 
	\[
		|H_{\tau \tau}(\sigma, \tau ,\varepsilon)| \leq C_1 C_5^\mu  + C_2 \frac{K_1^\mu}{\varepsilon} + C_5^\mu + \frac{C_5^\mu}{\varepsilon},
	\]
	for all $\sigma \in \mathbb{R}$, $\tau \in \mathbb{R}$, $\mu \in (0,\mu_1]$, and all $\varepsilon \in [0,\bar{\varepsilon}_{\mu})$.  Thus, there is $K_2^\mu>0$, dependent on $\mu$, satisfying $\varepsilon |H_{\tau \tau}(\sigma,\tau,\varepsilon)| \leq K_2^\mu$ for all $\sigma \in \mathbb{R}$, $\tau \in \mathbb{R}$, $\mu \in (0,\mu_1]$, and all $\varepsilon \in [0,\bar{\varepsilon}_{\mu})$. It finally follows that
	\[
		|h^\varepsilon_{\tau \tau}|_{C^0} \leq K_2^\mu,
\]
	for all $\mu \in (0,\mu_1]$ and all $\varepsilon \in [0,\bar{\varepsilon}_{\mu})$. The proof is finished by defining $\widetilde{C}^\mu := \max\{K_1^\mu, K_2^\mu\}$.
\end{proof}

\subsection{Proof of Proposition \ref{prop:main}}\label{sec:proofmain}

Having established the previous estimates, we can finally proceed to the proof of Proposition \ref{prop:main} itself. The proof is divided in two lemmas. In the first lemma, the existence of $h^{\bar{\varepsilon}_{\mu}}$, a $C^1$ function of the angular variables whose graph is an invariant manifold of system $E^{\bar{\varepsilon}_{\mu},\mu}$, is ensured for $\mu>0$ sufficiently small. In the second lemma, $M(\bar{\varepsilon}_{\mu},\mu)$ is shown to be $r$-normally hyperbolic for $\mu>0$ sufficiently small.

Throughout this section, it is assumed that the hypothesis in the statement of Proposition \ref{prop:main} holds, to wit: system (\ref{systempolartruncated}) has an attracting hyperbolic limit cycle $\Gamma=\{(0,\sigma):\sigma\in\s^1\}$. As before, for each $\e\in A^{\mu}$, $h^{\varepsilon}$ denotes a $C^r$ function of the angular variables $(\sigma,\tau)$ whose graph is a $r$-normally hyperbolic invariant manifold of system $E^{\varepsilon,\mu}$; $\gamma^\varepsilon(s,q):=(h^\varepsilon(\sigma^\varepsilon(s,q),\tau(s)),\sigma^\varepsilon(s,q),\tau(s))$ denotes the solution of system $E^{\varepsilon,\mu}$ starting at $(h^\varepsilon(q,0),q,0)$; and $\bar{\varepsilon}_{\mu}$ denotes the least upper bound of the set $A^\mu$ defined in Section \ref{sec:mc}.
\begin{lemma} \label{lemmal1}
	There is $\mu_a>0$ such that, if $\mu \in (0,\mu_a]$, then there is a $C^1$ function $h^{\bar{\varepsilon}_{\mu}}$ of the angular variables whose graph is an invariant manifold of system $E^{\bar{\varepsilon}_{\mu},\mu}$. 
\end{lemma}
\begin{proof}
	We let $\mu_a>0$ be less than all the $\widetilde{\mu}>0$ obtained in Lemmas \ref{lemmaestimatesh}, \ref{lemmahsigmasigma}, and \ref{lemmahsigmatauhtautau}. By hypothesis, for each $\varepsilon \in [0,\bar{\varepsilon}_{\mu})$, system $E^{\varepsilon,\mu}$ has a $r$-normally hyperbolic invariant manifold that is graph of a $C^r$ function $h^\varepsilon$ of the angular variables $\sigma$ and $\tau$. Since $r\geq2$, we can define the family $S:=\{h^\varepsilon: 0\leq\varepsilon<\bar{\varepsilon}_{\mu}\}$ in the space of $C^2$ functions of the angular variables. By Lemmas \ref{lemmaestimatesh}, \ref{lemmahsigmasigma}, and \ref{lemmahsigmatauhtautau}, this family is uniformly bounded in the $C^2$ norm. Hence, $S$ is uniformly bounded and equicontinuous in the $C^1$ norm. Thus, an increasing sequence $(\varepsilon_i)_{i \in \mathbb{N}}$ whose limit is $\bar{\varepsilon}_{\mu}$ gives rise to a sequence of $C^2$ functions $(h^{\varepsilon_i})_{i \in \mathbb{N}}$ to which we can apply Lemma \ref{lemmaC2normperiodic} and extract a subsequence that converges to a $C^1$ function $h^{\bar{\varepsilon}_{\mu}}$.
	The fact that the graph of $h^{\bar{\varepsilon}_{\mu}}$ is invariant for $E^{\bar{\varepsilon}_{\mu},\mu}$ is a direct consequence of the continuity of the flow with respect to the parameter $\varepsilon$.
\end{proof}

\begin{remark}\label{mual}
Observe that the value $\mu_a$ obtained in Lemma \ref{lemmal1} satisfies the inequality $T\,\mu_a^{\ell}<1$. Indeed, from the proof of Lemma \ref{lemmal1}, $\mu_a$ is less than the $\widetilde{\mu}>0$ obtained in Lemma \ref{lemmaestimatesh}, which is less than $\mu_1$ given by \eqref{mu1}. In particular, $T \mu_a^{\ell}<1/2^{\ell}<1$.
\end{remark}

\begin{lemma}
	There is $\mu_b \in (0,\mu_a]$ such that the invariant manifold found in Lemma \ref{lemmal1}, that is given as the graph of the function $h^{\bar{\varepsilon}_{\mu}}$, is $r$-normally hyperbolic and attracting. Furthermore, the function $h^{\bar{\varepsilon}_{\mu}}$ is $C^r$.
\end{lemma}
\begin{proof}
	Let $\mu_a>0$ be as given in Lemma \ref{lemmal1}, and $M(\bar{\varepsilon}_{\mu},\mu)$ be the invariant manifold of system $E^{\bar{\varepsilon}_{\mu},\mu}$ found in the same Lemma. Then, $M(\bar{\varepsilon}_{\mu},\mu)$ is the graph of the $C^1$ function $h^{\bar{\varepsilon}_{\mu}}$ of the angular variables. We remark that, by construction, $\mu_a$ is small enough to satisfy Lemmas \ref{lemmaestimatesh}, \ref{lemmahsigmasigma}, and \ref{lemmahsigmatauhtautau}. Moreover, by continuity, all the estimates found in those lemmas have their domain of validity extended to $\varepsilon \in [0,\bar{\varepsilon}_\mu]$.
	
	For each $\varepsilon \in [0, \bar{\varepsilon}_{\mu}]$, we define the vector fields 
	\begin{align*}
		\mathcal{X}_1^{\varepsilon}(\sigma,\tau) =
		\begin{pmatrix}
			F(h^{\varepsilon}(\sigma,\tau),\sigma,\tau,\mu,\varepsilon) \\
			G(h^{\varepsilon}(\sigma,\tau),\sigma,\tau,\mu,\varepsilon) \\
			1
		\end{pmatrix}\quad \text{and} \quad 
		\mathcal{X}_2^\varepsilon(\sigma,\tau) =
		\begin{pmatrix}
			h_\sigma^\varepsilon(\sigma,\tau) \\
			1 \\
			0
		\end{pmatrix},
	\end{align*}
	which are clearly tangent to the torus $M(\varepsilon,\mu)$. Furthermore, we remark that $v_1^\varepsilon(s,q)= \mathcal{X}_1(\sigma^\varepsilon(s,q),\tau(s))$ and $v_2^\varepsilon(s,q)=y^\varepsilon(s,q) \mathcal{X}_2(\sigma^\varepsilon(s,q),\tau(s))$ are solutions to the first variational equation associated to $E^{\varepsilon,\mu}$ along $\gamma^\varepsilon(s,q)$, that is, the differential system
	\begin{align} \label{firstvariationalsystem}
		v'  = Q^\varepsilon(s,q) \cdot v,\quad Q^\varepsilon(s,q): = \begin{pmatrix}
			F_\rho & F_\sigma & F_\tau \\
			G_\rho & G_\sigma & G_\tau \\
			0 & 0 & 0 
		\end{pmatrix},
	\end{align} 
	where the argument of each function in the matrix $Q^\varepsilon(s,q)$
	is 
	$$(h^\varepsilon(\sigma^\varepsilon(s,q),\tau(s)),\sigma^\varepsilon(s,q),\tau(s),\mu,\varepsilon).$$ As before, let us define $\Psi^\varepsilon(s,q)$ as the principal fundamental matrix solution of (\ref{firstvariationalsystem}) at $s=0$. Observe that, because of Lemma \ref{lemmal1}, $\Psi^\varepsilon$ is now well defined for $\varepsilon=\bar{\varepsilon}_{\mu}$. We also define $\eta^\varepsilon(\sigma,\tau):=R \cdot \mathcal{X}_2^\varepsilon(\sigma,\tau)$, where $R$ is the rotation matrix given by
\[
	\begin{aligned}
		R = 
		\begin{pmatrix}
			0 & -1 & 0 \\
			1 &0 &0 \\
			0 & 0 & 1
		\end{pmatrix}.
		\end{aligned}
\]
	Observe that, for any pair $(\sigma,\tau)$, the vectors $\eta^\varepsilon(\sigma,\tau)$ and $\mathcal{X}_2^\varepsilon(\sigma,\tau)$ are orthogonal and the set $\{\mathcal{X}_1^\varepsilon(\sigma,\tau),\mathcal{X}_2^\varepsilon(\sigma,\tau),\eta^\varepsilon(\sigma,\tau)\}$ is linearly independent.
	
	Since $v_2^\varepsilon(s,q)$ is a solution of (\ref{firstvariationalsystem}), it follows that 
	\[
		\Psi^\varepsilon(s,q) \cdot \mathcal{X}_2^\varepsilon(q,0) = y^\varepsilon(s,q) \mathcal{X}_2^\varepsilon(\sigma^\varepsilon(s,q),s),
\]	for all $s \in \mathbb{R}$ and all $q \in [0,2\pi]$. Furthermore, the invariance of the foliation $L=\{L_s\}_{s \in \mathbb{R}}$, where $L_s:=\{(\rho,\sigma,\tau): \tau=s \}$, ensures that we can find functions $a^\varepsilon(s,q)$ and $b^\varepsilon(s,q)$ such that 
	\begin{align} \label{eq:etaab}
			\Psi^\varepsilon(s,q) \cdot \eta^\varepsilon(q,0) =  y^\varepsilon(s,q) \big(a^\varepsilon(s,q) \mathcal{X}_2^\varepsilon(\sigma^\varepsilon(s,q),s) + b^\varepsilon(s,q) \eta^\varepsilon(\sigma^\varepsilon(s,q),s) \big).
	\end{align}
	In particular, by taking the inner product of (\ref{eq:etaab}) with $\eta^\varepsilon(\sigma^\varepsilon(s,q),s)$, it follows that
	\[
		b^\varepsilon(s,q) = \frac{\langle \Psi^\varepsilon(s,q) \cdot \eta^\varepsilon(q,0), \eta^\varepsilon(\sigma^\varepsilon(s,q),s) \rangle }{|\eta^\varepsilon(\sigma^\varepsilon(s,q),s)|^2}.
	\]
	
	Let $u_\alpha,u_\beta \in T_{(h^\varepsilon(\sigma,\tau),\sigma,\tau)}\R^3$. Then, there are unique $\alpha_i \in \mathbb{R}$ and $\beta_i \in \mathbb{R}$, $i\in\{1,2,3\}$, such that
	\[
	\begin{aligned}
		u_\alpha = \alpha_1 \mathcal{X}_1^\varepsilon(\sigma,\tau) + \alpha_2 \mathcal{X}_2^\varepsilon(\sigma,\tau) + \alpha_3 \eta^\varepsilon(\sigma,\tau), \nonumber \\
		u_\beta = \beta_1 \mathcal{X}_1^\varepsilon(\sigma,\tau) + \beta_2 \mathcal{X}_2^\varepsilon(\sigma,\tau) + \beta_3 \eta^\varepsilon(\sigma,\tau).
	\end{aligned}
	\]
	Define the metric $\langle \cdot, \cdot \rangle'$ on $T\mathbb{R}^3|_{M(\varepsilon,\mu)}$ by
\[
		\langle u_\alpha, u_\beta \rangle' = \alpha_1\beta_1 +\alpha_2\beta_2 +\alpha_3\beta_3.
\]
	Notice that, in this metric, the set $\{\mathcal{X}_1^\varepsilon(\sigma,\tau),\mathcal{X}_2^\varepsilon(\sigma,\tau),\eta^\varepsilon(\sigma,\tau)\}$ corresponds to an orthonormal basis of $T\mathbb{R}^3|_{M(\varepsilon,\mu)}$.
	Consider the splitting $T\mathbb{R}^3|_{M(\varepsilon,\mu)} := TM(\varepsilon,\mu) \oplus N(\varepsilon,\mu)$, where $N(\varepsilon,\mu)$ is the normal bundle of $M(\varepsilon,\mu)$ with respect to  $\langle \cdot, \cdot \rangle'$. Let $\pi_{\varepsilon,\mu}:T\mathbb{R}^3 \to N(\varepsilon,\mu)$ be the orthogonal projection on $N(\varepsilon,\mu)$ with respect to  $\langle \cdot, \cdot \rangle'$. Furthermore, let us denote by $\|\cdot\|$ the norm induced by $\langle \cdot, \cdot \rangle'$.
	
	Let us denote the flow of system $E^{\varepsilon,\mu}$ by $\phi^{\varepsilon,\mu}(t,x)$, where $\phi^{\varepsilon,\mu}(0,x)=x$. Following Fenichel in \cite{Fenichel1971}, define the operators $A^t_\varepsilon(q) = D\big(\phi^\varepsilon|_{M(\varepsilon,\mu)}\big)\big((h^\varepsilon(q,0),q,0),-t\big)$ on $TM(\varepsilon,\mu)$ and $B^t_\varepsilon(q)=\pi_{\varepsilon,\mu} \cdot  D\phi^{\varepsilon,\mu}\big((h^\varepsilon(\sigma^\varepsilon(-t,q),-t),\sigma^\varepsilon(-t,q),-t),t\big)$ on $N(\varepsilon,\mu)$ and the quantities
\[
	\begin{aligned}
		&\nu^\varepsilon(q):= \inf \left\{a>0 : \lim_{t \to \infty} \frac{\|B^t_\varepsilon(q)\|}{a^t} = 0\right\} = \limsup_{t \to \infty} \|B^t_\varepsilon(q)\|^{\frac{1}{t}}, \\
		&\kappa^\varepsilon(q):=\inf\left\{s>0: \lim_{t\to \infty}\|A^t_\varepsilon(q)\| \, \|B^t_\varepsilon(q)\|^s=0\right\}.
	\end{aligned}
\]
	We remark that, since the generalized Lyapunov type numbers are constant on orbits, in order to show that $M(\bar{\varepsilon}_{\mu},\mu)$ is $r$-normally hyperbolic, it suffices to show that $M(\bar{\varepsilon}_{\mu},\mu)$ is $C^r$, and that $\nu^{\bar{\varepsilon}_\mu}(q)<1$ and $\kappa^{\bar{\varepsilon}_\mu}(q)<1/r$ for all $q \in [0,2\pi]$.
	
	Let us begin with the following claim: $\nu^{\bar{\varepsilon}_\mu}(q)<1$ for all $q \in [0,2\pi]$. Observe that, since the unperturbed torus $M(0)$ is normally hyperbolic, Fenichel's Uniformity Lemma \cite{Fenichel1971} ensures that there are $C_0>0$ and $a<1$ such that 
	\begin{align} \label{ineq:c0a^t}
		\sup_{q \in [0,2\pi]} \|B^t_0(q)\|< C_0 \, a^t,
	\end{align}
	for all $t\geq0$. Therefore, there is $T_0>0$ such that
	\[
		\|B^t_0(q)\| <\frac{1}{4},
\]
	for all $t\geq T_0$ and all $q \in [0,2\pi]$. For each $\mu \in (0,\mu_a]$, define $n_\mu$ as the least positive integer satisfying $n_\mu T \mu^\ell \in [T_0,T_0+1)$, which exists because $T\mu^\ell_a<1$ by Remark \ref{mual}. Moreover, for each $\mu \in (0,\mu_a]$, let $T_\mu$ be defined as $T_\mu:=n_\mu T \mu^\ell$.
	
	Lemma \ref{gronwalllemma}, combined with Lemma \ref{lemmaestimatesh}, ensures that there is $C_1>0$ such that 
	\begin{align} \label{ineq:psiepsilonpsi0c1epsilon}
		|\Psi^\varepsilon(s,q) - \Psi^0(s,q)| < C_1 \varepsilon,
	\end{align}
	for all $s \in [T_0,T_0+1]$, $q \in [0,2\pi]$, $\mu \in (0,\mu_a]$, and all $\varepsilon \in [0,\bar{\varepsilon}_\mu]$. Moreover, by continuity, there is $C_2>0$ such that 
	\begin{align} \label{ineq:psi^0c2}
		|\Psi^0(s,q)|<C_2,
	\end{align}
	for all $s \in [T_0,T_0+1]$, $q \in [0,2\pi]$, and $\mu \in (0,\mu_a]$. Since $R$ is an isometry and applying Lemma \ref{lemmaestimatesh} once again, it follows that there is $C_3>0$ such that
	\begin{align} \label{ineq:etaepsilon-eta0}
		|\eta^\varepsilon(\sigma^\varepsilon(s,q),s) - \eta^0(\sigma^\varepsilon(s,q),s)| = |\mathcal{X}_2^\varepsilon(\sigma^\varepsilon(s,q),s) - \mathcal{X}_2^0(\sigma^\varepsilon(s,q),s)| < C_3 \varepsilon,
	\end{align}
	for all $s \in \mathbb{R}$, $q \in [0,2\pi]$, $\mu \in (0,\mu_a]$, and all $\varepsilon \in [0,\bar{\varepsilon}_\mu]$. 
	
	 The boundedness of $\eta^0$ over $(\sigma,\tau) \in \mathbb{R}^2$ combined with inequality (\ref{ineq:etaepsilon-eta0}) guarantees the existence of $C_4>0$ satisfying
	\begin{align} \label{ineq:nepsilonconstant}
		\frac{1}{C_4} < |\eta^0(\sigma,\tau)| - C_3\varepsilon< |\eta^\varepsilon(\sigma,\tau)| < |\eta^0(\sigma,\tau)| + C_3\varepsilon < C_4,
	\end{align}
	for all $\sigma \in \mathbb{R}$, $\tau \in \mathbb{R}$, $\mu \in (0,\mu_a]$, and all $\varepsilon \in [0,\bar{\varepsilon}_\mu]$. Taking squares of (\ref{ineq:nepsilonconstant}), we obtain $C_5>0$ such that
	\begin{align} \label{ineq:nepsilonsquared}
		 |\eta^0(\sigma,\tau)|^2 - C_5\varepsilon< |\eta^\varepsilon(\sigma,\tau)|^2 < |\eta^0(\sigma,\tau)|^2 + C_5\varepsilon,
	\end{align}
	for all $\sigma \in \mathbb{R}$, $\tau \in \mathbb{R}$, $\mu \in (0,\mu_a]$, and all $\varepsilon \in [0,\bar{\varepsilon}_\mu]$. Thus, combining (\ref{ineq:nepsilonconstant}) and (\ref{ineq:nepsilonsquared}), we obtain
	\begin{align} \label{ineq:1overeta}
		\left|\frac{1}{|\eta^\varepsilon(\sigma^\varepsilon(s,q),s)|^2} - \frac{1}{|\eta^0(\sigma^\varepsilon(s,q),s)|^2} \right| \leq \frac{C_5 \varepsilon}{C_4^4},
	\end{align}
	for all $s \in \mathbb{R}$, $q \in [0,2\pi]$, and $\mu \in (0,\mu_a]$. Moreover, by continuity, there is $C_6>0$ such that 
	\begin{align} \label{ineq:psi0eta0eta0}
		|\langle \Psi^0(s,q) \cdot \eta^0(q,0), \eta^0(\sigma^\varepsilon(s,q),s)\rangle | \leq C_6,
	\end{align}
	for all $s \in [T_0,T_0+1]$, $q \in [0,2\pi]$, and $\mu \in (0,\mu_a]$.
	
	Observe that $\eta^0$ is constant. This ensures that $\eta^0(\sigma^0(s,q),s) = \eta^0(\sigma^\varepsilon(s,q),s)$, so that
	\[
	\begin{aligned}
		|y^\varepsilon(s,q) b^\varepsilon(s,q) - &y^0(s,q)b^0(s,q)| = \nonumber \\ &\frac{\langle \Psi^\varepsilon(s,q) \cdot \eta^\varepsilon(q,0), \eta^\varepsilon(\sigma^\varepsilon(s,q),s) \rangle }{|\eta^\varepsilon(\sigma^\varepsilon(s,q),s)|^2} -\frac{\langle \Psi^0(s,q) \cdot \eta^0(q,0), \eta^0(\sigma^\varepsilon(s,q),s) \rangle }{|\eta^0(\sigma^\varepsilon(s,q),s)|^2}.
	\end{aligned}
	\]
	By applying the triangle inequality, we can ensure that 
	\begin{align}\label{difyep}
		|y^\varepsilon(s,q) &b^\varepsilon(s,q) - y^0(s,q)b^0(s,q)| \leq \nonumber \\
		&\frac{|\langle \Psi^\varepsilon(s,q) \cdot \eta^\varepsilon(q,0), \eta^\varepsilon(\sigma^\varepsilon(s,q),s) \rangle - \langle \Psi^0(s,q) \cdot \eta^0(q,0), \eta^0(\sigma^\varepsilon(s,q),s) \rangle|} {|\eta^\varepsilon(\sigma^\varepsilon(s,q),s)|^2}  \nonumber \\
		&+ \left|\frac{1}{|\eta^\varepsilon(\sigma^\varepsilon(s,q),s)|^2} - \frac{1}{|\eta^0(\sigma^\varepsilon(s,q),s)|^2} \right| 	|\langle \Psi^0(s,q) \cdot \eta^0(q,0), \eta^0(\sigma^\varepsilon(s,q),s)|,
	\end{align}
	for all $s \in \mathbb{R}$, $q \in [0,2\pi]$, $\mu \in (0,\mu_a]$, and $\varepsilon \in [0,\bar{\varepsilon}_\mu]$.
	
	On one hand, it follows from (\ref{ineq:1overeta}) and (\ref{ineq:psi0eta0eta0}) that 
	\begin{align} \label{ineq:secondmember}
		\left|\frac{1}{|\eta^\varepsilon(\sigma^\varepsilon(s,q),s)|^2} - \frac{1}{|\eta^0(\sigma^\varepsilon(s,q),s)|^2} \right| 	|\langle \Psi^0(s,q) \cdot \eta^0(q,0), \eta^0(\sigma^\varepsilon(s,q),s)| \leq \frac{C_5 C_6\varepsilon}{C_4^4},
	\end{align}
	for all $s \in [T_0,T_0+1]$, $q \in [0,2\pi]$, $\mu \in (0,\mu_a]$, and $\varepsilon \in [0,\bar{\varepsilon}_\mu]$. On the other hand, an application of the triangle inequality, combined with (\ref{ineq:psiepsilonpsi0c1epsilon}),(\ref{ineq:psi^0c2}), (\ref{ineq:etaepsilon-eta0}), and (\ref{ineq:nepsilonconstant}) ensures that there is $C_7>0$ such that
	\begin{align} \label{ineq:firstmember}
		\frac{|\langle \Psi^\varepsilon(s,q) \cdot \eta^\varepsilon(q,0), \eta^\varepsilon(\sigma^\varepsilon(s,q),s) \rangle - \langle \Psi^0(s,q) \cdot \eta^0(q,0), \eta^0(\sigma^\varepsilon(s,q),s) \rangle|} {|\eta^\varepsilon(\sigma^\varepsilon(s,q),s)|^2} \leq \nonumber \\
		\frac{|\Psi^\varepsilon(s,q) - \Psi^0(s,q)| |\eta^\varepsilon(q,0)|}{|\eta^\varepsilon(\sigma^\varepsilon(s,q),s)|} + \frac{|\Psi^0(s,q)| |\eta^\varepsilon(q,0) - \eta^0(q,0)|}{|\eta^\varepsilon(\sigma^\varepsilon(s,q),s)|}\quad \;  \nonumber \\+ \frac{|\Psi^0(s,q)| |\eta^0(q,0)| |\eta^\varepsilon(\sigma^\varepsilon(s,q),s) - \eta^0(\sigma^\varepsilon(s,q),s)| }{|\eta^\varepsilon(\sigma^\varepsilon(s,q),s)|^2} < C_7 \varepsilon,
	\end{align}
	for all $s \in[T_0,T_0+1]$, $q \in [0,2\pi]$, $\mu \in (0,\mu_a]$, and all $\varepsilon \in [0,\bar{\varepsilon}_\mu]$. Therefore, by substituting (\ref{ineq:secondmember}) and (\ref{ineq:firstmember}) into \eqref{difyep}, it follows that there is $C_8>0$ such that
	\begin{align} \label{ineq:y^epsilonb^epsilon-y^0b^0}
		|y^\varepsilon(s,q) &b^\varepsilon(s,q) - y^0(s,q)b^0(s,q)| < C_8 \varepsilon,
	\end{align}
	for all $s \in[T_0,T_0+1]$, $q \in [0,2\pi]$, $\mu \in (0,\mu_a]$, and all $\varepsilon \in [0,\bar{\varepsilon}_\mu]$.
	
	Observe that, by definition, we know that $\|B^t_\varepsilon(q)\| = |y^\varepsilon(t,\sigma^\varepsilon(-t,q)) b^\varepsilon(t,\sigma^\varepsilon(-t,q))|$. Thus, it follows from (\ref{ineq:y^epsilonb^epsilon-y^0b^0}) that $\|B^t_\varepsilon(q)-B^t_0(q)\|<C_8\varepsilon$ for all $t \in [T_0,T_0+1]$, $q \in [0,2\pi]$, $\mu \in (0,\mu_a]$, and all $\varepsilon \in [0,\bar{\varepsilon}_\mu]$. Therefore, by defining 
	\[
		\mu_1:=\min \left\{\mu_a,\frac{1}{4C_8}\right\},
\]
	it follows that 
\[
		\|B^t_\varepsilon(q)\| < \|B^t_0(q)\| + \frac{1}{4} <\frac{1}{2},
\]
	for all $s \in[T_0,T_0+1]$, $q \in [0,2\pi]$, $\mu \in (0,\mu_a]$, and all $\varepsilon \in [0,\bar{\varepsilon}_\mu]$. In particular, we have that
	\begin{align}\label{ineq:BTmuepsilon}
		\left\|B^{T_\mu}_{\bar{\varepsilon}_\mu}(q)\right\| < \frac{1}{2},
	\end{align}
	for all $q \in [0,2\pi]$ and all $\mu \in (0,\mu_1]$. Furthermore, for $\mu \in (0,\mu_1]$ fixed, the function $(s,q) \mapsto B^s_{\bar{\varepsilon}_\mu}(q)$ is continuous on the compact set $[0,T_0+1] \times [0,2\pi]$, so that there is $C_9>0$ satisfying 
	\begin{align} \label{ineq:Bsconstant}
		\left\|B^s_{\bar{\varepsilon}_\mu}(q)\right\| < C_9,
	\end{align}
	for all $s \in [0,T_0+1]$ and all $q \in [0,2\pi]$.
	
	For each $t\geq0$, let $k_\mu \in \mathbb{N}$ and $r_\mu \in [0,T_\mu)$ be such that $t = k_\mu T_\mu  +r_\mu$. Let $q_0 \in [0,2\pi]$ be given and define, for each $i \in \mathbb{N}$, the angle $q_i:=\sigma^\varepsilon(-iT_\mu,q_0)$. Then, by definition of $B_\varepsilon^t$, it follows that 
	\begin{align}\label{eq:Bproduct}
		B^{k_\mu T_\mu + r_\mu}_{\bar{\varepsilon}_\mu}(q_0) = B^{T_\mu}_{\bar{\varepsilon}_\mu}(q_0) \cdot B^{T_\mu}_{\bar{\varepsilon}_\mu}(q_1)  \cdots B^{T_\mu}_{\bar{\varepsilon}_\mu}(q_{k_\mu -1}) \cdot B^{r_\mu}_{\bar{\varepsilon}_\mu}(q_{k_\mu})
	\end{align}
	so that (\ref{ineq:BTmuepsilon}) and (\ref{ineq:Bsconstant}) guarantee that
	\[
		\left\|B^{k_\mu T_\mu + r_\mu}_{\bar{\varepsilon}_\mu}(q_0)\right\|^{\frac{1}{k_\mu T_\mu + r_\mu}} < C_9^{\frac{r_\mu}{k_\mu T_\mu + r_\mu}} \,  \left(\frac{1}{2}\right)^\frac{k_\mu}{k_\mu T_\mu + r_\mu},
\]
	for all $t\geq0$, $q_0 \in [0,2\pi]$, and $\mu \in (0,\mu_1]$. Therefore, we conclude that 
	\[
		\nu^{\bar{\varepsilon}_\mu}(q_0)=\limsup_{t \to \infty} \|B^t_{\bar{\varepsilon}_\mu}(q_0)\|^{\frac{1}{t}} \leq \limsup_{k_\mu \to \infty} \left\|B^{k_\mu T_\mu + r_\mu}_{\bar{\varepsilon}_\mu}(q_0)\right\|^{\frac{1}{k_\mu T_\mu + r_\mu}} \leq \left(\frac{1}{2}\right)^{\frac{1}{T_\mu}} < 1,
\]
	for all $q_0 \in [0,2\pi]$ and all $\mu \in (0,\mu_1]$. This concludes the proof of the claim.
	
	We proceed to proving the following claim: $\kappa^{\bar{\varepsilon}_\mu}(q)<1/r$ for all $q \in [0,2\pi]$. In order to do so, let $C$ be as given in Lemma \ref{lemmayepsilon}. We remind the reader that $r\geq2$ is the smoothness class of the functions appearing in $E^{\varepsilon,\mu}$. Define
	\[
		\mu_2 : = \min \left\{\mu_1, \left(\frac{-\ln a}{4C(r+1)}\right) \right\}.
\]
	By (\ref{ineq:c0a^t}), it follows that
	\[
		\sup_{q \in [0,2\pi]} e^{2C\mu_2 (r+1)(t+1)} \|B^t_0(q)\| < C_0 e^{2C\mu_2(r+1)} \left(e^{2C\mu_2(r+1)} a\right)^t < C_0 e^{2C\mu_2(r+1)} \left(\sqrt{a}\right)^t,
\]
	for all $t\geq0$. Since $\sqrt{a}<1$, there is $T_1>0$ such that 
	\begin{align} \label{ineq:supB0}
		\sup_{q \in [0,2\pi]} \|B^t_0(q)\| < \frac{1}{4} e^{-2C\mu_2(r+1)(t+1)} \leq \frac{1}{4} e^{-2C\mu_2(r+1)(T_1+1)},
	\end{align}
	for all $t\geq T_1$.
	
	For each $\mu \in (0,\mu_2]$, define $\tilde{n}_\mu \in \mathbb{N}$ as the least positive integer such that $\tilde{n}_\mu T \mu^\ell \in [T_1,T_1+1)$, which exists because $T\mu^\ell_a<1$ by Remark \ref{mual}. Define furthermore $\tilde{T}_\mu := \tilde{n}_\mu T \mu^\ell$. Proceeding as before, we can find $C_{10}>0$ such that 
	\begin{align}\label{ineq:BepsilonB0C10}
		\|B^t_\varepsilon(q)-B^t_0(q)\|<C_{10} \varepsilon,
	\end{align}
	for all $t \in [T_1,T_1+1]$, $q \in [0,2\pi]$, $\mu \in (0,\mu_2]$, and $\varepsilon \in [0,\bar{\varepsilon}_\mu]$.
	
	Define 
	\[
		\mu_3 : = \min \left\{\mu_2, \frac{1}{4C_{10}} e^{-2C\mu_2(r+1)(T_1+1)}\right\}.
\]
	Then, considering (\ref{ineq:supB0}) and (\ref{ineq:BepsilonB0C10}), it follows that 
	\[
		\|B^t_\varepsilon(q)\| < \|B^t_0(q)\| + C_{10} \mu_3 < \frac{1}{2} e^{-2C\mu_2(r+1)(T_1+1)},
	\]
	for all $t \in [T_1,T_1+1]$, $q \in [0,2\pi]$, $\mu \in (0,\mu_3]$, and $\varepsilon \in [0,\bar{\varepsilon}_\mu]$. In particular, we obtain that
	\[
		\left\|B^{\tilde{T}_\mu}_{\bar{\varepsilon}_\mu}(q)\right\| < \frac{1}{2} e^{-2C\mu_2(r+1)(T_1+1)},
	\]
	for all $q \in [0,2\pi]$ and all $\mu \in (0,\mu_3]$.
	
	For each $t\geq 0$, there are $\tilde{k}_\mu \in \mathbb{N}$ and $\tilde{r}_\mu \in [0,\tilde{T}_\mu)$ such that $t = \tilde{k}_\mu \tilde{T}_\mu + \tilde{r}_\mu$. Let $q_0 \in [0,2\pi]$ and define, as before, $q_i:=\sigma^\varepsilon(-iT_\mu,q_0)$. Reasoning as in (\ref{eq:Bproduct}), it follows that
	\begin{align} \label{ineq: Br+1}
		\left\|B^{\tilde{k}_\mu \tilde{T}_\mu + \tilde{r}_\mu}_{\bar{\varepsilon}_\mu}(q_0)\right\|^{\frac{1}{r+1}} \leq \left(\prod_{i=0}^{\tilde{k}_\mu-1}\left\|B^{\tilde{T}_\mu}_{\bar{\varepsilon}_\mu}(q_i)\right\| \right)^{\frac{1}{r+1}} \left\|B^{\tilde{r}_\mu}_{\bar{\varepsilon}_\mu}(q_{\tilde{k}_\mu})\right\|^{\frac{1}{r+1}},
	\end{align}
	for all $q_0 \in [0,2\pi]$ and all $\mu \in (0,\mu_3]$. By continuity of $(s,q) \mapsto B^s_{\bar{\varepsilon}_\mu}(q)$, for each $\mu \in (0,\mu_3]$ fixed, there is $C_{11}>0$ such that
	\[
		\left\|B^s_{\bar{\varepsilon}_\mu}(q)\right\|^{\frac{1}{r+1}} < C_{11},
\]
	for all $s \in [0,T_1+1]$, $q \in [0,2\pi]$, and $\mu \in (0,\mu_3]$. Thus, by (\ref{ineq: Br+1}), it follows that 
	\begin{align}\label{ineq:Btepsilonmu}
		\left\|B^t_{\bar{\varepsilon}_\mu}(q_0)\right\|^{\frac{1}{r+1}} \leq C_{11} \left(\frac{1}{2} e^{-2C\mu_2\tilde{k}_\mu(T_1+1)}\right),
	\end{align}
	for all $t\geq0$, $q_0 \in [0,2\pi]$, and all $\mu \in (0,\mu_3]$.
	
	By expressing $A^t_\varepsilon(q)$ in the corresponding basis of the family $\{\mathcal{X}^\varepsilon_1(\sigma,\tau), \mathcal{X}^\varepsilon_2(\sigma,\tau)\}$, we obtain
	\[\begin{aligned}
		A^t_\varepsilon(q) =
		\begin{bmatrix}
			1 & 0 \\
			0 & y^\varepsilon(-t,q)
		\end{bmatrix}.
	\end{aligned}\]
	Hence, it follows that 
	\[
		\|A^t_\varepsilon(q)\| \leq \max \{1, y^\varepsilon(-t,q)\},
\]
	for all $t\geq0$, $q \in [0,2\pi]$, $\mu \in (0,\mu_3]$, and $\varepsilon \in [0,\bar{\varepsilon}_\mu]$. Moreover, by Lemma \ref{lemmayepsilon}, we know that
	\[
		y^\varepsilon(-t,q) \leq e^C e^{C\varepsilon \tilde{k}_\mu \tilde{T}_\mu + \tilde{r}_\mu} \leq e^C e^{C \mu_3 \tilde{k}_\mu (T_1+1)} e^{C \mu_3 (T_1+1)},
\]
	for all $t\geq0$, $q \in [0,2\pi]$, $\mu \in (0,\mu_3]$, and $\varepsilon \in [0,\bar{\varepsilon}_\mu]$. Therefore, we conclude that 
	\begin{align} \label{ineq:Atepsilonmu}
		\left\|A^t_{\bar{\varepsilon}_\mu}(q)\right\| \leq e^{C+C\mu_3(T_1+1)} e^{C\mu_3 \tilde{k}_\mu(T_1+1)},
	\end{align}
	for all $t\geq0$, $q \in [0,2\pi]$, and all $\mu \in (0,\mu_3]$.
	
	Considering (\ref{ineq:Btepsilonmu}) and (\ref{ineq:Atepsilonmu}), it follows that
	\[
		\left\|A^t_{\bar{\varepsilon}_\mu}(q)\right\| \,	\left\|B^t_{\bar{\varepsilon}_\mu}(q)\right\|^{\frac{1}{r+1}} \leq C_{11} \, e^{C+C\mu_3(T_1+1)} e^{-C\mu_3 \tilde{k}_\mu(T_1+1)},
\]
	for all $t \geq 0$, $q \in [0,2\pi]$, and all $\mu \in (0,\mu_3]$. Hence, we have that
\[
		\lim_{t \to \infty} \left\|A^t_{\bar{\varepsilon}_\mu}(q)\right\| \,	\left\|B^t_{\bar{\varepsilon}_\mu}(q)\right\|^{\frac{1}{r+1}} = 0,
\]
	for all $q \in [0,2\pi]$ and all $\mu \in (0,\mu_3]$. It is thus proved that
\[
		\kappa^{\bar{\varepsilon}_\mu}(q) \leq \frac{1}{r+1}<\frac{1}{r},
\]
		for all $q \in [0,2\pi]$ and all $\mu \in (0,\mu_3]$. The second claim is thus proved.
		
		All that is left is proving that $M(\bar{\varepsilon}_\mu,\mu)$ is $C^r$. Observe that Theorem 6 in \cite{Fenichel1971} ensures that there is a bundle $I(\bar{\varepsilon}_\mu,\mu)$ that is transversal to $TM(\bar{\varepsilon}_\mu,\mu)$, homeomorphic to $N(\bar{\varepsilon}_\mu,\mu)$, and invariant under the derivative of the flow, $D\phi^{\bar{\varepsilon}_\mu,\mu}(t,\cdot)$ for all $t \in \mathbb{R}$. Then, the splitting $T\mathbb{R}^3=TM(\bar{\varepsilon}_\mu,\mu) \oplus I(\bar{\varepsilon}_\mu,\mu)$ is invariant under this derivative. Hence, we can apply Theorem 4.1.(f) in \cite{hirschpughshub} to guarantee that $M(\bar{\varepsilon}_\mu,\mu)$ is indeed $C^r$. It follows thus that $M(\bar{\varepsilon}_\mu,\mu)$ is $r$-normally hyperbolic.
\end{proof}

\section{Invariant tori in a family of jerk differential equations}\label{sec:app}

This section is dedicated to the proof of Proposition \ref{exprop}. In order to apply Theorem \ref{theoremA} for proving it, we need to write system \eqref{example} in the standard form \eqref{eq:e1}.  Recall that system \eqref{example} is given by
\[
\dddot x=-\dot x+\e^{N-1} P(x,\dot x,\ddot x)+\e^N Q(x,\dot x,\ddot x)+\e^{N+1} R(x,\dot x,\ddot x,\e),
\]
where $P$ and $Q$ satisfy:
\begin{itemize}
\item[{\bf H1.}] the following functions have vanishing average for every $z$ and $r$,
\[
\T\mapsto P(r\sin\T-z,r\cos\T,-r\sin\T)\,\,\text{ and }\,\, \T\mapsto P(r\sin\T-z,r\cos\T,-r\sin\T) \sin(\T)
\]
\item[{\bf H2.}] and
\[
\begin{aligned}
 Q(x,\dot x,\ddot x)=&\ddot x\big(-x-\ddot x+(\dot x^2+\ddot x^2-2)\big(1-(x+\ddot x)^2-\big(\dot x^2+\ddot x^2-2\big)^2\big)\big)\\&+2\ddot x^2(\dot x^2+\ddot x^2-2).
\end{aligned}
\]
\end{itemize}

Applying the change of variables  $(x,\dot x,\ddot x)=G(r,z,\T),$ with 
$$
G(r,z,\T)=\left(\sqrt{r} \sin \theta -z,\, \sqrt{r} \cos\T, \,-\sqrt{r} \sin\T\right),
$$
we obtain the following differential system  
\begin{align*}
\dot{r}=&-\e^{N-1}2\sqrt{r}P(\sqrt{r}\sin\T-z,\sqrt{r}\cos\T,-\sqrt{r}\sin\T)\sin\T\\
&-\e^{N}2\sqrt{r}Q(\sqrt{r}\sin\T-z,\sqrt{r}\cos\T,-\sqrt{r}\sin\T)\sin\T+\CO(\e^{N+1}),\\
\dot{z}=&-\e^{N-1}P(\sqrt{r}\sin\T-z,\sqrt{r}\cos\T,-\sqrt{r}\sin\T)\\
&-\e^{N}Q(\sqrt{r}\sin\T-z,\sqrt{r}\cos\T,-\sqrt{r}\sin\T)+\CO(\e^{N+1}),\\
\dot{\T}=&1-\dfrac{\e^{N-1}}{\sqrt{r}}P(\sqrt{r}\sin\T-z,\sqrt{r}\cos\T,-\sqrt{r}\sin\T)\cos\T\\
&-\dfrac{\e^{N}}{\sqrt{r}}Q(\sqrt{r}\sin\T-z,\sqrt{r}\cos\T,-\sqrt{r}\sin\T)\cos\T+\CO(\e^{N+1}).
\end{align*}
Since $\dot{\theta}=1+\CO\left(\e^{N-1}\right)$, we get that $\dot{\theta}>0$ for $\e$ sufficiently small. Thus, by taking $\theta$ as the new independent variable, the previous system becomes the following non-autonomous $2\pi-$periodic differential equation 
\begin{equation}\label{std1}
\dot \bx=\e^{N-1}F_{N-1}(\T,\bx)+\e^NF_N(\T,\bx)+\CO(\e^{N+1}),
\end{equation}
where the dot denotes the derivative with respect to the variable $\theta$, $\bx=(r,z)$, and the smooth functions $F_{j}(\T,\bx)=(F^2_{j}(\T,\bx),F^1_{j}(\T,\bx)),$ $j\in\{N-1,N\}$, are given by
\begin{equation}\label{fn}
\begin{aligned}
&F_{N-1}^{1}(\T, x)=-2\sqrt{r}\,P(\sqrt{r}\sin\T-z,\sqrt{r}\cos\T,-\sqrt{r}\sin\T)\sin\T,\\
&F_{N-1}^{2}(\T, x)=-P(\sqrt{r}\sin\T-z,\sqrt{r}\cos\T,-\sqrt{r}\sin\T),\\
&F_{N}^{1}(\T, x)=-2\sqrt{r}\,Q(\sqrt{r}\sin\T-z,\sqrt{r}\cos\T,-\sqrt{r}\sin\T)\sin\T,\\
&F_{N}^{2}(\T, x)=-Q(\sqrt{r}\sin\T-z,\sqrt{r}\cos\T,-\sqrt{r}\sin\T).
\end{aligned}
\end{equation}

Now, by using the formulae provided in \eqref{avfunc} and \eqref{yi}, we compute the averaged functions $\f_i,$ for $i=1,\dots,N$, of \eqref{std1}. First, since $N\geq3$, notice that $y_1(t,\bx)=\dots=y_{N-2}(t,\bx)=0$ and, then,
$$
y_{N-1}(t,\bx)=\int_0^t(N-1)!F_{N-1}(s,\bx)ds=(N-1)!\int_0^tF_{N-1}(s,\bx)ds.
$$
Thus, from conditions ${\bf H1}$ and ${\bf H2}$ and taking into account expression \eqref{fn}, we have that $\f_1=\cdots \f_{N-1}=0$. Finally, since $B_{1,1}(y_1(s,\bx))=B_{1,1}(0)=0$, we have that
\[
y_{N}(t,\bx)=\int_0^tN!F_{N}(s,\bx)+N!\partial_\bx F_{N-1}B_{1,1}(y_1(s,\bx))ds=N!\int_0^tF_{N}(s,\bx)ds.
\]
Therefore, from conditions ${\bf H1}$ and ${\bf H2}$  and taking into account the expressions in \eqref{fn}, one can see that
\[
\f_N(\bz)=2\pi \left(r (6 + z (1 + 2 z) - r (11 + (-6 + r) r + z^2)\,,\,r(2-r) \right),
\]
for every $N\geq 3$.
Hence, we conclude that the guiding system \eqref{guiding} for the differential equation \eqref{std1} is given by 
\begin{equation}\label{exemploguiding}
\begin{aligned}
\dot r=&r(6 + z (1 + 2 z) - r (11 + (-6 + r) r + z^2)),\\
\dot z=& r(2-r).
\end{aligned}
\end{equation}

Now, we claim that $\gamma=\mathbb{S}^1+(2,0)$ is an attracting hyperbolic limit cycle of the differential system \eqref{exemploguiding}. In fact, applying a time-rescaling in \eqref{exemploguiding}, by diving the right-hand side of the equations by $r>0$, and then applying translation $r=\rho+2$, we obtain
\begin{equation}\label{exemplofinal}
\begin{aligned}
\rho'=&\rho+z-\rho^3-\rho z^2,\\
z'=& - \rho,
\end{aligned}
\end{equation}
where the prime denotes the derivative with respect to the new time.
Consider the vector field $X(\rho, z)=\left(\rho+z-\rho^3-\rho z^2, -\rho \right)$ and the function $V(\rho,z)=\rho^2+z^2.$ 
Notice that
$$
\big\langle \nabla{V}(\rho,z), X(\rho, z) \big\rangle=-2 \rho ^2 \big(\rho ^2+z^2-1\big).
$$
Thus, $\mathbb{S}^1=V^{-1}(1)$ is an invariant set invariant set of the differential system \eqref{exemplofinal}, which has no equilibria on $\mathbb{S}^1.$ Consequently, from {\it Poincaré-Bendixson Theorem}, $\mathbb{S}^1$ is a closed orbit of \eqref{exemplofinal}. Denote by $\vf(t)=(v_1(t),v_2(t))$ the corresponding periodic solution of \eqref{exemplofinal} satisfying $v_1(\tau)^2+v_2(\tau)^2=1$. Let $\sigma_0$ and $\sigma$, $\sigma_0\subset\sigma$, be a transversal sections of $X$ at $p=(1,0)$ in such way that a Poincaré Map $\pi:\sigma_0\rightarrow\sigma$ of $X$ is well defined. From classical results (see, for instance, \cite[Theorem 1.23]{DLA}), we know that
\[ 
\begin{aligned}
\pi'(p)&=\exp\left[\int^T_0 div X(\vf(t))dt\right]\\
&=\exp\left[\int^T_0 \left(1-v_2(\tau)^2-3v_1(\tau)^2\right)dt\right]\\
&=\exp\left[-2\int^T_0 v_1(\tau)^2 dt\right]<1.
\end{aligned}
\]
Therefore, $\mathbb{S}^1$ is an attracting hyperbolic limit cycle of $X$ and, consequently,  $\gamma=\mathbb{S}^1+(2,0)$ is an attracting hyperbolic limit cycle of the guiding system \eqref{exemploguiding}. Then, the result follows by applying Theorem \ref{theoremA}.

\section*{Appendix: Technical lemmas}
\begin{lemma} \label{gronwalllemma}
	Let $U$ be an open subset of $\mathbb{R}^n$, $f:U \times \R^m \to \mathbb{R}^{n+m}$ be a $C^1$ vector field, and also $F_\mu:U \times \R^m \to \mathbb{R}^{n+m}$, $\mu \in(0,\mu_0)$, be a family of $C^1$ vector fields. Consider the differential equation
	\begin{align} \label{eq:lemmagronwall}
		(\dot x, \dot y) = f(x,y) +\varepsilon F_\mu(x,y),
	\end{align}
	where $\varepsilon \in [0,\varepsilon_0]$. Let $z(t,a,b,\varepsilon,\mu)=\big(x(t,a,b,\varepsilon,\mu),y(t,a,b,\varepsilon,\mu)\big)$ be the solution to (\ref{eq:lemmagronwall}) satisfying $z(0,a,b,\varepsilon,\mu) = (a,b)$, and let  $\Psi(t,a,b,\varepsilon,\mu)$ be the principal fundamental matrix solution at $t=0$ of the first variational equation 
	\[
		\Psi' = \big(Df(z(t,a,b,\varepsilon,\mu))+ \varepsilon DF_\mu(z(t,a,b,\varepsilon,\mu))\big) \Psi.
\]
	Assume that
	\begin{enumerate}[label=\alph*.]
		\item There are $t_*>0$, $Q \subset \R^n \times \R^m \times \R \times \R$, and a compact convex set $K\subset \mathbb{R}^n$ such that $x(t,a,b,\varepsilon,\mu) \in K$ for all $(t,a,b,\varepsilon,\mu) \in [0,t_*] \times Q$;
		\item There is $M_f>0$ such that $|f(x,y)|+|Df(x,y)|+|D^2f(x,y)|<M_f$ for all $(x,y) \in K \times \mathbb{R}^m$;
		\item There is $M_F>0$ such that $|F_\mu(x,y)|+|DF_\mu(x,y)|+|D^2F_\mu(x,y)|<M_F$ for all $(x,y) \in K \times \mathbb{R}^m$ and all $\mu \in \pi_5(Q)$, where $\pi_5$ denotes the canonical projection with respect to the fifth entry;
		\item There is $M_\Psi>0$ such that $|\Psi(t,a,b,0,\mu)|<M_\Psi$ for all $(t,a,b,0,\mu) \in [0,t_*] \times Q$.
	\end{enumerate}
	Then, the following holds:
	\begin{enumerate}[label=\roman*.]
		\item There is $K_1>0$ such that
		\begin{align*}
			|z(t,a_2,b_2,\varepsilon,\mu) - z(t,a_1,b_1,0,\mu)| \leq \left( |(a_2,b_2) - (a_1,b_1)| + \varepsilon|t| \right) K_1 \, e^{K_1 |t|},
		\end{align*} 
		for all $t \in [0,t_*]$, and all $(a_1,b_1, 0,\mu ), (a_2,b_2,\varepsilon,\mu) \in Q$;
		\item There is $K_2>0$ such that 
		\begin{align*}
			|\Psi(t,a_2,b_2,\varepsilon,\mu) - \Psi(t,a_1,b_1,0,\mu)| \leq \left( |(a_2,b_2) - (a_1,b_1)| + \varepsilon|t| \right) K_2 \, e^{K_2 |t|},
		\end{align*}
		for all $t \in [0,t_*]$, and all $(a_1,b_1, 0,\mu ), (a_2,b_2,\varepsilon,\mu) \in Q$.
	\end{enumerate}
\end{lemma}
\begin{proof}
	By definition, $z(t,a,b,\varepsilon,\mu)$ is a solution of $z'=f(z)+\varepsilon F_\mu(z)$ with $z(0,a,b,\varepsilon,\mu) = (a,b)$. Let $(t,a_2,b_2,\varepsilon,\mu), (t,a_1,b_1,0,\mu) \in [0,A] \times Q$. It follows from the triangle inequality that
\[
	\begin{aligned}
		|z(t,a_2,b_2,\varepsilon,\mu) - z(t,a_1,b_1,0,\mu)| \leq  |&(a_2,b_2) - (a_1,b_1)|  \\ + &\int_0^t |f(z(s,a_2,b_2,\varepsilon,\mu)) - f(z(s,a_1,b_1,0,\mu))| ds  \nonumber \\ + &\int_0^t \varepsilon |F_\mu (z(s,a_2,b_2,\varepsilon,\mu))| ds. \nonumber 
	\end{aligned}
	\]
	By the mean value inequality combined with hypotheses $(a)$ and $(b)$, we have
	\begin{align*}
		|f(z(s,a_2,b_2,\varepsilon,\mu)) - f(z(s,a_1,b_1,0,\mu))| \leq M_f  |z(s,a_2,b_2,\varepsilon,\mu) - z(s,a_1,b_1,0,\mu)|,
	\end{align*}
	for all $s \in [0,A]$. Moreover,
	hypotheses $(a)$ and $(c)$ ensure that 	\begin{align*}
		\int_0^t \varepsilon |F_\mu (z(s,a_2,b_2,\varepsilon,\mu))| ds \leq \varepsilon \, M_F |t|.
	\end{align*}
	Thus, it follows that 
	\begin{align*}
		|z(t,a_2,b_2,\varepsilon,\mu) - z(t,a_1,b_1,0,\mu)| \leq |&(a_2,b_2) - (a_1,b_1)|  \\ + &\int_0^t M_f  |z(s,a_2,b_2,\varepsilon,\mu) - z(s,a_1,b_1,0,\mu)| ds  \nonumber \\ + & \;\varepsilon M_F |t|.
	\end{align*}
	By Gronwall's inequality, we have
	\begin{align*}
		|z(t,a_2,b_2,\varepsilon,\mu) - z(t,a_1,b_1,0,\mu)| \leq \left(|(a_2,b_2)-(a_1,b_1)| +M_F \varepsilon |t| \right) e^{M_f|t|}.
	\end{align*}
	Defining $K_1 := \max\{1,M_f,M_F\}$ concludes the proof of $(i)$. For the second part, we observe that the triangle inequality guarantees that
	\begin{align*}
		|\Psi&(t,a_2,b_2,\varepsilon,\mu) - \Psi(t,a_1,b_1,0,\mu)| \leq  |\Psi(0,a_2,b_2,\varepsilon,\mu) - \Psi(0,a_1,b_1,0,\mu)|  \\ + &\int_0^t |D(f+\varepsilon F_\mu)(z(s,a_2,b_2,\varepsilon,\mu)) \Psi(s,a_2,b_2,\varepsilon,\mu) - Df(z(s,a_1,b_1,0,\mu))\Psi(s,a_1,b_1,0,\mu)| ds.  \nonumber 
	\end{align*}
	Observe that, since $\Psi$ is a principal fundamental matrix solution at $t=0$, it follows that $|\Psi(0,a_2,b_2,\varepsilon,\mu) - \Psi(0,a_1,b_1,0,\mu)|=0$. Furthermore, we can rewrite the integrand as
	\begin{align*}
		|D(f+\varepsilon F_\mu)&(z(s,a_2,b_2,\varepsilon,\mu)) (\Psi(s,a_2,b_2,\varepsilon,\mu)-\Psi(s,a_1,b_1,0,\mu))\\ +& (	Df(z(s,a_2,b_2,\varepsilon,\mu))-Df(z(s,a_1,b_1,0,\mu)))\Psi(s,a_1,b_1,0,\mu) \\ +& \varepsilon DF_\mu(z(s,a_2,b_2,\varepsilon,\mu)) \Psi(s,a_1,b_1,0,\mu)|.
	\end{align*}
	Since $\varepsilon \in [0,\varepsilon_0]$ and by hypotheses $(a)$, $(b)$, and $(c)$, the inequality
	\begin{align} \label{lemmatecineq1}
		|D(f+\varepsilon F_\mu)(z(s,a_2,b_2,\varepsilon,\mu))| \leq M_f + \varepsilon_0 M_F,
	\end{align}
	holds for all $s \in [0,t]$. We define $M_0:=M_f + \varepsilon_0 M_F$ to simplify notation. The mean value inequality applied to $Df$ ensures that
	\begin{align*}
		|(Df(z(s,a_2,b_2,\varepsilon,\mu))-&Df(z(s,a_1,b_1,0,\mu)))\Psi(s,a_1,b_1,0,\mu)|\\ &\leq M_f |z(s,a_2,b_2,\varepsilon,\mu) - z(s,a_1,b_1,0,\mu)| |\Psi(s,a_1,b_1,0,\mu)|.
	\end{align*}
	Considering hypothesis $(d)$ and that we have already proved $(i)$, it follows that
	\begin{align*}
		|(Df(z(s,a_2,b_2,\varepsilon,\mu))-&Df(z(s,a_1,b_1,0,\mu)))\Psi(s,a_1,b_1,0,\mu)| \\ &\leq M_f \, M_\Psi \left( |(a_2,b_2) - (a_1,b_1)| + \varepsilon|s| \right) K_1 \, e^{K_1 |s|},
	\end{align*}
	for all $s \in [0,t] \subset [0,t_*]$. Since $0\leq s \leq t \leq t_*$, there is $M_1>0$ such that 
	\begin{align} \label{lemmatecineq2}
		|(Df(z(s,a_2,b_2,\varepsilon,\mu))-&Df(z(s,a_1,b_1,0,\mu)))\Psi(s,a_1,b_1,0,\mu)| \nonumber \\ &\leq M_1 \left( |(a_2,b_2) - (a_1,b_1)| + \varepsilon \right).
	\end{align}
	Furthermore, hypotheses $(a)$, $(c)$, and $(d)$ combined guarantee that
	\begin{align} \label{lemmatecineq3}
		|\varepsilon DF_\mu(z(s,a_2,b_2,\varepsilon,\mu)) \Psi(s,a_1,b_1,0,\mu)| < \varepsilon M_F \, M_{\Psi},
	\end{align}
	for all $s \in [0,t]$. We define $M_2 = M_F \, M_\Psi$. By applying the triangle inequality to the integrand and considering (\ref{lemmatecineq1}), (\ref{lemmatecineq2}), and (\ref{lemmatecineq3}), it follows that
	\begin{align*}
		|\Psi(t,a_2,b_2,\varepsilon,\mu) - \Psi(t,a_1,b_1,0&,\mu)| \leq  \int_0^t M_0 |\Psi(s,a_2,b_2,\varepsilon,\mu)-\Psi(s,a_1,b_1,0,\mu)| ds \\ &+ \int_0^t M_1 \left( |(a_2,b_2) - (a_1,b_1)| + \varepsilon \right) ds + \int_0^t \varepsilon M_2 \, ds.
	\end{align*}
	Considering that $t \in [0,t_*]$, we have that
	\begin{align*}
		|\Psi(t,a_2,b_2,\varepsilon,\mu) - \Psi(t,a_1,b_1,0&,\mu)| \leq  \int_0^t M_0 |\Psi(s,a_2,b_2,\varepsilon,\mu)-\Psi(s,a_1,b_1,0,\mu)| ds \\ &+  M_1 \left( |(a_2,b_2) - (a_1,b_1)|t_* + \varepsilon|t| \right)  +  \varepsilon M_2 |t|.
	\end{align*}
	Hence, an application of Gronwall's inequality ensures that
	\begin{align*}
		|\Psi(t,a_2,b_2,\varepsilon,\mu) - \Psi(t,a_1,b_1,0,\mu)| \leq  \left(M_1 t_* |(a_2,b_2) - (a_1,b_1)|  + (M_1+M_2)\varepsilon|t| \right) e^{M_0|t|}.
	\end{align*}
	In order to finish the proof, it suffices to define $K_2:=\max\{M_0, \, M_1t_*,\, (M_1+M_2)\}$.
\end{proof}
\begin{remark}
	Lemma \ref{gronwalllemma} remains valid, by adjusting the hypotheses accordingly, if $t_*<0$ and $t \in [t_*,0]$. In order to prove it, it suffices to apply Lemma \ref{gronwalllemma} to the system $(\dot x, \dot y) = -f(x,y)-\varepsilon F_\mu(x,y)$. 
\end{remark}

\begin{lemma} \label{lemmatecnico}
	Let $M: \mathbb{R}^2 \to M_2(\mathbb{R})$ be a continuous function that is $2\pi$-periodic in the second argument, $\lambda:\mathbb{R} \to \mathbb{R}$ a continuous $2\pi$-periodic real function and $B \subset \mathbb{R}$ bounded. Suppose there is $M_0>0$ such that $|m_{21}(s,q) \lambda(q) + m_{22}(s,q)| \geq M_0 \,$ for all $(s,q) \in \mathbb{R} \times [0,2\pi]$. Then, for each $\varepsilon>0$, there is $r>0$ such that if $t \in B$; $q,\xi \in \mathbb{R}$; and $D \in M_2(\mathbb{R})$ are such that $|\xi-\lambda(q)|<r$, and $|D - M(t,q)| < r$, then the following inequalities hold:
	\begin{align*}
		|d_{21} \xi +d_{22}| > \frac{2}{3} M_0, \qquad |\det D - \det M(t,q)|< \varepsilon.
	\end{align*}
\end{lemma}
\begin{proof}
	Since $\lambda$ is continuous and periodic, the set $I_1:=\lambda(\mathbb{R})$ is compact. Similarly, $S_2:=M(\overline{B},\mathbb{R})$ is also compact. Define $S_1':=\{\xi \in \mathbb{R}: d(\xi,S_1)\leq 1\}$ and $S_2':=\{D \in M_2(\mathbb{R}): d(D,S_2)\leq 1\}$. It is clear that there are compact sets $\mathcal{K}_1 \subset \mathbb{R}$ and $\mathcal{K}_2 \subset M_2(\mathbb{R})$ satisfying $S_1' \subset \mathcal{K}_1$ and $S_2' \subset \mathcal{K}_2$. Observe that the function $T:\mathbb{R} \times M_2(\mathbb{R}) \to \mathbb{R}$ given by $T(\xi,D) = |d_{21}\xi+d_{22}|$ is continuous. Hence, $T$ is uniformly continuous on the compact set $\mathcal{K}_1 \times \mathcal{K}_2$. We may thus find $\delta_1>0$ such that, if $(\xi,D),(\xi^0,D^0) \in \mathcal{K}_1 \times \mathcal{K}_2$ and $|\xi - \xi^0|+ |D-D^0|< \delta_1$, then $|d_{21} \xi +d_{22} - (d^0_{21}\xi^0 +d^0_{22})| < M_0/3$. In particular, it follows that, for all $q \in \mathbb{R}$ and all $t \in \overline{B}$, if $|\xi-\lambda(q)| + |D-M(t,q)|<\delta_1$, then 
	\[
		|d_{21}\xi + d_{22} - (m_{21}(t,q)\lambda(q) + m_{22}(t,q))|< \frac{M_0}{3}.
\]
	An application of the reverse triangle inequality ensures that, for all $q \in \mathbb{R}$ and all $t \in \overline{B}$, if $|\xi-\lambda(q)| + |D-M(t,q)|<\delta_1$, then  
	\[
		|m_{21}(t,q)\lambda(q) + m_{22}(t,q)| -\frac{M_0}{3} < |d_{21}\xi + d_{22}|.
\]
	By hypothesis, this implies that, if $|\xi-\lambda(q)| + |D-m(t,q)|<\delta_1$, then
\[
		|d_{21}\xi + d_{22}| > \frac{2M_0}{3}.
\] 
	Since the determinant is also a continuous function, thus uniformly continuous on $\mathcal{K}_2$, for each $\varepsilon>0$ given, there is $\delta_2>0$ such that, if $D,D^0 \in \mathcal{K}_2$ and $|D-D^0|<\delta_2$, then $|\det D - \det D^0|<\varepsilon$. In particular, for all $q \in \mathbb{R}$ and all $t \in \overline{B}$, if $D \in \mathcal{K}_2$ and $|D - M(t,q)|<\delta_2$, then $|\det D - \det M(t,q)|<\varepsilon$.
	
	Finally we define
\[
		r : = \frac{1}{2} \min\{1,\delta_1,\delta_2\}.
\]
	Let $\xi \in \mathbb{R}$ and $D \in M_2(\mathbb{R})$. Observe that, if there are $t \in B$, $q \in \mathbb{R}$ such that $|\xi - \lambda(q)|< r < 1$ and $|D-M(t,q)|<r<1$, then $\xi \in S_1' \subset \mathcal{K}_1$ and $D \in S_2' \subset \mathcal{K}_2$. Moreover, in this case, it follows that $|\xi - \lambda(q)| + |D- M(t,q)| <\delta_1$ and $|D-M(t,q)|<\delta_2$, so that 
	\begin{align*}
		|d_{21} \xi +d_{22}| > \frac{2}{3} M_0, \qquad |\det D - \det M(t,q)|< \varepsilon,
	\end{align*}
	as we wished to prove.
\end{proof}

\begin{lemma} \label{lemmaC2normbounded}
	Let $D:= (a,b) \times (c,d) \subset \mathbb{R}^2$ and $\mathcal{F}:=\{f_n\}_{n \in \mathbb{N}}$ a sequence of continuous real valued functions defined on $\overline{D}$. If the family $\mathcal{F}_D:=\{f_n|_D : n \in \mathbb{N}\}$ is in $C^2(D)$ and is uniformly bounded in the $C^2$ norm, then there is a subsequence of $(f_n)_{n \in \mathbb{N}}$ that converges uniformly to a function $f_*:\overline{D} \to \mathbb{R}$ satisfying $f_*|_D \in C^1(D)$.
\end{lemma}
\begin{proof}
	Let $D$ and $\mathcal{F}=\{f_n\}_{n \in \mathbb{N}}$ be as stated. Since the family $\{f_n|_D\}$ is uniformly bounded in the $C^2$-norm, it follows that there is $M>0$ such that $\|f_n|_D\|_{C^2}\leq M$ for all $n \in \mathbb{N}$. In particular, $\|f_n|_D\|_{C^0}\leq M$ for all $n \in \mathbb{N}$. Thus, since each $f_n$ is continuous on $\overline{D}$, it follows that $\mathcal{F}=\{f_n\}_{n \in \mathbb{N}}$ is uniformly bounded by $M$ in the $C^0$ norm.
	
	We shall prove that $\mathcal{F}$ is also uniformly equicontinuous. Let $\varepsilon>0$ be given. We must show that there is $\delta>0$ such that, for all $f \in \mathcal{F}$ and all $x_1,x_2 \in \overline{D}$, if $|x_1-x_2| < \delta$, then $|f(x_1)-f(x_2)|< \varepsilon$. In order to do so, take $\delta =\varepsilon/(2M)$. Then, since $\|f'(x)\|\leq \|f|_D\|_{C^2} \leq M $ for all $x \in D$ and all $f \in \mathcal{F}$, by the mean value inequality, it follows that
	\begin{align*}
		|f(x_1)-f(x_2)| \leq M |x_1-x_2| \leq M \frac{\varepsilon}{2M} < \varepsilon,
	\end{align*}
	for all $f \in \mathcal{F}$ and all $x_1,x_2 \in \overline{D}$. Thus, $\mathcal{F}$ is uniformly equicontinuous.
	
	Considering that $\mathcal{F}$ is uniformly bounded and uniformly equicontinuous, we may apply the Arzelà-Ascoli theorem to obtain a subsequence $(f_{n_k})_{k \in \mathbb{N}}$ that converges uniformly to $f_* \in C^0(\overline{D})$. 
	
	It remains to be shown that $f_*|_D \in C^1(D)$. Define $g_k:= \partial_1 f_{n_k}|_D$, where $\partial_1$ denotes the partial derivative with respect to the first entry. Let $\mathcal{G}:=\{g_k : k \in \mathbb{N}\}$. Observe that $\mathcal{G}$ is uniformly bounded in the $C^0$ norm, because $\|g_k\|_{C^0} \leq \|f'_{n_k}|_D\|_{C^0}\leq \|f_{n_k}|_D\|_{C^2}\leq M$. Furthermore, $\|g'_k\|_{C^0}= \|(\partial_1 f_{n_k}|_D)'\|_{C^0} \leq \|(f_{n_k}|_D)''\|_{C^0} \leq \|f_{n_k}|_D\|_{C^2}\leq M$, so that the derivative of each $g_k$ is bounded by $M$. Thus, an application of the mean value inequality, as done above with $\mathcal{F}$, ensures that the family $\mathcal{G}$ is also uniformly equicontinuous. Once again, we apply the Arzelà-Ascoli theorem to obtain a subsequence $(g_{k_l})_{l \in \mathbb{N}}$ that converges uniformly to $g_* \in C^0(D)$. 
	
	Now, define $h_j:=\partial_2 f_{n_{k_j}}|_D$. Proceeding as above, we obtain a subsequence $(h_{j_i})_{i \in \mathbb{N}}$ that converges uniformly to $h_* \in C^0(D)$. In order to simplify notation, define a new set of indices $n_i:=n_{k_{j_i}}$. Then, it follows that
	\begin{align*}
		\lim_{i\to \infty} f_{n_i} = f_*, \qquad \lim_{i\to \infty} g_{n_i} = g_*, \quad \text{and} \quad \lim_{i\to \infty} h_{n_i} = h_*,
	\end{align*}
	where all the limits are uniform limits.
	
	We shall prove that $\partial_1 f_*|_D = g_*$. In fact, let $(s,t) \in D$. By the fundamental theorem of calculus, we have that 
	\begin{align*}
		f_{n_i}(s,t) = f_{n_i}(a,t) + \int_a^s (\partial_1 f_{n_i}|_D(\sigma,t)) d\sigma = f_{n_i}(a,t) + \int_a^s g_{n_i}(\sigma,t) d\sigma.
	\end{align*}
	Taking the limit as $i \to \infty$, it follows that 
	\begin{align*}
		f_*(s,t) = f_*(a,t) + \int_a^s g_*(\sigma,t) d\sigma.
	\end{align*}
	By taking the partial derivative with respect to $s$, it follows that $\partial_1 f_*(s,t) = g_*(s,t)$ for all $(s,t) \in D$, that is, $\partial_1 f_*|_D = g_*$. A similar argument shows that $\partial_2 f_*|_D = h_*$.
	
	Finally, since $g_*$ and $h_*$ are continuous on $D$, it follows that $f_*$ has continuous partial derivatives on $D$. Therefore, $f_* \in C^1(D)$.	
\end{proof}
\begin{lemma} \label{lemmaC2normperiodic}
	Let $\mathcal{F}:=\{f_n:\mathbb{R}^2 \to \mathbb{R}: n \in \mathbb{N}\}$ be a sequence of continuous real valued functions, each $T_1$-periodic in the first entry and $T_2$-periodic in the second entry. Suppose that the family $\mathcal{F}$ is in $C^2(\mathbb{R}^2)$ and is uniformly bounded in the $C^2$ norm. Then, there is a subsequence of $(f_n)_{n \in \mathbb{N}}$ that converges uniformly to a function $f_* \in C^1(\mathbb{R}^2)$. 
\end{lemma}
\begin{proof}
	Let $D_0:=(0,T_1) \times (0,T_2)$. Then, the family $\mathcal{F}_0:=\{f_n|_{\overline{D}_0}\}_{n \in \mathbb{N}}$ satisifies the hypotheses of Lemma \ref{lemmaC2normbounded}. Thus, there are $f^0_* \in C^0(\overline{D}_0)$ and a subsequence $(f_{n(k)})_{k \in \mathbb{N}}$ satisfying $f_{n(k)} \to f^0_*$ as $k\to \infty$ and $f^0_*|_{D^0} \in C^1(D^0)$. Similarly, define $D_1:=(T_1/2,3T_1/2) \times (0,T_2)$. The same Lemma applied to $\mathcal{F}_1:=\{f_{n(k)}|_{\overline{D}_1}\}_{k \in \mathbb{N}}$ ensures that there are $f^1_* \in C^0(\overline{D}_1)$ and a subsequence $(f_{n(k(j))})_{j \in \mathbb{N}}$ such that $f_{n(k(j))} \to f^1_*$ as $j \to \infty$ and $f_*^1|_{D^1}  \in C^1(D_1)$. Similarly, defining $D_2:=(0,T_1) \times (T_2/2,3T_2/2)$ and applying the Lemma to $\mathcal{F}_2:=\{f_{n(k(j))}|_{\overline{D}_2}\}_{j \in \mathbb{N}}$  guarantees the existence of $f^2_* \in C^0(\overline{D}_2)$ and a further subsequence $(f_{n_k(j(i))})_{i \in \mathbb{N}}$ such that $f_{n(k(j(i)))} \to f^2_*$ as $i \to \infty$ and $f^2_*|_{D^2}  \in C^1(D_2)$.
	
	Observe that $f^1_*(x,y)=f_*^0(x,y)$ for all $(x,y) \in \overline{D}_0 \cap \overline{D}_1$, because
	\begin{align*}
		f_*^1(x,y) = \lim_{j\to \infty} f_{n(k(j))}(x,y) = \lim_{k \to \infty} f_{n(k)}(x,y) = f^0_*(x,y).
	\end{align*}
	Similarly, $f^2_*(x,y)=f_*^0(x,y)$ for all $(x,y) \in \overline{D}_0 \cap \overline{D}_2$. Moreover, $f^1_*(x+T_1,y) = f_*^0(x,y)$ for all $(x,y) \in \overline{D}_0 \setminus \overline{D}_1$, because
	\begin{align*}
		f_*^1(x+T_1,y) = \lim_{j\to \infty} f_{n(k(j))}(x+T_1,y) =\lim_{j\to \infty} f_{n(k(j))}(x,y) = \lim_{k \to \infty} f_{n(k)}(x,y) = f^0_*(x,y).
	\end{align*}
	Similarly, $f^2_*(x,y+T_2)=f_*^0(x,y)$ for all $(x,y) \in \overline{D}_0 \setminus \overline{D}_2$.
	
	Define $g^0:\mathbb{R}^2 \to \mathbb{R}$, $g^1:\mathbb{R}^2 \to \mathbb{R}$, and $g^2:\mathbb{R}^2 \to \mathbb{R}$ by
	\begin{align*}
		&g^0(x,y) = f_*^0(x \mod T_1, y \mod T_2), \\
		&g^1(x,y) = f_*^1\left(\frac{T_1}{2} + \left(x+\frac{T_1}{2} \mod T_1 \right), y \mod T_2 \right), \\
		&g^2(x,y) = f_*^2\left(x \mod T_1, \frac{T_2}{2} + \left(y+\frac{T_2}{2} \mod T_2 \right) \right).
	\end{align*}
	It is clear that $g^0$, $g^1$, and $g^2$ are $T_1$-periodic in the first variable and $T_2$-periodic in the second variable. Moreover, if $S_0:=\{(x,y) \in \mathbb{R}^2 : \exists k \in \mathbb{Z} \; \text{such that} \; x=kT_1 \; \text{or} \; y=kT_2\}$ and $Q^0 := \mathbb{R}^2 \setminus S_0$, then $g^0|_{Q_0} \in C^1(Q_0)$. For the same reason, if
	\begin{align*}
		&S_1:=\left\{(x,y) \in \mathbb{R}^2 : \exists k \in \mathbb{Z} \; \text{such that} \; x=\frac{T_1}{2}+ kT_1 \; \text{or} \; y=kT_2\right\}, \\
		&S_2:=\left\{(x,y) \in \mathbb{R}^2 : \exists k \in \mathbb{Z} \; \text{such that} \; x= kT_1 \; \text{or} \; y= \frac{T_2}{2}+kT_2\right\},
	\end{align*}
	and $Q_i:=\mathbb{R}^2\setminus S_i$ for $i \in \{1,2\}$, then $g^1|_{Q_1} \in C^1(Q_1)$ and $g^2|_{Q_2} \in C^1(Q_2)$.
	
	Since $f_*^0|_{\overline{D}_0 \cap \overline{D}_1} = f_*^1|_{\overline{D}_0 \cap \overline{D}_1}$, it follows that:
	\begin{align*}
		g^0(x,y) = f_0^*(x,y) = f^1_*(x,y) = f_*^1\left(\frac{T_1}{2} + \left(x+\frac{T_1}{2} \mod T_1 \right), y \mod T_2 \right) = g^1(x,y),
	\end{align*}
	for all $(x,y) \in [T_1/2,T_1) \times [0,T_2)$. Furthermore, since $f_*^0(x,y) = f_*^1(x+T_1,y)$ for all $(x,y) \in \overline{D}_0 \setminus \overline{D}_1$, it follows that:
	\begin{align*}
		g^0(x,y) = f_0^*(x,y) = f^1_*(x+T_1,y) &= f_*^1\left(\frac{T_1}{2} + \left(x+\frac{3T_1}{2} \mod T_1 \right), y \mod T_2 \right)\\
		&= f_*^1\left(\frac{T_1}{2} + \left(x+\frac{T_1}{2} \mod T_1 \right), y \mod T_2 \right) \\
		&= g^1(x,y),
	\end{align*}
	for all $(x,y) \in \overline{D}_0 \setminus \overline{D}_1$. Thus, it follows that $g^0|_{\overline{D}_0} = g^1|_{\overline{D}_0}$. By periodicity, this is equivalent to $g^0=g^1$.
	
	A similar argument as the one in the paragraph above ensures also that $g^0=g^2$. Thus, by defining $f_*:=g^0=g^1=g^2$, it follows that $f_*|_{D_0} \in C^1(D_0)$, $f_*|_{D_1} \in C^1(D_1)$, and $f_*|_{D_2} \in C^1(D_2)$, that is, $f_*|_{D_0 \cup D_1 \cup D_2} \in C^1(D_0 \cup D_1 \cup D_2)$. By periodicity, we conclude finally that $f|_* \in C^1(\mathbb{R}^2)$.
\end{proof}

\section*{Acknowledgements}

MRC is supported by S\~{a}o Paulo Research Foundation (FAPESP) grants 2018/07344-0 and 2019/05657-4. DDN is supported by S\~{a}o Paulo Research Foundation (FAPESP) grants 2018/13481-0 and 2019/10269-3, and by Conselho Nacional de Desenvolvimento Cient\'{i}fico e Tecnol\'{o}gico (CNPq) grants 438975/2018-9 and 309110/2021-1. PCCRP is supported by S\~{a}o Paulo Research Foundation (FAPESP) grant 2020/14232-4.

\bibliographystyle{abbrv}
\bibliography{references}

\end{document}